\documentclass[11pt]{amsart}

\usepackage{subcaption}
\usepackage{pdfsync}
\usepackage{amsthm}
\usepackage{amsmath}
\usepackage{amssymb}
\usepackage{mathtools}
\usepackage{array}
\usepackage{tabularx}
\usepackage{bm}
\usepackage{verbatim}
\usepackage{geometry}

\usepackage[shortlabels]{enumitem}
\setlist{nosep}
\newcommand{\eps}{\ensuremath{\varepsilon}}

\newcommand{\Prob}[1]{\ensuremath{%
    \mathbb P\left[#1\right]
  }}
\newcommand{\Expect}[1]{\ensuremath{%
    \mathbb E\left[#1\right]
  }}
  
 \newcommand{\myArc}[2]{#1#2}

\usepackage[T1]{fontenc}
\usepackage{amssymb}
\usepackage{amsmath}
\usepackage{amsthm}
\usepackage{fancyhdr}
\usepackage{mathtools}
\usepackage[british]{babel}
\usepackage{geometry}
\usepackage{enumitem}
\usepackage{algpseudocode}
\usepackage{dsfont}
\usepackage{centernot}
\usepackage{xstring}
\usepackage{colortbl}
\usepackage[section]{algorithm}
\usepackage{graphicx}
\usepackage[font={footnotesize}]{caption}
\usepackage[usenames,dvipsnames,table]{xcolor}
\usepackage{pstricks,tikz}
\definecolor{mypink}{RGB}{0,110,0}
\definecolor{mypurple}{RGB}{130,0,130}
\definecolor{mygreen}{RGB}{0,222,0}
\definecolor{myblue}{RGB}{0,235,255}
\definecolor{dblue}{RGB}{0,170,190}
\definecolor{dpink}{RGB}{240,40,100}
\definecolor{forest}{RGB}{0,110,0}

\usetikzlibrary{calc}
\usetikzlibrary{shapes}
\usetikzlibrary{snakes}
\usetikzlibrary{arrows.meta}
\usepackage{xcolor}
\usepackage[h]{esvect}
\usepackage[noadjust]{cite}
\usepackage[
		bookmarksopen=true,
		bookmarksopenlevel=1,
		colorlinks=true,
		linkcolor=darkblue,
        linktoc=page,
		citecolor=darkblue,
]{hyperref}

\title[]{Hamilton transversals in random Latin squares}
\date{\today}

\author[Stephen~Gould]{Stephen~Gould}
\email{%\vspace{-20mm}
    \{spg377,t.j.kelly\}@bham.ac.uk}

\author[Tom~Kelly]{Tom~Kelly}

\address{School of Mathematics, University of Birmingham,
Edgbaston, Birmingham, B15 2TT, United Kingdom}

\thanks{The second author is partially supported by the EPSRC, grant no. EP/N019504/1.}

\newtheorem{theorem}{Theorem}[section]
\newtheorem{prop}[theorem]{Proposition}
\newtheorem{lemma}[theorem]{Lemma}

\newtheorem{conj}[theorem]{Conjecture}

\theoremstyle{definition}

\newtheorem{defin}[theorem]{Definition}

\newtheoremstyle{claimstyle}{5pt}{5pt}{\em}{5pt}{\em}{:}{5pt}{}
\theoremstyle{claimstyle}
\newtheorem{claim}{Claim}

\newtheoremstyle{stepstyle}{10pt}{5pt}{\em}{0pt}{\em}{:}{5pt}{}
\theoremstyle{stepstyle}

\numberwithin{equation}{section}

\definecolor{darkblue}{rgb}{0,0,0.5}

\def\noproof{{\unskip\nobreak\hfill\penalty50\hskip2em\hbox{}\nobreak\hfill%
       $\square$\parfillskip=0pt\finalhyphendemerits=0\par}\goodbreak}
\def\endproof{\noproof\bigskip}

\def\noclaimproof{{\unskip\nobreak\hfill\penalty50\hskip2em\hbox{}\nobreak\hfill%
       $-$\parfillskip=0pt\finalhyphendemerits=0\par}\goodbreak}

\def\endclaimproof{\noclaimproof\medskip}

\newdimen\margin
\def\textno#1&#2\par{
   \margin=\hsize
   \advance\margin by -4\parindent
          \setbox1=\hbox{\sl#1}
   \ifdim\wd1 < \margin
      $$\box1\eqno#2$$
   \else
      \bigbreak
      \hbox to \hsize{\indent$\vcenter{\advance\hsize by -3\parindent
      \it\noindent#1}\hfil#2$}
      \bigbreak
   \fi}

\def\lateproof#1{\removelastskip\penalty55\medskip\noindent\setcounter{claim}{0}\setcounter{step}{0}{\bf Proof of #1. }} % in each main proof, claim counter set back

\DeclareMathOperator{\comp}{comp}
\def\claimproof{\removelastskip\penalty55\medskip\noindent{\em Proof of claim: }}

\begin{document}

\newcommand{\new}[1]{\textcolour{red}{#1}}
\def\COMMENT#1{}
\def\TASK#1{}
\let\TASK=\footnote             % COMMENT OUT for clean output
\newcommand{\APPENDIX}[1]{}
\newcommand{\NOTAPPENDIX}[1]{#1}
\renewcommand{\APPENDIX}[1]{#1}                    %COMMENT OUT FOR NO APPENDIX
\renewcommand{\NOTAPPENDIX}[1]{}                   %COMMENT OUT FOR NO APPENDIX
\newcommand{\todo}[1]{\begin{center}\textbf{to do:} #1 \end{center}}
\def\tomOPT#1{{}{\sf\textcolor{magenta}{Tom: #1}}}
\def\tom#1{}
\let\tom=\tomOPT % COMMENT OUT to ignore Tom's comments
\def\stephenOPT#1{{}{\sf\textcolor{mygreen}{Stephen: #1}}}
\def\stephen#1{}
\let\stephen=\stephenOPT % COMMENT OUT to ignore Stephen's comments

\def\eps{{\varepsilon}}
\newcommand{\ex}{\mathbb{E}}
\newcommand{\pr}{\mathbb{P}}
\newcommand{\cB}{\mathcal{B}}
\newcommand{\cA}{\mathcal{A}}
\newcommand{\cE}{\mathcal{E}}
\newcommand{\cS}{\mathcal{S}}
\newcommand{\cF}{\mathcal{F}}
\newcommand{\cG}{\mathcal{G}}
\newcommand{\bL}{\mathbb{L}}
\newcommand{\bF}{\mathbb{F}}
\newcommand{\bZ}{\mathbb{Z}}
\newcommand{\cH}{\mathcal{H}}
\newcommand{\cC}{\mathcal{C}}
\newcommand{\cM}{\mathcal{M}}
\newcommand{\bN}{\mathbb{N}}
\newcommand{\bR}{\mathbb{R}}
\def\O{\mathcal{O}}
\newcommand{\cP}{\mathcal{P}}
\newcommand{\cQ}{\mathcal{Q}}
\newcommand{\cR}{\mathcal{R}}
\newcommand{\cJ}{\mathcal{J}}
\newcommand{\cL}{\mathcal{L}}
\newcommand{\cK}{\mathcal{K}}
\newcommand{\cD}{\mathcal{D}}
\newcommand{\cI}{\mathcal{I}}
\newcommand{\cV}{\mathcal{V}}
\newcommand{\cT}{\mathcal{T}}
\newcommand{\cU}{\mathcal{U}}
\newcommand{\cW}{\mathcal{W}}
\newcommand{\cX}{\mathcal{X}}
\newcommand{\cY}{\mathcal{Y}}
\newcommand{\cZ}{\mathcal{Z}}
\newcommand{\1}{{\bf 1}_{n\not\equiv \delta}}
\newcommand{\eul}{{\rm e}}
\newcommand{\Erd}{Erd\H{o}s}
\newcommand{\cupdot}{\mathbin{\mathaccent\cdot\cup}}
\newcommand{\whp}{whp }
\newcommand{\bX}{\mathcal{X}}
\newcommand{\bV}{\mathcal{V}}
\newcommand{\ordsubs}[2]{(#1)_{#2}}
\newcommand{\unordsubs}[2]{\binom{#1}{#2}}
\newcommand{\ordelement}[2]{\overrightarrow{\mathbf{#1}}\left({#2}\right)}
\newcommand{\ordered}[1]{\overrightarrow{\mathbf{#1}}}
\newcommand{\reversed}[1]{\overleftarrow{\mathbf{#1}}}
\newcommand{\weighting}[1]{\mathbf{#1}}
\newcommand{\weightel}[2]{\mathbf{#1}\left({#2}\right)}
\newcommand{\unord}[1]{\mathbf{#1}}
\newcommand{\ordscript}[2]{\ordered{{#1}}_{{#2}}}
\newcommand{\revscript}[2]{\reversed{{#1}}_{{#2}}}
\newcommand{\dirK}{\overleftrightarrow{K_{n}}}

\newcommand{\doublesquig}{%
  \mathrel{%
    \vcenter{\offinterlineskip
      \ialign{##\cr$\rightsquigarrow$\cr\noalign{\kern-1.5pt}$\rightsquigarrow$\cr}%
    }%
  }%
}

\newcommand{\defn}{\emph}

\newcommand\restrict[1]{\raisebox{-.5ex}{$|$}_{#1}}

\newcommand{\probfc}[1]{\mathrm{\mathbb{P}}_{F^{*}}\left[#1\right]}
\newcommand{\probd}[1]{\mathrm{\mathbb{P}}_{D}\left[#1\right]}
\newcommand{\probf}[1]{\mathrm{\mathbb{P}}_{F}\left[#1\right]}
\newcommand{\prob}[1]{\mathrm{\mathbb{P}}\left[#1\right]}
\newcommand{\probb}[1]{\mathrm{\mathbb{P}}_{b}\left[#1\right]}
\newcommand{\expn}[1]{\mathrm{\mathbb{E}}\left[#1\right]}
\newcommand{\expnb}[1]{\mathrm{\mathbb{E}}_{b}\left[#1\right]}
\newcommand{\probxj}[1]{\mathrm{\mathbb{P}}_{x(j)}\left[#1\right]}
\newcommand{\expnxj}[1]{\mathrm{\mathbb{E}}_{x(j)}\left[#1\right]}
\newcommand{\neigh}[3]{N^{#1}_{#2}\left(#3\right)}
\newcommand{\neighset}[3]{N^{#1}_{#2}\left(#3\right)}
\newcommand{\vhead}[1]{\textrm{head}\left(#1\right)}
\newcommand{\vtail}[1]{\textrm{tail}\left(#1\right)}
\def\gnp{G_{n,p}}
\def\G{\mathcal{G}}
\def\lflr{\left\lfloor}
\def\rflr{\right\rfloor}
\def\lcl{\left\lceil}
\def\rcl{\right\rceil}

\newcommand{\qbinom}[2]{\binom{#1}{#2}_{\!q}}
\newcommand{\binomdim}[2]{\binom{#1}{#2}_{\!\dim}}

\newcommand{\grass}{\mathrm{Gr}}

\newcommand{\brackets}[1]{\left(#1\right)}
\def\sm{\setminus}
\newcommand{\Set}[1]{\{#1\}}
\newcommand{\set}[2]{\{#1\,:\;#2\}}
\newcommand{\krq}[2]{K^{(#1)}_{#2}}
\newcommand{\ind}[1]{$\mathbf{S}(#1)$}
\newcommand{\indcov}[1]{$(\#)_{#1}$}
\def\In{\subseteq}

%\begin{abstract}  \noindent
%...
%\end{abstract}
%

\begin{abstract}
%\vspace{-4mm}
Gy\'{a}rf\'{a}s and S\'{a}rk\"{o}zy conjectured that every $n\times n$ Latin square has a `cycle-free' partial transversal of size~$n-2$.
We confirm this conjecture in a strong sense for almost all Latin squares, by showing that as $n \rightarrow \infty$, all but a vanishing proportion of $n\times n$ Latin squares have a \textit{Hamilton} transversal, i.e.\ a full transversal for which any proper subset is cycle-free.  In fact, we prove a counting result that in almost all Latin squares, the number of Hamilton transversals is essentially that of Taranenko's upper bound on the number of full transversals. This result strengthens a result of Kwan (which in turn implies that almost all Latin squares also satisfy the famous Ryser-Brualdi-Stein conjecture).% \hlc[cyan]{As part of the proof, we also prove that almost all $n\times n$ Latin squares have no symbol appearing more than $\omega(\log n / \log\log n)$ times on the leading diagonal.} which may be of independent interest.
\end{abstract}
\maketitle
\section{Introduction}
\subsection{Transversals in Latin squares}
An~\textit{$n\times n$ Latin square} is an arrangement of~$n$ symbols into~$n$ rows and~$n$ columns, such that each row and each column contains precisely one instance of each symbol.
A \textit{(full) transversal} in an~$n\times n$ Latin square is a collection of~$n$ positions of the Latin square that use each row, column, and symbol exactly once, and a \textit{partial transversal} is a collection of at most~$n$ positions not using any row, column, or symbol more than once.
%In this paper, we are primarily concerned with determining the prevalence of Latin squares that admit transversals with a strong extra structural property.
The most famous open problem on the topic of transversals in Latin squares is the following.
\begin{conj}[Ryser, Brualdi, and Stein~\cite{BR91,R67, S75}]\label{conj:rbs}
All~$n\times n$ Latin squares have a partial transversal of size~$n-1$.
%Moreover, if~$n$ is odd, then all~$n \times n$ Latin squares have a full transversal.
\end{conj}
Conjecture~\ref{conj:rbs} would be best-possible, because for even~$n$ the addition table of the integers modulo~$n$ is a Latin square which has no transversal.
If~$n$ is odd, it is actually conjectured that all~$n \times n$ Latin squares have a full transversal. 
%Further, Cavenagh and Wanless~\cite{CW17} showed that these addition tables are not the only obstacle when~$n$ is even, by demonstrating the existence of at least~$n^{(1/2-o(1))n^{3/2}}$ structurally distinct~$n\times n$ Latin squares admitting no transversal.
For nearly forty years the best result towards Conjecture~\ref{conj:rbs} was the theorem of Hatami and Shor~\cite{HS08, Shor82} (improving~\cite{W78, BdVW78}) that all~$n\times n$ Latin squares %(regardless of parity)
have a partial transversal of size $n-O(\log^{2}n)$.
%(Here and throughout, we use standard asymptotic notation, where~$n$ is the parameter implicitly understood to be tending to infinity.)
Recently however, Keevash, Pokrovskiy, Sudakov, and Yepremyan~\cite{KPSY20} improved the error term to~$O(\log n/\log \log n)$.% \by combining the celebrated R\"{o}dl Nibble technique with some innovative use of robust expansion properties.
%Improving the error term to be a constant remains elusive.
%For results on some interesting variations of Conjecture~\ref{conj:rbs}, see for example~\cite{AB09,CW05,ES91,P18}.

\begin{comment}
Note that one can identify~$n\times n$ Latin squares with $1$-factorizations of~$K_{n,n}$, where a $1$-factorization is a decomposition of the host graph into $1$-regular subgraphs, called $1$-factors (or perfect matchings).
Indeed, if the vertex classes of~$K_{n,n}$ are~$R$ and~$C$ then one can identify the rows of the Latin square with~$R$, the columns with~$C$, and each symbol with a $1$-factor, such that the edges in a $1$-factor give all the positions in which to place the associated symbol.
Imagining a $1$-factorization of $K_{n,n}$ as a proper edge-colouring using~$n$ colours, it is clear that partial transversals correspond to rainbow matchings, where a subgraph of an edge-coloured graph is called rainbow if all edges of that subgraph have different colours.
Thus, the main point of Conjecture~\ref{conj:rbs} is equivalent to the statement that all $1$-factorizations of~$K_{n,n}$ have a rainbow matching of size~$n-1$.
From this viewpoint, it is easy to see the parallels between Conjecture~\ref{conj:rbs} and the following conjecture.
\end{comment}
Conjecture~\ref{conj:rbs} is related to the following conjecture of Andersen~\cite{A89}.  An edge-coloured graph is \textit{rainbow} if all of its edges have different colours, and an edge-colouring is \textit{proper} if no two edges of the same colour share a vertex.
\begin{conj}[Andersen~\cite{A89}]\label{conj:and}
All proper edge-colourings of~$K_{n}$, the complete graph on $n$ vertices, admit a rainbow path of length~$n-2$.  
\end{conj}
In light of the result of Maamoun and Meyniel~\cite{MM84} that for infinitely many~$n$ there are proper edge-colourings of~$K_{n}$ without a rainbow Hamilton path, Conjecture~\ref{conj:and} would be best-possible.
%Intuitively, optimal edge-colourings should be the most difficult case of Conjecture~\ref{conj:and}, since here one has the fewest colours to work with, but this intuition has yet to be formalized.
Similarly to Conjecture~\ref{conj:rbs}, progress towards Conjecture~\ref{conj:and} has largely focussed on increasing the length of the longest rainbow path known to exist for any proper edge-colouring of~$K_{n}$ (see for example~\cite{CL15,GM12,GM10,GRSS11}).
Alon, Pokrovskiy, and Sudakov~\cite{APS17} were the first to asymptotically prove Conjecture~\ref{conj:and} by exhibiting the existence of a rainbow path of length $n-O(n^{3/4})$, with the best known error bound now being~$O(n^{1/2}\log n)$, provided by Balogh and Molla~\cite{BM19}.

%We remark that Conjecture~\ref{conj:and} asks for a large rainbow connected subgraph in coloured symmetrical versions of~$K_{n}$, and that in the light of the viewpoint of colourings of~$\dirK$, Conjecture~\ref{conj:rbs} asks for a large rainbow (possibly disconnected) subgraph in coloured asymmetrical versions of~$K_{n}$.
%Each of these changes in setting and goal naturally brings its own challenges.
Let~$\dirK$ be the digraph obtained from the complete $n$-vertex graph~$K_{n}$ by replacing each edge with two arcs (one in each direction) and adding a directed loop at each vertex.
For every $n\times n$ Latin square we can uniquely associate an arc-colouring of~$\dirK$ as follows: for every position $(i, j)$ of the Latin square, assign the symbol of $(i, j)$ as a colour to the arc in $\dirK$ with tail~$i$ and head~$j$.  Importantly, a partial transversal corresponds to a rainbow subgraph of $\dirK$ with maximum in-degree and out-degree one.  A set of positions is a \textit{cycle} if the corresponding subgraph of $\dirK$ is a directed cycle, and a partial transversal is \textit{cycle-free} if it contains no cycle.  Thus, cycle-free partial transversals correspond to linear directed forests in $\dirK$.
Gy\'{a}rf\'{a}s and S\'{a}rk\"{o}zy~\cite{GS14} proposed the following conjecture, which combines aspects of Conjectures~\ref{conj:rbs} and~\ref{conj:and}.
%First, note that a transversal of a Latin square defines a permutation~$\pi$ on the rows (say) of the Latin square, by putting $\pi(i)\coloneqq j$ iff position~$(i,j)$ is used in the transversal.
%Then a position of the transversal on the leading diagonal gives a singleton cycle in~$\pi$, and two positions symmetric across the leading diagonal would give a $2$-cycle of~$\pi$, etc.
%We say that cycles of~$\pi$ are cycles of the transversal, and we say that 
%A partial transversal is \textit{cycle-free} if it uses no set of positions which would form a cycle of a full transversal.
\begin{conj}[Gy{\'a}rf{\'a}s and S{\'a}rk{\"o}zy~\cite{GS14}]\label{gsconj}
  All $n\times n$ Latin squares have a cycle-free partial transversal of size $n - 2$.
\end{conj}

%To illustrate the relationship between Conjectures~\ref{conj:rbs},~\ref{conj:and}, and~\ref{gsconj}, we now introduce a third way to represent Latin squares, which will be the way we view them for most of the paper.
%We define~$\dirK$ to be the graph obtained from~$K_{n}$ by replacing each edge with two directed edges (arcs), one in each direction, and adding a (directed) loop at each vertex.
%A $1$-factor of~$\dirK$ is a collection of arcs in~$\dirK$ such that every vertex is the head of precisely one such arc, and the tail of precisely one such arc (where an arc points from `tail' to `head' by definition).
%A $1$-factorization of~$\dirK$ is a decomposition of~$\dirK$ into $1$-factors.
A \textit{proper $k$-arc-colouring} of a digraph
%~$G$ (simple or otherwise) 
is a colouring of its arcs with~$k$ colours such that no two arcs of the same colour have a common head, or a common tail.  
%for all colours~$c$ and all vertices $v\in V(G)$ we have that~$v$ is the head of at most one arc coloured~$c$, and the tail of at most one arc coloured~$c$.
%Then the colour classes of a proper $n$-arc-colouring of~$\dirK$ form a $1$-factorization}, and we frequently make this identification.
The set of~$n\times n$ Latin squares is in fact in bijection with the set of proper $n$-arc-colourings of~$\dirK$, with the correspondence described above.
%; if we say that $V(\dirK)=[n]$ then for each arc we can put a symbol corresponding to the colour of that arc in the row whose index is the tail of the arc, and the column indexed by the head.
%In this setting, transversals become rainbow $1$-factors, which need not be connected %(for example, if the set of loops is rainbow, then this corresponds to a transversal).
Thus, Conjecture~\ref{gsconj} is equivalent to the following: all proper $n$-arc-colourings of $\dirK$ contain a rainbow directed linear forest with at least $n - 2$ arcs.  No undirected analogue of this conjecture is known -- Balogh and Molla~\cite{BM19} proved that for every proper edge-colouring of $K_n$, there is a rainbow linear forest with at least $n - O(\log^2 n)$ edges, and this bound is the best known.

Less is known in the directed setting. Gy{\'a}rf{\'a}s and S{\'a}rk{\"o}zy~\cite{GS14} proved that every $n\times n$ Latin square has a cycle-free partial transversal of size $n - O(n\log\log n / \log n)$, and Benzing, Pokrovskiy, and Sudakov~\cite{BPS20} improved the error bound to $O(n^{2/3})$.  Benzing, Pokrovskiy, and Sudakov~\cite{BPS20} also proved that every proper arc-colouring of $\dirK$ contains a rainbow directed cycle of size $n - O(n^{4/5})$ and asked by how much this bound can be improved.  We believe it is also interesting to consider by how much this bound can be improved if one considers both rainbow directed cycles and paths, i.e.\ rainbow \textit{connected} subgraphs of maximum in-degree and out-degree at most one.  We conjecture the following.

\begin{conj}\label{conj:common-generalization}
    All proper arc-colourings of $\dirK$ admit a rainbow directed cycle or path of length at least $n - 1$.  %Moreover, if $n$ is odd, then all proper arc-colourings of $\dirK$ admit a rainbow directed Hamilton cycle.
\end{conj} 

We define a set of positions in a Latin square to be \textit{connected} if the corresponding subgraph of $\dirK$ is (weakly) connected, and we say a transversal is \textit{Hamilton} if it is both full and connected.  For the case of $n$-arc-colourings, Conjecture~\ref{conj:common-generalization} is equivalent to the following: all $n\times n$ Latin squares have a connected transversal of size $n - 1$.  %, and if $n$ is odd, then all $n\times n$ Latin squares have a Hamilton transversal.
If true, Conjecture~\ref{gsconj} implies that every $n\times n$ Latin square has a partial transversal of size one less than what is predicted by Conjecture~\ref{conj:rbs} and also that every proper $n$-edge-colouring of $K_n$ contains either a rainbow path of length $n - 2$, as predicted by Conjecture~\ref{conj:and}, or a spanning rainbow forest with two components.  Conjecture~\ref{conj:common-generalization}, if true, implies all of Conjectures~\ref{conj:rbs}--\ref{gsconj}.

\subsection{Random Latin squares}

In this paper, we study the above conjectures in the probabilistic setting.  
%In recent years, significant progress has been made towards Conjectures~\ref{conj:rbs} and~\ref{conj:and} by considering random Latin squares and optimally edge-coloured complete graphs respectively, and searching for the relevant structure.
%Adapting his work on random Steiner triple systems to this setting, 
Recently, Kwan~\cite{K20} proved that at most a vanishing proportion of Latin squares fail to satisfy the statement of Conjecture~\ref{conj:rbs}, even finding (many) full transversals in most Latin squares, as follows.
%Here and throughout the paper, `almost all' (resp.\ `with high probability') means a proportion (resp.\ with a probability) that tends to~$1$ as~$n$ tends to infinity.
\begin{theorem}[Kwan~\cite{K20}]\label{kwanthm}
Almost all~$n\times n$ Latin squares have at least
\[
\left((1-o(1))\frac{n}{e^{2}}\right)^{n}
\]
transversals.
\end{theorem}
Equivalently, a uniformly random $n\times n$ Latin square has at least $\left((1-o(1)){n}/{e^{2}}\right)^{n}$ transversals with high probability.
We note that it was proven by Taranenko~\cite{T15} (with a simpler proof later found by Glebov and Luria~\cite{GL16}) that~$n\times n$ Latin squares can have at most $\left((1+o(1))n/e^{2}\right)^{n}$ transversals, so that the counting term given in Theorem~\ref{kwanthm} is best possible, up to the exponential error term.
Analogously, the authors, together with K\"{u}hn and Osthus~\cite{GKKO20}, proved that almost all optimal edge-colourings (proper edge-colourings using the minimum possible number of colours) of~$K_{n}$ admit a rainbow Hamilton path, which proves a stronger statement than Conjecture~\ref{conj:and} for all but a vanishing proportion of such colourings.

%To state our main theorem, we need two preliminary definitions.
%The main result of this paper is the following strengthening of Conjecture~\ref{gsconj} for all but a vanishing proportion of $1$-factorizations of~$\dirK$.
%A directed cycle is a cycle with consistently directed arcs.
%A rainbow directed Hamilton cycle in a $1$-factorization of~$\dirK$ corresponds to a transversal which is itself a cycle (so that any strict partial transversal is cycle-free) in the associated Latin square.
%We say that such transversals are \textit{Hamilton transversals}.
%The authors of~\cite{GKKO20} initiated the study of Hamilton transversals\COMMENT{Is this true?} by showing that if~$n$ is odd, then almost all~$n\times n$ Latin squares whose entries are symmetrical across the leading diagonal admit a Hamilton transversal.
The main result of this paper is the following strengthening of Theorem~\ref{kwanthm}.
\begin{theorem}\label{main-thm}
Almost all proper $n$-arc-colourings of $\dirK$ %chosen uniformly at random, then with high probability~$\phi$ 
contain at least
\[
\left((1-o(1))\frac{n}{e^{2}}\right)^{n}
\]
rainbow directed Hamilton cycles.
Equivalently, almost all~$n\times n$ Latin squares have at least $\left((1-o(1))n/e^{2}\right)^{n}$ Hamilton transversals.
\end{theorem}
Theorem~\ref{main-thm} implies that a uniformly random proper $n$-arc-colouring satisfies Conjecture~\ref{conj:common-generalization} with high probability, which in turn implies that a uniformly random $n\times n$ Latin square satisfies Conjectures~\ref{conj:rbs} and~\ref{gsconj} with high probability as well.  We note that the number of optimal edge-colourings of $K_n$ is a vanishing fraction of the number of $n\times n$ Latin squares, so Theorem~\ref{main-thm} does not imply the result of~\cite{GKKO20}.

Random Latin squares can be difficult to analyze, in part due to their `rigidity' and lack of independence.  To prove Theorem~\ref{kwanthm}, Kwan~\cite{K20} -- using Keevash's~\cite{K18, K18counting} breakthrough results on the existence of combinatorial designs -- developed a method for approximating a uniformly random Latin square by an outcome of the `triangle-removal process', which is in comparison much easier to analyze.  Prior to Kwan's~\cite{K20} work, a limited number of results (e.g.~\cite{CGW08, KS18, MW99, W04}) were proved using so-called `switching' methods.  Our proof, notably, does not rely on Keevash's~\cite{K18, K18counting} results and instead introduces new techniques for analyzing `switchings' to study Latin squares, thus providing a more elementary proof of Theorem~\ref{kwanthm}.

\subsection{Organization of the paper}
In Section~\ref{section-notation} we clarify some notation and definitions that we will use throughout the paper.
We overview the proof of Theorem~\ref{main-thm} in Section~\ref{overview-section}, and give some preliminary probabilistic results and useful theorems of other authors in Section~\ref{section:preliminaries}.
%\hlc[cyan]{(((((}We devote Section~\ref{diagsection} to the proof of Theorem~\ref{main-thm-2},\hlc[cyan]{)))))} and 
Sections~\ref{switchsection}--\ref{sec:proof} are devoted to the proof of Theorem~\ref{main-thm}. %and~\ref{section:absorption}.
%In Section~\ref{switchsection} we analyze some important properties of uniformly random proper $n$-arc colourings of~$\dirK$, and in Section~\ref{section:absorption} we fix a typical such colouring and describe how to use the aforementioned properties to build many distinct rainbow directed Hamilton cycles. 
\section{Notation}\label{section-notation}
For a natural number~$n$ we define $[n]\coloneqq\{1,2,\dots,n\}$ and $[n]_{0}\coloneqq[n]\cup\{0\}$.
We say that a partition~$\cP=\{D_{i}\}_{i=1}^{m}$ of a finite set~$D$ into~$m$ parts is \textit{equitable} if $|D_{i}|\in\{\lfloor|D|/m\rfloor,\lceil|D|/m\rceil\}$ for all $i\in[m]$, and when~$|D|$ is large we assume that each part~$D_{i}$ has the same size~$|D|/m$, where this does not affect the argument.

For a digraph~$G$, we write the arc set of~$G$ as~$E(G)$, and we denote an arc from a vertex~$u$ to a vertex~$v$ as~$uv$, and we say that~$u$ is the \textit{tail} of the arc $e=uv$, denoted $u=\vtail{e}$, and that~$v$ is the \textit{head} of~$e$, denoted $v=\vhead{e}$.
We say that any vertex~$v$ such that $uv\in E(G)$ is an \textit{out-neighbour} of~$u$ in~$G$, and that any~$v$ such that $vu\in E(G)$ is an \textit{in-neighbour} of~$u$ in~$G$.
We define~$N^{+}_{G}(v)$ to be the set of out-neighbours of~$v$ in~$G$, sometimes dropping the subscript~$G$ when~$G$ is clear from context, and we call~$N^{+}_{G}(v)$ the \textit{out-neighbourhood} of~$v$ in~$G$.
We define the \textit{in-neighbourhood} of~$v$ in~$G$, denoted~$N^{-}_{G}(v)$, analogously, and we define the \textit{neighbourhood} of~$v$ in~$G$ to be $N_{G}(v)\coloneqq N^{+}_{G}(v)\cup N^{-}_{G}(v)$.
We define $d_{G}^{+}(v)\coloneqq|N^{+}_{G}(v)|$ and $d_{G}^{-}(v)\coloneqq |N^{-}_{G}(v)|$.
For (not necessarily distinct) vertex sets $A,B\subseteq V(G)$ we define $E_{G}(A,B)\coloneqq\{ab\in E(G)\colon a\in A,b\in B\}$, and $e_{G}(A,B)\coloneqq|E_{G}(A,B)|$.
Suppose now that~$G$ is equipped with an arc-colouring~$\phi_{G}$ in colour set~$C$.
Then for a colour $c\in C$ and an arc $e\in E(G)$ we write $\phi_{G}(e)=c$ to mean that~$e$ has colour~$c$ in the colouring~$\phi_{G}$ of~$G$.
We frequently drop the notation~$G$ when~$G$ is clear from context.
Further, if~$\phi(e)=c$ then we say that~$e$ is a $c$-arc, and in the case that~$e$ is a loop we say that~$e$ is a $c$-loop.
We write~$E_{c}(G)$ for the set of $c$-arcs in~$G$ (including $c$-loops), and we refer to~$E_{c}(G)$ as the \textit{colour class} of~$c$.
Fix $u\in V(G)$.
If $d\in C$ is such that there is a $d$-arc~$uv$ in~$G$, then the (unique) vertex~$v$ is called the $d$-\textit{out-neighbour} of~$u$, which we denote by~$\neigh{+}{d}{u}$.
We define the $d$-\textit{in-neighbour}~$\neigh{-}{d}{u}$ of~$u$ analogously.
For $D\subseteq C$ we define $\neighset{+}{D}{u}\coloneqq\{\neigh{+}{d}{u}\colon d\in D\}$ and $\neighset{-}{D}{u}\coloneqq\{\neigh{-}{d}{u}\colon d\in D\}$, and for $A,B\subseteq V$ we define $E_{G,D}(A,B)\coloneqq\{e\in E_{G}(A,B)\colon \phi(e)\in D\}$ and $e_{G,D}(A,B)\coloneqq|E_{G,D}(A,B)|$.
For a subdigraph $H\subseteq G$ we define $\phi_{G}(H)\coloneqq\{\phi_{G}(e)\colon e\in E(H)\}$.
%\textcolor{blue}{A brief notation section here could neaten the rest of the paper. Putting it before the overview of the proof may help us to present a more concise overview too?}
%Let~$\cL_{n}$ be the set of~$n\times n$ Latin squares with symbols~$[n]$.
%let~$C$ be the absolute constant given by Theorem~\ref{main-thm-2}, and let~$\cL_{n}^{*}$ be the set of Latin squares $L\in\cL_{n}$ whose every symbol appears at most~$C\log n/\log \log n$ times on the leading diagonal, so that $|\cL_{n}^{*}|=(1-o(1))|\cL_{n}|$.

With a slight abuse of notation, we often refer to a pair $(H, \phi)$ where $H$ is a digraph and $\phi$ is a proper arc-colouring of $H$ as a `coloured digraph' $H$ implicitly equipped with a proper arc-colouring $\phi_H$ (or simply $\phi$ if it is clear from the context).  
Using this convention, we let $\Phi(\dirK)$ denote the set of all properly $n$-arc-coloured digraphs $G\cong \dirK$ with vertex set and colour set $[n]$.  (That is, the set of pairs $(G, \phi)$ where $G\cong\dirK$ and $\phi$ is a proper $n$-arc-colouring of $G$).
For a coloured digraph $G\in\Phi(\dirK)$ and a set of colours $D\subseteq [n]$ we define~$G|_{D}$ to be the coloured digraph obtained by deleting all arcs of~$G$ having colours not in~$D$, and we set $\cG^{n}_{D}\coloneqq\{G|_{D}\colon G\in\Phi(\dirK)\}$, though we always drop the~$n$ in the superscript as~$n$ will be clear from context.
By symmetry of the roles of rows, columns, and symbols in Latin squares, the correspondence between Latin squares and elements $G\in\Phi(\dirK)$, and the well-known result that any Latin rectangle has a completion to a Latin square, it is clear that~$\cG_{D}$ could be equivalently defined as the set of all pairs~$(H,\phi_{H})$, where~$H$ is a $|D|$-regular digraph on vertices~$[n]$, and~$\phi_{H}$ is a proper arc-colouring of~$H$ in colours~$D$.\COMMENT{Such considerations are important so that it is clear that when we switch from one $H\in\cG_{D}$ using one of our `switching cycles', the result is also an element of~$\cG_{D}$.}
%We will use the notation $H\in\cG_{D}$, where~$H$ is to be viewed as an arc-coloured digraph, and we will sometimes use the notation~$\phi_{H}$ to denote the implicit $1$-factorization of~$H$, occasionally simply writing~$\phi$ if~$H$ is clear from the context.
Throughout the paper we will use the letter~$G$ for an element of~$\Phi(\dirK)$, and the letter~$H$ for an element of~$\cG_{D}$ (for any~$D$).
%For a coloured digraph $H\in\cG_{D}$, we define~$\text{comp}(H)$ to be the number of distinct ways to complete~$H$ to an element $G\in\Phi(\dirK)$, or more precisely the number of $H'\in\cG_{[n]\setminus D}$ having $E(H)\cap E(H')=\emptyset$ (and therefore $E(H)\cup E(H')=E(\dirK)$).
We often write random variables and objects in bold notation.
For an event~$\cE$ in any probability space we use the notation~$\overline{\cE}$ to denote the complement of~$\cE$.
\section{Overview of the proof}\label{overview-section}
The proof of Theorem~\ref{main-thm} proceeds in two key steps.
We first analyze uniformly random $\mathbf{G}\in\Phi(\dirK)$ and show that with high probability,~$\mathbf{G}$ satisfies three key properties.
It then suffices to suppose that a fixed $G\in\Phi(\dirK)$ satisfies these three properties, and use that hypothesis to build many rainbow directed Hamilton cycles in~$G$.  Before describing these properties, we discuss our strategy for building rainbow directed Hamilton cycles.  To that end, we introduce the following definition.
\begin{defin}\label{def:rrh}
  A subgraph $H\subseteq G\in\Phi(\dirK)$ is \textit{robustly rainbow-Hamiltonian} (with respect to \textit{flexible sets} $V_{\mathrm{flex}}\subseteq V(H)$ and $C_{\mathrm{flex}}\subseteq\phi(H)$ of vertices and colours, and \textit{initial} vertex~$u\in V(H)$ and \textit{terminal} vertex $v\in V(H)$), if for any pair of equal-sized subsets $X\subseteq V_{\mathrm{flex}}$ and $Y\subseteq C_{\mathrm{flex}}$ of size at most~$\min\{|V_{\mathrm{flex}}|/2, |C_{\mathrm{flex}}|/2\}$, the graph $H - X$ contains a rainbow directed Hamilton path with tail~$u$ and head~$v$, not containing a colour in $Y$.
\end{defin}
%\begin{defin}
%  robustly rainbow hamiltonian with respect to flexible sets -- fixed start and end
%
%  An edge-coloured graph~$G$ is \textit{robustly rainbow-Hamiltonian} (with respect to ``flexible'' sets $V_{\mathrm{flex}}$ and $C_{\mathrm{flex}}$ of vtcs and colors) if
%  \begin{itemize}[label=($\star$)]
%  \item for any pair of equal-sized subsets $X\subseteq V_{\mathrm{flex}}$ and $Y\subseteq C_{\mathrm{flex}}$ of size at most $|V_{\mathrm{flex}}|/2$ and $|C_{\mathrm{flex}}|/2$, the graph $G - X$ contains a rainbow Hamilton path not containing a color in $Y$.
%  \end{itemize}
%\end{defin}
We show that for almost all $G\in\Phi(\dirK)$ and arbitrary sets~$V_{\mathrm{flex}}$,~$C_{\mathrm{flex}}\subseteq[n]$ of sizes $|V_{\mathrm{flex}}|=|C_{\mathrm{flex}}|=\Omega(n/\log^{3}n)$,~$G$ contains a robustly rainbow-Hamiltonian subgraph~$H$ with flexible sets~$V_{\mathrm{flex}}$ and~$C_{\mathrm{flex}}$, such that~$H$ has~$O(n/\log^{3}n)$ vertices and arcs in total. We construct rainbow directed Hamilton cycles by using the popular `absorption' method, and~$H$ will form the key absorbing structure.  More precisely, we find a rainbow directed path $P$ having the terminal vertex $v$ of $H$ as its tail, the initial vertex $u$ of $H$ as its head, such that $V(G)\setminus V(H) \subseteq V(P)$, $V(P) \cap V(H)$ is a subset of $V_{\mathrm{flex}}$ of size at most $|V_{\mathrm{flex}}/2|$, and likewise for the colours.  Letting $X \coloneqq V(P)\cap V(H)$ and $Y\coloneqq \phi(P)\cap \phi(H)$, the robust rainbow-Hamiltonicity of $H$ guarantees there is a rainbow directed Hamilton path $P'$ in $H - X$ with tail $u$ and head $v$, not containing a colour in $Y$, and $P \cup P'$ is a rainbow directed Hamilton cycle.

We find~$H$ by piecing together smaller building blocks we call `absorbers' in a delicate way, where each absorber has the ability to `absorb' a vertex~$v$ and a colour~$c$ not used by $P$. We delay a definition of a $(v,c)$-absorber to Definition~\ref{defn:(v,c)-absorber}, but we give a figure now (see Figure~\ref{fig:vcabsorber}).
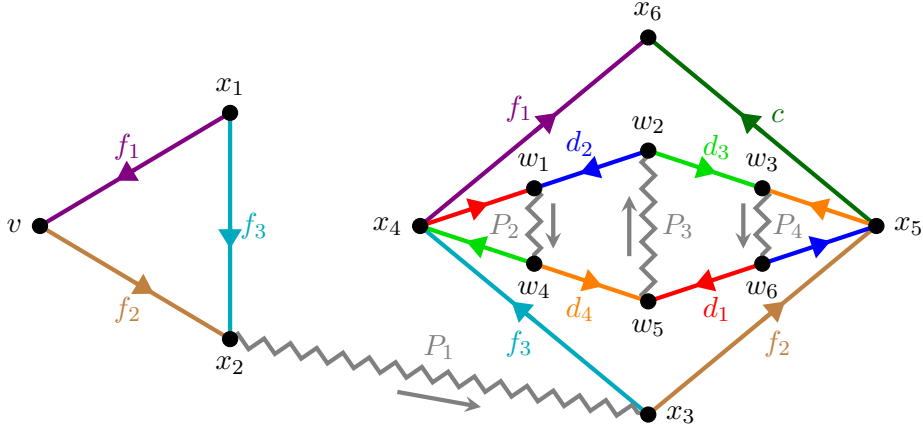
\begin{figure}
    \centering
    \begin{tikzpicture} [scale=1]
    \draw [->,ultra thick,mypurple,>=Triangle] (1.5,1.5)--(0,0.6);
    \draw [ultra thick,mypurple] (0.25,0.75)--(-1,0);
    \draw [->,ultra thick,brown,>=Triangle] (-1,0)--(0.5,-0.9);
    \draw [ultra thick,brown] (0.25,-0.75)--(1.5,-1.5);
    \draw [->,ultra thick,dblue,>=Triangle] (1.5,1.5)--(1.5,-0.3);
    \draw [ultra thick,dblue] (1.5,0)--(1.5,-1.5);
    \draw [->,ultra thick,mypurple,>=Triangle] (4,0)--(5.8,1.5);
    \draw [ultra thick,mypurple] (5.5,1.25)--(7,2.5);
    \draw [->,ultra thick,dblue,>=Triangle] (7,-2.5)--(5.2,-1);
    \draw [ultra thick,dblue] (5.5,-1.25)--(4,0);
    \draw [->,ultra thick,forest,>=Triangle] (10,0)--(8.2,1.5);
    \draw [ultra thick,forest] (8.5,1.25)--(7,2.5);
    \draw [->,ultra thick,brown,>=Triangle] (7,-2.5)--(8.8,-1);
    \draw [ultra thick,brown] (8.5,-1.25)--(10,0);
    \draw [->,ultra thick,mygreen,>=Triangle] (7,1)--(7.9,0.7);
    \draw [ultra thick,mygreen] (7.75,0.75)--(8.5,0.5);
    \draw [->,ultra thick,orange,>=Triangle] (10,0)--(9.1,0.3);
    \draw [ultra thick,orange] (9.25,0.25)--(8.5,0.5);
    \draw [->,ultra thick,red,>=Triangle] (4,0)--(4.9,0.3);
    \draw [ultra thick,red] (4.75,0.25)--(5.5,0.5);
    \draw [->,ultra thick,blue,>=Triangle] (7,1)--(6.1,0.7);
    \draw [ultra thick,blue] (6.25,0.75)--(5.5,0.5);
    \draw [->,ultra thick,red,>=Triangle] (8.5,-0.5)--(7.6,-0.8);
    \draw [ultra thick,red] (7.75,-0.75)--(7,-1);
    \draw [->,ultra thick,orange,>=Triangle] (5.5,-0.5)--(6.4,-0.8);
    \draw [ultra thick,orange] (6.25,-0.75)--(7,-1);
    \draw [->,ultra thick,blue,>=Triangle] (8.5,-0.5)--(9.4,-0.2);
    \draw [ultra thick,blue] (9.25,-0.25)--(10,0);
    \draw [->,ultra thick,mygreen,>=Triangle] (5.5,-0.5)--(4.6,-0.2);
    \draw [ultra thick,mygreen] (4.75,-0.25)--(4,0);
    \draw [snake=zigzag,ultra thick, gray] (1.5,-1.5)--(7,-2.5);
    \draw [snake=zigzag,ultra thick, gray] (5.5,0.5)--(5.5,-0.5);
    \draw [snake=zigzag,ultra thick, gray] (7,-1)--(7,1);
    \draw [snake=zigzag,ultra thick, gray] (8.5,0.5)--(8.5,-0.5);
    \draw [fill] (-1,0) circle [radius=0.1];
    \draw [fill] (1.5,1.5) circle [radius=0.1];
    \draw [fill] (1.5,-1.5) circle [radius=0.1];
    \draw [fill] (4,0) circle [radius=0.1];
    \draw [fill] (7,2.5) circle [radius=0.1];
    \draw [fill] (7,-2.5) circle [radius=0.1];
    \draw [fill] (10,0) circle [radius=0.1];
    \draw [fill] (7,1) circle [radius=0.1];
    \draw [fill] (7,-1) circle [radius=0.1];
    \draw [fill] (8.5,0.5) circle [radius=0.1];
    \draw [fill] (8.5,-0.5) circle [radius=0.1];
    \draw [fill] (5.5,0.5) circle [radius=0.1];
    \draw [fill] (5.5,-0.5) circle [radius=0.1];
    \node [left] at (-1.1,0) {$v$};
    \node [above] at (1.5,1.6) {$x_{1}$};
    \node [below] at (1.5,-1.6) {$x_{2}$};
    \node [left] at (3.9,0) {$x_{4}$};
    \node [below] at (7,-1.1) {$w_{5}$};
    \node [below] at (5.5,-0.6) {$w_{4}$};
    \node [below] at (8.5,-0.6) {$w_{6}$};
    \node [above] at (7,2.6) {$x_{6}$};
    \node [right] at (7.1,-2.5) {$x_{3}$};
    \node [above] at (5.5,0.6) {$w_{1}$};
    \node [above] at (8.5,0.6) {$w_{3}$};
    \node [above] at (7,1.1) {$w_{2}$};
    \node [right] at (10.1,0) {$x_{5}$};
    \node [above,mypurple] at (0.15,0.75) {$f_{1}$};
    \node [below,brown] at (0.15,-0.8) {$f_{2}$};
    \node [right,dblue] at (1.5,0) {$f_{3}$};
    \node [below,brown] at (8.7,-1.25) {$f_{2}$};
    \node [below,orange] at (6.1,-0.8) {$d_{4}$};
    \node [above,blue] at (6.1,0.8) {$d_{2}$};
    \node [below,red] at (7.9,-0.8) {$d_{1}$};
    \node [above,mygreen] at (7.9,0.75) {$d_{3}$};
    \node [above,mypurple] at (5.3,1.25) {$f_{1}$};
    \node [below,dblue] at (5.3,-1.25) {$f_{3}$};
    \node [above,forest] at (8.7,1.25) {$c$};
    \node [above,gray] at (4.25,-1.9) {$P_{1}$};
    \draw [->,>=stealth,gray,ultra thick] (3.7,-2.2)--(4.8,-2.4);
    \draw [->,>=stealth,gray,ultra thick] (5.75,0.3)--(5.75,-0.3);
    \node [left,gray] at (5.45,0) {$P_{2}$};
    \node [right,gray] at (7.05,0) {$P_{3}$};
    \draw [->,>=stealth,gray,ultra thick] (6.75,-0.4)--(6.75,0.4);
    \draw [->,>=stealth,gray,ultra thick] (8.25,0.3)--(8.25,-0.3);
    \node [right,gray] at (8.5,0) {$P_{4}$};
    \end{tikzpicture}
    \caption{A $(v,c)$-absorber. Here $\phi(x_{4}w_{1})=d_{1}$, $\phi(w_{6}x_{5})=d_{2}$, $\phi(w_{4}x_{4})=d_{3}$, and $\phi(x_{5}w_{3})=d_{4}$. $P_{1},\dots,P_{4}$ are rainbow directed paths with directions as indicated, sharing no colours with each other or with the rest of the $(v,c)$-absorber.}
    \label{fig:vcabsorber}
\end{figure}
Notice that a $(v,c)$-absorber has a rainbow directed Hamilton path with tail~$x_{1}$ and head~$x_{6}$, and a rainbow directed path with the same head and tail using all vertices except~$v$ and all colours except~$c$. This is the key property of a $(v,c)$-absorber, and by piecing these together in a precise way we ensure that the resulting union of absorbers (with the sets of specified vertices~`$v$' and colours~`$c$' forming~$V_{\text{flex}}$ and~$C_{\text{flex}}$ respectively) has the desired robustly rainbow-Hamiltonian property.
For technical reasons, we find $(v,c)$-absorbers %in random $\mathbf{G}\in\Phi(\dirK)$
by piecing together two smaller structures we call $(v,c)$\textit{-absorbing gadgets} and $(y,z)$\textit{-bridging gadgets} (see Definitions~\ref{def:absgadg} and~\ref{def:bridgegadg}, respectively), together with the short rainbow directed paths~$P_{1}$,~$P_{2}$,~$P_{3}$,~$P_{4}$ as in Figure~\ref{fig:vcabsorber}.

Thus, the first key property that we need almost all $G\in\Phi(\dirK)$ to satisfy, is that $G$ contains many absorbing gadgets and bridging gadgets, in a `well-spread' way that enables us to construct an appropriate robustly rainbow-Hamiltonian subgraph.  We prove this in Section~\ref{switchsection} using `switchings' in Latin rectangles, then using permanent estimates (see~\cite{B73, E81, F81}, encapsulated by Proposition~\ref{prop:wf} in the current paper) to compare a uniformly random $k \times n$ Latin rectangle to the first $k$ rows of a uniformly random  $n \times n$ Latin square. Lemma~\ref{masterswitch1} ensures the existence of the absorbing gadgets we need, and Lemma~\ref{main-switching-lemma} accomplishes the same for the bridging gadgets. This approach of using permanent estimates to translate statements between these probability spaces was pioneered by McKay and Wanless~\cite{MW99}, who investigated the typical prevalence of 2 × 2 Latin subsquares (also called ‘intercalates’) in a uniformly random Latin square.
For further insight into the usage of this method to study intercalates in random Latin squares, see for example~\cite{KS18, KSS21, KSSS22}.
As the substructures we seek are more complex than intercalates, and we moreover require that they are `well-spread', our proof introduces new techniques for switching arguments in Latin rectangles.  We note that in~\cite{GKKO20}, the authors, with K\"uhn and Osthus, used switching arguments to analyze a uniformly random 1-factorization of $K_n$ and show that with high probability there is a large collection of subgraphs of a form analogous to that of our $(v, c)$-absorbing gadgets in the undirected setting.  \APPENDIX{Fortunately, this argument also works in the directed setting with only minor changes, so we defer the proof of Lemma~\ref{masterswitch1} to the appendix.}\NOTAPPENDIX{Fortunately, this argument also works in the directed setting with only minor changes that we do not cover here, instead providing the proof of Lemma~\ref{masterswitch1} in the appendix of the arXiv version of the paper.}  Thus, Section~\ref{switchsection} is primarily devoted to the proof of Lemma~\ref{masterswitch2}.

The second property of almost all $G\in\Phi(\dirK)$ that we will need concerns the colours of the loops.
Clearly, if we seek to find any rainbow directed Hamilton cycle of $G\in\Phi(\dirK)$, we need to know that there is no colour appearing only on loops in~$G$, and this is given for almost all $G\in\Phi(\dirK)$ (in the context of Latin squares and in considerably stronger form) by Lemma~\ref{fewloops}.  

The third and final property of almost all $G\in\Phi(\dirK)$ that we will need is an appropriate notion of `lower-quasirandomness', which roughly states that for any two subsets $U_1, U_2$ of vertices of $G$ and any set $D$ of colours, the number of arcs in $G$ with tail in $U_1$, head in $U_2$, and colour in $D$, is close to what we would expect if the colours of the arcs of $G$ were assigned independently and uniformly at random.
We delay the precise definition of lower-quasirandomness of $G\in\Phi(\dirK)$ to Definition~\ref{def:lowerquas}.
The desired property that almost all $G\in\Phi(\dirK)$ are lower-quasirandom will follow immediately from~\cite[Theorem 2]{KS18} (see Theorem~\ref{thm:discrepancy} of the current paper), originally stated in the context of `discrepancy' of random Latin squares.

Armed with the three properties of typical $G\in\Phi(\dirK)$ described above, it then suffices to fix such a~$G$ and build many rainbow directed Hamilton cycles.
In Section~\ref{section:absorption}, we show that the existence of many well-spread absorbing and bridging gadgets enables us to greedily build a small robustly rainbow-Hamiltonian subgraph~$H\subseteq G$ with arbitrary flexible sets~$V_{\mathrm{flex}}$ and~$C_{\mathrm{flex}}$ of size $\Theta(n/\log^{3}n)$, and in Section~\ref{sec:proof}, we use this to prove Theorem~\ref{main-thm}.  
The rough idea is to first choose the flexible sets $V_{\mathrm{flex}}$ and $C_{\mathrm{flex}}$ randomly.
Next, we use the lower-quasirandomness property of~$G$ to build a rainbow directed spanning path forest~$Q$ of $G-H$, one arc at a time, until~$Q$ has very few components.
Then, we use the random choice of~$V_{\mathrm{flex}}$ and~$C_{\mathrm{flex}}$, together with Lemma~\ref{fewloops}, to find short rainbow directed paths linking the components of~$Q$ and the designated start and end of~$H$, which use all remaining colours of~$G-H$, and at most half of~$V_{\mathrm{flex}}$ and~$C_{\mathrm{flex}}$.
Finally, we use the key robustly rainbow-Hamiltonian property of~$H$ to absorb the remaining vertices and colours in $V_{\mathrm{flex}}$ and $C_{\mathrm{flex}}$ as described above, completing the rainbow directed Hamilton cycle of~$G$.
To obtain the counting result on the number of rainbow directed Hamilton cycles in~$G$, it suffices to count the number of choices we can make whilst building the rainbow directed spanning path forest~$Q$ of~$G-H$.

We remark that this particular absorption strategy, wherein we create an absorbing structure with `flexible' sets, is an instance of the `distributive absorption' method, which was introduced by Montgomery~\cite{M18} in 2018 and has been found to have several applications since.  In particular, this method is also used in~\cite{K20} and~\cite{GKKO20} to find transversals in random Latin squares and rainbow Hamilton paths in random 1-factorizations, respectively. Our approach differs from that of~\cite{K20} and~\cite{GKKO20} in a few key ways.  First, the `asymmetry' of proper $n$-arc-colourings of $\dirK$ (in comparison to proper edge-colourings of $K_n$ with at most~$n$ colours, which correspond to proper $n$-arc-colourings of~$\dirK$ with monochromatic digons) and `connectedness' of rainbow Hamilton cycles/ Hamilton transversals (in comparison to general transversals in Latin squares) necessitate a more complex absorbing structure than the one of either~\cite{K20} or~\cite{GKKO20}, which is more challenging to create and construct.  Nevertheless, as mentioned, we show that switching arguments are sufficient for finding our absorbing structure, yielding a more elementary proof than that of~\cite{K20}, and moreover, by choosing our flexible sets randomly, we avoid complications involving vertices with few out- or in-neighbours in $V_{\mathrm{flex}}$ on arcs with colour in $C_{\mathrm{flex}}$, providing a further simplification of the approach in~\cite{K20}. In~\cite{GKKO20}, results~\cite{PippengerSpencer, AY05} on nearly perfect matchings in nearly regular hypergraphs are applied to auxiliary hypergraphs to construct both the absorbing structure and a nearly spanning rainbow path in a random 1-factorization of $K_n$, but since the absorbers we use here (minus the internal vertices of the linking paths $P_1, \dots, P_4$) are not regular, the analogous approach fails in the directed setting (as the corresponding auxiliary hypergraphs are not regular).  However, as we show, the `lower-quasirandomness' of typical $G\in\Phi(\dirK)$ is enough for us to find $Q$, the nearly spanning rainbow path forest, without these hypergraph matching results, and our absorbing structure is robust enough to augment it to a rainbow directed path.

\section{Preliminaries}\label{section:preliminaries}
In this brief section we state some results that we will use in the proof of Theorem~\ref{main-thm}.
We begin with a well-known concentration inequality for independent random variables.

Let $X_1, \dots, X_m$ be independent random variables taking values in $\mathcal X$, and let $f : \mathcal X^m \rightarrow \mathbb R$.
If for all $i\in [m]$ and $x'_i, x_1, \dots, x_m \in \mathcal X$, we have
\begin{equation*}
  |f(x_1, \dots, x_{i - 1}, x_i, x_{i + 1}, \dots, x_m) - f(x_1, \dots, x_{i - 1}, x'_i, x_{i + 1}, \dots, x_m)| \leq c_i,
\end{equation*}
then we say $X_i$ \textit{affects} $f$ by at most $c_i$.
    
\begin{theorem}[McDiarmid's Inequality~\cite{M89}]\label{mcd}
  If $X_1, \dots, X_m$ are independent random variables taking values in $\mathcal X$ and $f : \mathcal X^m \rightarrow \mathbb R$ is such that $X_i$ affects $f$ by at most $c_i$ for all $i\in [m]$, then for all $t > 0$,
  \begin{equation*}
    \Prob{|f(X_1, \dots, X_m) - \Expect{f(X_1, \dots, X_m)}| \geq t} \leq \exp\left(-\frac{2t^2}{\sum_{i=1}^m c^2_i}\right).
  \end{equation*}
\end{theorem}
Next, we need the notion of `robustly matchable' bipartite graphs, which will form a key part of our absorption argument.
\begin{defin}\label{def:rmbg}
  Let $T$ be a bipartite graph with bipartition $(A, B)$ such that $|A| = |B|$.
  \begin{itemize}
  \item We say $T$ is \textit{robustly matchable} with respect to \textit{flexible sets} $A' \subseteq A$ and $B'\subseteq B$, if for every pair of equal-sized subsets $X \subseteq A'$ and $Y\subseteq B'$ of size at most $\min\{|A'| / 2, |B'| / 2\}$, there is a perfect matching in $T - (X\cup Y)$.
  \item For $m \in \mathbb N$, we say $T$ is a $2RMBG(7m, 2m)$ if $|A| = |B| = 7m$ and $T$ is robustly matchable with respect to flexible sets $A'\subseteq A$ and $B'\subseteq B$ where $|A'| = |B'| = 2m$.
  \end{itemize}
\end{defin}
The concept of using robustly matchable bipartite graphs in absorption arguments was first introduced by Montgomery~\cite{M18}.
We need the following observation of the authors, K\"uhn, and Osthus~\cite[Lemma~4.5]{GKKO20}, which is based on the work of Montgomery.  

\begin{lemma}[Gould, Kelly, K\"uhn, and Osthus~\cite{GKKO20}]\label{lemma:2rmbg}
  For all sufficiently large $m$, there is a $2RMBG(7m, 2m)$ that is $256$-regular.
\end{lemma}
For a coloured digraph $H\in\cG_{D}$, we define~$\text{comp}(H)$ to be the number of distinct ways to complete~$H$ to an element $G\in\Phi(\dirK)$, or more precisely the number of $H'\in\cG_{[n]\setminus D}$ having $E(H)\cap E(H')=\emptyset$ (and therefore $E(H)\cup E(H')=E(\dirK)$).
We will use the following proposition to compare the probabilities of events in the probability spaces corresponding to uniformly random $\mathbf{H}\in\cG_{D}$ (for some small $D\subseteq[n]$) and uniformly random $\mathbf{G}\in\Phi(\dirK)$ (see for example the proof of Lemma~\ref{main-switching-lemma}).
\begin{prop}\label{prop:wf}
For any $D\subseteq[n]$ and $H,H'\in\cG_{D}$ we have
\[
\frac{\comp(H)}{\comp(H')}\leq\exp(O(n\log^{2}n)).
\]
\end{prop}
Proposition~\ref{prop:wf} follows immediately from (for example)~\cite[Proposition 5]{KS18} as~$\cG_{D}$ can easily be seen to be equivalent to the set of $|D|\times n$ Latin rectangles.

%\begin{comment}
Next, we show (in the context of Latin squares) that a uniformly random $\mathbf{G}\in\Phi(\dirK)$ does not have too many loops of a fixed colour. We first need the following well-known result on the number of fixed points of a random permutation.
\begin{lemma}\label{montperm}
Let~$\bm{\sigma}$ be a uniformly random permutation of~$[n]$, and let~$\mathbf{X}$ denote the number of fixed points of~$\bm{\sigma}$.
Then, for $k\in[n]_{0}$, we have $\prob{\mathbf{X}=k}=\frac{1}{k!}\sum_{j=0}^{n-k}\frac{(-1)^{j}}{j!}$.
\end{lemma}
%We now show (in the context of Latin squares) that almost all $G\in\Phi(\dirK)$ have few loops of any given colour.
\begin{lemma}\label{fewloops}
Let~$\mathbf{L}$ be a uniformly random $n\times n$ Latin square with entries in~$[n]$, and suppose $t\geq 3\log n/\log\log n$.
Let~$\mathbf{X}$ be the random variable which returns the maximum (over the symbol set~$[n]$) number of times that any symbol appears on the  leading diagonal, in~$\mathbf{L}$.
Then $\prob{\mathbf{X}\geq t}\leq\exp(-\Omega(t\log t))$.
\end{lemma}
\begin{proof}
Let~$\cL_{n}$ be the set of~$n\times n$ Latin squares with symbols~$[n]$, and for $L,L'\in\cL_{n}$, write $L\sim L'$ if~$L'$ can be obtained from~$L$ via a permutation of the rows.
Clearly,~$\sim$ is an equivalence relation on~$\cL_{n}$.
Note that $\mathbf{L}$ can be obtained by first choosing an equivalence class $\mathbf S \in \cL_n /{\sim}$ uniformly at random and then choosing $\mathbf L \in \mathbf S$ uniformly at random.  We actually prove the stronger statement that for every equivalence class $S \in \cL_n /{\sim}$, if $\mathbf{L}\in S$ is chosen uniformly at random, then $\prob{\mathbf X \geq t} \leq \exp(-\Omega(t \log t))$.

Each equivalence class $S \in \cL_n/{\sim}$ has size~$n!$ and contains a unique representative~$L_{S,i}$ with every symbol on the leading diagonal being~$i$, for each $i\in[n]$.
Applying a uniformly random row permutation~$\bm{\sigma}$ to~$L_{S,i}$ yields a uniformly random element~$\mathbf{L}$ of $S$, and the number of appearances~$\mathbf{X}_{i}$ of~$i$ on the leading diagonal of~$\mathbf{L}$ is equal to the number of fixed points of~$\bm{\sigma}$.
Then, if $t\geq 3\log n/\log \log n$ and~$n$ is sufficiently large, we have by Lemma~\ref{montperm} and Stirling's formula that\COMMENT{Note $n/t!=\exp(\log n -t\log t+t-O(\log t))=\exp(\log n -t\log t +o(t\log t))\leq\exp(\log n -\frac{9}{10}t\log t)$. For $t=\omega(\log n/\log \log n)$, we have $t\log t = \omega(\log n)$, whence clearly $n/t!\leq\exp(-\frac{1}{2}t\log t)$, and applying the further union bound over symbols continues to satisfy the lemma statement. For $t=\Theta(\log n/\log \log n)$, we use $t\geq 3\log n/\log \log n$ to observe that
\[
\log n =\frac{\log n}{\log \log n}\log\left(\frac{\log n}{\log \log n}\right) +\frac{\log n \log \log \log n}{\log \log n}\leq\frac{1}{3}t\log\left(\frac{1}{3}t\right)+o(t\log t)\leq \frac{2}{5}t\log t.
\]
Thus $\prob{\mathbf{X}_{i}\geq t}\leq \exp(-\frac{1}{2}t\log t)$ in this case too, and applying the union bound we note that we still have $\exp(\log n-\frac{1}{2}t\log t)\leq \exp(-\frac{1}{10}t\log t)$.}
\[
\prob{\mathbf{X}_{i}\geq t}=\sum_{k=t}^{n}\frac{1}{k!}\sum_{j=0}^{n-k}\frac{(-1)^{j}}{j!}\leq \sum_{k=t}^{n}\frac{1}{k!}\leq\frac{n}{t!}\leq\exp\left(-\frac{1}{2}t\log t\right),
\]
where we have used the simple observation\COMMENT{For $k=n$, this sum is just the first term, namely $1$. For smaller~$k$, the length of the sum increases. Suppose~$k$ is such that the sum has $m\geq 2$ terms (i.e.\ set $m\coloneqq n-k+1$ and suppose $k\leq n-1$). If~$m$ is odd then the sum is equal to $1-\left(\frac{1}{1!}-\frac{1}{2!}\right) -\left(\frac{1}{3!}-\frac{1}{4!}\right)-\dots\leq 1$, since each of the bracketed terms is positive. If~$m$ is even then the sum is equal to the sum of the first $m-1$ terms (at most~$1$ as above) plus the $m$th term, but the $m$th term is negative.}
that $\sum_{j=0}^{n-k}\frac{(-1)^{j}}{j!}\leq 1$ for all $k\in[n]_{0}$.
A union bound over symbols $i\in[n]$ now completes the proof.
\end{proof}
%\end{comment}
%\textcolor{purple}{By symmetry of the roles of rows, columns, and colours in Latin squares, the weighting factor will also apply to our setting where in truth we restrict to a set of colours rather than rows and so the `reduced structures' we're looking at don't immediately look like Latin rectangles.}

Finally, we need the following theorem of Kwan and Sudakov~\cite[Theorem 2]{KS18}, originally stated in the context of `discrepancy' of random Latin squares.
Theorem~\ref{thm:discrepancy} ensures in particular that almost all $G\in\Phi(\dirK)$ are `lower-quasirandom' (see Definition~\ref{def:lowerquas}), which we will use when building and counting the almost-spanning rainbow directed path forests (see Lemma~\ref{lemma:pathforests}) that we later absorb into rainbow directed Hamilton cycles.
\begin{theorem}[Kwan and Sudakov~\cite{KS18}]\label{thm:discrepancy}
Let $\mathbf{G}\in\Phi(\dirK)$ be chosen uniformly at random.
Then with high probability, for all (not necessarily distinct) sets $U_{1}, U_{2}, D\subseteq[n]$, we have that
\[
\left|e_{G,D}(U_{1},U_{2})-\frac{|U_{1}||U_{2}||D|}{n}\right|=O\left(\sqrt{|U_{1}||U_{2}||D|}\log n + n\log^{2}n\right).
\]
\end{theorem}
%\textcolor{purple}{Clearly this implies the lower-quasirandomness property that we will need to build rainbow directed path forests, as discussed in Section~\ref{overview-section}. Talk about the fact that Kwan and Sudakov stated this result in the context of `discrepancy' in Latin squares.}

%\textcolor{purple}{Mention the opposite problem? I.e. to find the smallest~$k$ such that whp some symbol appears on~$D$ at least~$k$ times. I personally suspect that Theorem~\ref{main-thm-2} is \textit{not} best possible, as the term we get comes almost completely from the union bound. I slightly think that maybe even $k=2$ is the answer to the opposite problem, but I'm not sure and certainly wouldn't call it a conjecture. Also worth noting that the best possible results for Theorem~\ref{main-thm-2} and the opposite problem needn't actually be even the same order.}\textcolor{red}{Stuff like this in a brief Concluding Remarks section?}
\section{Absorbers via switchings}\label{switchsection}
The aim of this section is to prove that almost all $G\in\Phi(\dirK)$ have many well-distributed absorbing gadgets and bridging gadgets, which we define now (see also Figure~\ref{fig:gadgets}).
\begin{defin}\label{def:absgadg}
For a vertex~$v$ and a colour~$c$, a $(v,c)$\textit{-absorbing gadget} is a digraph~$A$ having vertex set $V(A)=\{v,x_{1},x_{2},\dots,x_{6}\}$ and arcs $E(A)=\{x_{1}v,vx_{2},x_{1}x_{2},x_{3}x_{4},x_{3}x_{5},x_{4}x_{6},x_{5}x_{6}\}$, equipped with a proper arc-colouring~$\phi_{A}$, such that the following holds:
\begin{itemize}
    \item $\phi_{A}(x_{1}v)=\phi_{A}(x_{4}x_{6})\eqqcolon f_{1}$;
    \item $\phi_{A}(vx_{2})=\phi_{A}(x_{3}x_{5})\eqqcolon f_{2}$;
    \item $\phi_{A}(x_{1}x_{2})=\phi_{A}(x_{3}x_{4})\eqqcolon f_{3}$;
    \item $\phi_{A}(x_{5}x_{6})=c$;
    \item the colours $f_{1},f_{2},f_{3},c$ are distinct.
\end{itemize}
In this case, we say $(x_4, x_5)$ is the pair of \textit{abutment vertices} of~$A$.
\end{defin}
\begin{defin}\label{def:bridgegadg}
For distinct vertices~$y$ and~$z$, a $(y,z)$\textit{-bridging gadget} is a digraph~$B$ such that $V(B)=\{y,z,w_{1},w_{2},\dots,w_{6}\}$ and $E(B)=\{yw_{1}, w_{2}w_{1}, w_{2}w_{3}, zw_{3}, w_{4}y, w_{4}w_{5}, w_{6}w_{5}, w_{6}z\}$, equipped with a proper arc-colouring~$\phi_{B}$, such that the following holds:
\begin{itemize}
    \item $\phi_{B}(yw_{1})=\phi_{B}(w_{6}w_{5})\eqqcolon d_{1}$;
    \item $\phi_{B}(w_{2}w_{1})=\phi_{B}(w_{6}z)\eqqcolon d_{2}$;
    \item $\phi_{B}(w_{4}y)=\phi_{B}(w_{2}w_{3})\eqqcolon d_{3}$;
    \item $\phi_{B}(w_{4}w_{5})=\phi_{B}(zw_{3})\eqqcolon d_{4}$;
    \item the colours $d_{1},\dots,d_{4}$ are distinct.
\end{itemize}
\end{defin}
\captionsetup{labelfont={rm}}
\captionsetup[sub]{font=small,labelfont={rm}}
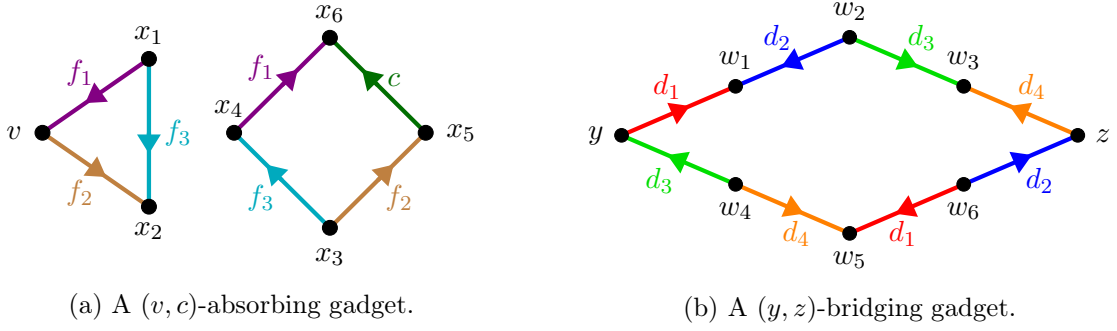
\begin{figure}
   \centering
  \vspace{0pt}
  \begin{subfigure}{.5\linewidth}
    \centering
    \begin{tikzpicture}[scale=1.4]
    \draw [->,ultra thick,mypurple,>=Triangle] (1,0.7)--(0.4,0.28);
    \draw [ultra thick,mypurple] (0.5,0.35)--(0,0);
    \draw [->,ultra thick,brown,>=Triangle] (0,0)--(0.6,-0.42);
    \draw [ultra thick,brown] (0.5,-0.35)--(1,-0.7);
    \draw [->,ultra thick,dblue,>=Triangle] (1,0.7)--(1,-0.2);
    \draw [ultra thick,dblue] (1,0)--(1,-0.7);
    \draw [->,ultra thick,dblue,>=Triangle] (2.7,-0.9)--(2.1,-0.3);
    \draw [ultra thick,dblue] (2.25,-0.45)--(1.8,0);
    \draw [->,ultra thick,mypurple,>=Triangle] (1.8,0)--(2.4,0.6);
    \draw [ultra thick,mypurple] (2.25,0.45)--(2.7,0.9);
    \draw [->,ultra thick,brown,>=Triangle] (2.7,-0.9)--(3.3,-0.3);
    \draw [ultra thick,brown] (3.15,-0.45)--(3.6,0);
    \draw [->,ultra thick,mypink,>=Triangle] (3.6,0)--(3,0.6);
    \draw [ultra thick,mypink] (3.15,0.45)--(2.7,0.9);
     \draw [fill] (0,0) circle [radius=0.07];
     \draw [fill] (1,0.7) circle [radius=0.07];
     \draw [fill] (1,-0.7) circle [radius=0.07];
     \draw [fill] (1.8,0) circle [radius=0.07];
     \draw [fill] (2.7,-0.9) circle [radius=0.07];
     \draw [fill] (2.7,0.9) circle [radius=0.07];
     \draw [fill] (3.6,0) circle [radius=0.07];
     \node [left] at (-0.1,0) {$v$};
     \node [below] at (1,-0.75) {$x_{2}$};
     \node [above] at (1,0.75) {$x_{1}$};
     \node [below] at (2.7,-1) {$x_{3}$};
     \node [above] at (1.75,0.05) {$x_{4}$};
     \node [right] at (3.7,0) {$x_{5}$};
     \node [above] at (2.7,0.95) {$x_{6}$};
     \node [above,mypurple] at (0.35,0.35) {$f_{1}$};
     \node [below,brown] at (0.35,-0.35) {$f_{2}$};
     \node [right,dblue] at (1.05,0) {$f_{3}$};
     \node [below,dblue] at (2.05,-0.4) {$f_{3}$};
     \node [below,brown] at (3.35,-0.4) {$f_{2}$};
     \node [above,mypurple] at (2.05,0.4) {$f_{1}$};
     \node [above,mypink] at (3.3,0.35) {$c$};
    \end{tikzpicture}%
    \caption{A $(v,c)$-absorbing gadget.}
    \label{fig:absgadg}
  \end{subfigure}%
%  \hfill\vline\hfill
  \begin{subfigure}{.5\linewidth}
    \centering
    \begin{tikzpicture}[scale=1]
   \draw[->,ultra thick,>=Triangle,red] (0,0)--(0.9,0.39);
\draw [ultra thick,red] (0.75,0.325)--(1.5,0.65);
\draw[->,ultra thick,>=Triangle,blue] (3,1.3)--(2.1,0.91);
\draw [ultra thick,blue] (2.25,0.975)--(1.5,0.65);
\draw[->,ultra thick,>=Triangle,mygreen] (3,1.3)--(3.9,0.91);
\draw [ultra thick,mygreen] (3.75,0.975)--(4.5,0.65);
\draw[->,ultra thick,>=Triangle,orange] (6,0)--(5.1,0.39);
\draw [ultra thick,orange] (5.25,0.325)--(4.5,0.65);
\draw[->,ultra thick,>=Triangle,mygreen] (1.5,-0.65)--(0.6,-0.26);
\draw [ultra thick,mygreen] (0.75,-0.325)--(0,0);
\draw[->,ultra thick,>=Triangle,orange] (1.5,-0.65)--(2.4,-1.04);
\draw [ultra thick,orange] (2.25,-0.975)--(3,-1.3);
\draw[->,ultra thick,>=Triangle,red] (4.5,-0.65)--(3.6,-1.04);
\draw [ultra thick,red] (3.75,-0.975)--(3,-1.3);
\draw[->,ultra thick,>=Triangle,blue] (4.5,-0.65)--(5.4,-0.26);
\draw [ultra thick,blue] (5.25,-0.325)--(6,0);
\draw [fill] (0,0) circle [radius=0.09];
\draw [fill] (1.5,0.65) circle [radius=0.09];
\draw [fill] (3,1.3) circle [radius=0.09];
\draw [fill] (4.5,0.65) circle [radius=0.09];
\draw [fill] (1.5,-0.65) circle [radius=0.09];
\draw [fill] (3,-1.3) circle [radius=0.09];
\draw [fill] (4.5,-0.65) circle [radius=0.09];
\draw [fill] (6,0) circle [radius=0.09];
\node [left] at (-0.1,0) {$y$};
\node [right] at (6.1,0) {$z$};
\node [above] at (1.5,0.75) {$w_{1}$};
\node [above] at (3,1.4) {$w_{2}$};
\node [above] at (4.5,0.75) {$w_{3}$};
\node [below] at (1.5,-0.75) {$w_{4}$};
\node [below] at (3,-1.4) {$w_{5}$};
\node [below] at (4.5,-0.75) {$w_{6}$};
\node [above,red] at (0.6,0.35) {$d_{1}$};
\node [above,blue] at (2.05,1) {$d_{2}$};
\node [above,mygreen] at (3.95,1) {$d_{3}$};
\node [above,orange] at (5.4,0.35) {$d_{4}$};
\node [below,mygreen] at (0.5,-0.3) {$d_{3}$};
\node [below,orange] at (2.3,-1) {$d_{4}$};
\node [below,red] at (3.7,-1) {$d_{1}$};
\node [below,blue] at (5.5,-0.3) {$d_{2}$};
    \end{tikzpicture}%
    \caption{A $(y,z)$-bridging gadget.}
    \label{fig:bridgegadg}
\end{subfigure}
%\end{center}%
    \caption{The key building blocks for the absorbing structure we build in Section~\ref{section:absorption}.}
    \label{fig:gadgets}
\end{figure}
As discussed in Section~\ref{overview-section}, the union of a $(v,c)$-absorbing gadget and an $(x_{4},x_{5})$-bridging gadget (together with some short rainbow directed paths) forms a structure we will call a $(v,c)$-\textit{absorber} (see Figure~\ref{fig:vcabsorber} and Definition~\ref{defn:(v,c)-absorber}), which is the key building block of our absorption structure.
To show that almost all $G\in\Phi(\dirK)$ contain the gadgets we need, we analyze switchings in the probability space corresponding to uniformly random $\mathbf{H}\in\cG_{D}$ (recall that~$\cG_{D}$ is the set of digraphs obtained from the digraphs in~$\Phi(\dirK)$ by deleting all arcs with colour not in~$D$)  for small~$D\subseteq[n]$, before applying Proposition~\ref{prop:wf} to compare this probability space with that of uniformly random $\mathbf{G}\in\Phi(\dirK)$ (see the proof of Lemma~\ref{main-switching-lemma}).

First, we need the following lemma, which asserts that for small~$D\subseteq[n]$, a uniformly random $\mathbf{H}\in\cG_{D}$ does not have too many more arcs than we would expect between any pair of vertex sets, each of size~$|D|$.
\begin{defin}\label{def:quas}
For $D\subseteq[n]$, we say that $H\in\cG_{D}$ is \textit{$\ell$-upper-quasirandom} if $e_{H}(A,B)\leq(1+\ell)|D|^{3}/n$ for all (not necessarily distinct) vertex sets $A,B\subseteq V(H)$ of sizes $|A|=|B|=|D|$.
We define~$\cQ_{D}^{\ell}\coloneqq\{H\in\cG_{D}\colon H\,\,\text{is}\,\,\ell\text{-upper-quasirandom}\}$.
\end{defin}

For a colour~$c\in D$ and uniformly random $\mathbf{H}\in\cG_{D}$, we write~$\mathbf{F}_{c}=F_{c}(\mathbf{H})$ for the random colour class of~$c$ in~$\mathbf{H}$ ($F$ here standing for `factor'), so that~$\mathbf{H}$ is determined by the random variables~$\{\mathbf{F}_{c}\}_{c\in D}$.
%We use the notation~$F_{c}$ for a fixed outcome of~$\mathbf{F}_{c}$.
%\textcolor{blue}{A little bit here about why it is important we do this?}
\begin{lemma}\label{quasi-lemma}
Suppose $D\subseteq[n]$ has size $|D|=n/10^{6}$.
Fix $c\in D$, let $\mathbf{H}\in\cG_{D}$ be chosen uniformly at random, and let $\mathbf{F}_{c}=F_{c}(\mathbf{H})$.
Then for any outcome~$F$ of~$\mathbf{F}_{c}$ we have
\[
\prob{\mathbf{H}\in\cQ_{D}^{1} \mid \mathbf{F}_{c}=F}\geq 1-\exp(-\Omega(n^{2})).
\]
\end{lemma}
The authors of~\cite{GKKO20} proved a lemma (\cite[Lemma 6.3]{GKKO20}) analogous to Lemma~\ref{quasi-lemma} in the undirected setting.
The proof of Lemma~\ref{quasi-lemma} is similar so we omit it here.  In the appendix\NOTAPPENDIX{ of the arXiv version of the paper}, we describe how the proof of \cite[Lemma 6.3]{GKKO20} can be modified to obtain a proof of Lemma~\ref{quasi-lemma}.

We condition on versions of upper-quasirandomness when we are using switching arguments to show that almost all $H\in\cG_{D}$ admit many absorbing gadgets and bridging gadgets.
%Indeed, without such a local sparsity condition, the density of arcs in some regions of~$H$ may mean that it is impossible to find many structures in~$H$ having arcs absent in the places necessary for our switching operation to be performed, creating a new bridging gadget, for example (see Definition~\ref{def:twistsystem}).
Further, we will need that~$H$ does not have many $c$-loops in order to find many $(v,c)$-absorbing gadgets, for any $v\in V(H)$. % as well as conditioning on a notion of upper-quasirandomness.
Lemma~\ref{quasi-lemma} enables us to `uncondition' from these two events, so as to study simply the probability that a uniformly random~$\mathbf{H}$ has many absorbing gadgets. % it is useful to have a statement of Lemma~\ref{quasi-lemma} which says that the upper-quasirandomness of~$\mathbf{H}$ is in some sense almost independent of~$\mathbf{F}_{c}$.
%(Indeed, this is to be expected since~$\mathbf{H}$ is the union of~$\Theta(n)$ colour classes.)

Since, as discussed in Section~\ref{overview-section}, we eventually piece together gadgets in a greedy fashion to build an absorbing structure in a typical $G\in\Phi(\dirK)$, it will be important to know that we can find collections~$\cA$ of gadgets which are `well-spread', in that no vertex or colour of~$G$ is contained in too many $A\in\cA$.
We formalise this notion in the following definition.
\begin{defin}\label{def:wellspread}
Suppose that~$G$ is an $n$-vertex directed, arc-coloured digraph with vertices~$V$ and colours~$C$.
Fix $v\in V$, $c\in C$, and fix $y,z\in V$ distinct.
We say that a collection~$\cA$ of $(v,c)$-absorbing gadgets in~$G$ is \textit{well-spread} if for all $u\in V\setminus\{v\}$ and $d\in C\setminus\{c\}$, there are at most~$n$ distinct $A\in\cA$ which contain~$u$, and at most~$n$ distinct $A\in\cA$ which contain~$d$.
We say that a collection~$\cB$ of $(y,z)$-bridging gadgets in~$G$ is \textit{well-spread} if for all $u\in V\setminus\{y,z\}$ and $d\in C$, there are at most~$n$ distinct $B\in\cB$ which contain~$u$, and at most~$n$ distinct $B\in\cB$ which contain~$d$.
\end{defin}
The next lemma ensures that almost all $G\in\Phi(\dirK)$ contain the collections of well-spread absorbing gadgets that we need.
\begin{lemma}\label{masterswitch1}
Let $\mathbf{G}\in\Phi(\dirK)$ be chosen uniformly at random, and let~$\cE$ be the event that for all $v,c\in[n]$,~$\mathbf{G}$ contains a well-spread collection of at least~$n^{2}/2^{100}$ $(v,c)$-absorbing gadgets.
Let~$\cC$ be the event that no colour class of~$\mathbf{G}$ has more than~$n/10^{9}$ loops.
Then $\prob{\cE\mid\cC}\geq 1-\exp(-\Omega(n^{2}))$, and in particular, $\prob{\cE}\geq 1-\exp(-\Omega(n\log n))$ by Lemma~\ref{fewloops}.
\end{lemma}
As with Lemma~\ref{quasi-lemma}, the authors of~\cite{GKKO20} proved an analogous lemma (\cite[Lemma 3.8]{GKKO20}) in the undirected setting with a similar proof, so we omit it here but provide details in the appendix\NOTAPPENDIX{ of the arXiv version of the paper} of how the proof of \cite[Lemma 3.8]{GKKO20} may be modified to prove Lemma~\ref{masterswitch1}.

%\begin{proof}
 %\textcolor{blue}{This here for my own notes. Comment on necessary changes from RHP paper to obtain this result (quasi -> switchings conditioned on degrading quasi -> this means many well-spread -> weighting factor). Suffices to do this in the setting where we use~$\mu$ to be a small constant, since then we obtain better than $n$-well-spread (but $O(n)$-well-spread is all we need), and the collection we obtain has size $\Theta(\mu^{4}n^{2})$, so just lower bound all the constants there by $1/\log n$. So here I've prioritised a neat statement that is good enough for our purposes, over a statement that is as strong as we can make it.}
%
% \textcolor{blue}{Later, we have flexible sets of size $n/\log^{2}n$. Suppose we are trying to find a gadget for some pair $(v,c)$. That pair $(v,c)$ has $n^{2}/\log n$ gadgets. For each of the $O(n/\log^{2}n)$ previously sorted pairs $(v',c')$ there are $O(n)$ $(v,c)$-gadgets disallowed by the choice of~$A_{v',c'}$ so total forbidden choice for $A_{v,c}$ is $O(n^{2}/\log^{2}n)=o(n^{2}/\log n)$. }
%\end{proof}
%\begin{lemma}\label{justgadgets1}
%Suppose that $1/n\ll\mu$, and let $D\subseteq[n]$ be such that $|D|\leq\mu n$.
%Let $x\in V$, let $c\in [n]\setminus D$, and let $\cP=\{D_{i}\}_{i=1}^{4}$ be an equitable partition of~$D$.
%Then for any integer $t\geq 0$ and any $G\in\cG_{D\cup\{c\}}$, if $r_{1}(G)\geq t$, then~$G$ contains a $5\mu n/4$-well-spread collection of~$t$ distinct $(x,c)$-absorbing gadgets.
%\end{lemma}
The rest of this section is dedicated to showing that almost all $G\in\Phi(\dirK)$ have large well-spread collections of bridging gadgets (recall Figure~\ref{fig:bridgegadg}). For technical reasons that make the switching argument a little easier to analyze, we instead actually look for a slightly more special structure.
In particular, we add some extra arcs so that all vertices we find are in the neighbourhood of~$y$ or of~$z$, we partition the colours to limit the number of `roles' certain arcs can play when we apply the switching operation, and we introduce the notion of \textit{distinguishability}, which will be useful when arguing that the gadgets we find are well-spread.
\begin{defin}\label{def:bridgestuff}
Let $D\subseteq[n]$, let $H\in\cG_{D}$, and let $\cP=(D_{i})_{i=1}^{6}$ be an equitable (ordered) partition of~$D$ into six parts. Let $y,z\in[n]$ be distinct vertices.
\begin{itemize}
\item We say that a subgraph $B\subseteq H$ is a $(y,z,\cP)$\textit{-bridge} (see Figure~\ref{fig:yzpbridge}) if~$B$ is the union of a $(y,z)$-bridging gadget~$B'$ (with vertex- and colour-labelling as in Definition~\ref{def:bridgegadg}) and the extra arcs~$yw_{2}$,~$zw_{5}$, such that $d_{i}\in D_{i}$ for all $i\in[4]$,~$\phi_{H}(yw_{2})\in D_{5}$, and~$\phi_{H}(zw_{5})\in D_{6}$;
\item we say that~a $(y,z,\cP)$-bridge~$B$ is \textit{distinguishable} in~$H$ if~$B$ is the only $(y,z,\cP)$-bridge in~$H$ containing any of the arcs~$w_{2}w_{1},w_{2}w_{3},w_{4}w_{5},w_{6}w_{5}$;
\item we write~$r_{(y,z,\cP)}(H)$ for the number of distinguishable $(y,z,\cP)$-bridges in~$H$;
\item for $s\in[n|D|]_{0}$, we write~$M_{s}^{(y,z,\cP)}$ for the set of $H\in\cG_{D}$ such that $r_{(y,z,\cP)}(H)=s$ and $e_{H}(A,B)\leq 2|D|^{3}/n +12s$ for all $A,B\subseteq[n]$ of size $|A|=|B|=|D|$, and we define $\widehat{\cQ}_{D}^{(y,z,\cP)}\coloneqq\bigcup_{s=0}^{n|D|}M_{s}^{(y,z,\cP)}$.
\end{itemize}
\end{defin}
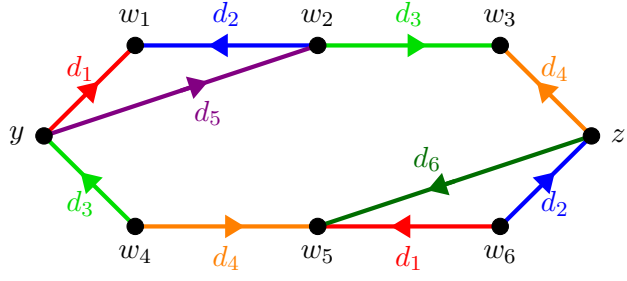
\begin{figure}
\centering
\begin{tikzpicture} [scale=1.2]
%\tikzstyle{arrow} = {thick,->,>=Triangle};
\draw[->,ultra thick,>=Triangle,red] (0,0)--(0.6,0.6);
\draw [ultra thick,red] (0.5,0.5)--(1,1);
\draw[->,ultra thick,>=Triangle,blue] (3,1)--(1.8,1);
\draw [ultra thick,blue] (2,1)--(1,1);
\draw[->,ultra thick,>=Triangle,mygreen] (3,1)--(4.2,1);
\draw [ultra thick,mygreen] (4,1)--(5,1);
\draw[->,ultra thick,>=Triangle,orange] (6,0)--(5.4,0.6);
\draw [ultra thick,orange] (5.5,0.5)--(5,1);
\draw[->,ultra thick,>=Triangle,mygreen] (1,-1)--(0.4,-0.4);
\draw [ultra thick,mygreen] (0.5,-0.5)--(0,0);
\draw[->,ultra thick,>=Triangle,orange] (1,-1)--(2.2,-1);
\draw [ultra thick,orange] (2,-1)--(3,-1);
\draw[->,ultra thick,>=Triangle,red] (5,-1)--(3.8,-1);
\draw [ultra thick,red] (4,-1)--(3,-1);
\draw[->,ultra thick,>=Triangle,blue] (5,-1)--(5.6,-0.4);
\draw [ultra thick,blue] (5.5,-0.5)--(6,0);
\draw[->,ultra thick,>=Triangle,mypurple] (0,0)--(1.8,0.6);
\draw [ultra thick,mypurple] (1.5,0.5)--(3,1);
\draw[->,ultra thick,>=Triangle,mypink] (6,0)--(4.2,-0.6);
\draw [ultra thick,mypink] (4.5,-0.5)--(3,-1);
\draw [fill] (0,0) circle [radius=0.09];
\draw [fill] (1,1) circle [radius=0.09];
\draw [fill] (3,1) circle [radius=0.09];
\draw [fill] (5,1) circle [radius=0.09];
\draw [fill] (1,-1) circle [radius=0.09];
\draw [fill] (3,-1) circle [radius=0.09];
\draw [fill] (5,-1) circle [radius=0.09];
\draw [fill] (6,0) circle [radius=0.09];
\node [left] at (-0.1,0) {$y$};
\node [right] at (6.1,0) {$z$};
\node [above] at (1,1.1) {$w_{1}$};
\node [above] at (3,1.1) {$w_{2}$};
\node [above] at (5,1.1) {$w_{3}$};
\node [below] at (1,-1.1) {$w_{4}$};
\node [below] at (3,-1.1) {$w_{5}$};
\node [below] at (5,-1.1) {$w_{6}$};
\node [above,red] at (0.4,0.5) {$d_{1}$};
\node [above,blue] at (2,1.1) {$d_{2}$};
\node [above,mygreen] at (4,1.1) {$d_{3}$};
\node [above,orange] at (5.6,0.5) {$d_{4}$};
\node [below,mygreen] at (0.4,-0.5) {$d_{3}$};
\node [below,orange] at (2,-1.1) {$d_{4}$};
\node [below,red] at (4,-1.1) {$d_{1}$};
\node [below,blue] at (5.6,-0.5) {$d_{2}$};
\node [below,mypurple] at (1.8,0.5) {$d_{5}$};
\node [above,mypink] at (4.2,-0.5) {$d_{6}$};
%\draw [arrow] (0,0)--(1,1);
%\draw [green,ultra thick] (0,0)--(0,4);
%\draw [dashed] (0,0)--(3.4,-2);
%\draw [fill] (3.4,-2) circle [radius=0.15];
%\node [below] at (0,0) {$u_{2}$};
\end{tikzpicture}
\caption{A $(y,z,\cP)$-bridge, with $\cP=(D_{i})_{i=1}^{6}$ and $d_{i}\in D_{i}$ for each $i\in[6]$.}
\label{fig:yzpbridge}
\end{figure}
We frequently drop the $(y,z,\cP)$-notation in the terminology introduced above when the tuple~$(y,z,\cP)$ is clear from context.
For every distinct $y,z\in[n]$ and equitable partition $\cP=(D_{i})_{i=1}^{6}$,
\begin{equation}\label{eq:qhat}
\cQ_{D}^{1}\subseteq\widehat{\cQ}_{D}, \text{ and if $s \leq |D|^4 / (10^{24}n^2)$, then $M_s \subseteq \cQ_D^{2}$}.
\end{equation}
In Lemma~\ref{masterswitch2} we use switchings on some $H\in\cG_{D}$ to produce some $H'\in \cG_{D}$ having $r(H')=r(H)+1$.
As mentioned earlier, we condition on upper-quasirandomness in this lemma; more specifically, we will condition that $\mathbf{H}\in\widehat{\cQ}_{D}$.
The notion of distinguishability of $(y,z,\cP)$-bridges is useful because, as we show in Lemma~\ref{main-switching-lemma},~Claim~\ref{claim:spreadgadgets}, a collection of distinguishable $(y,z,\cP)$-bridges in~$G|_{D}$ is necessarily well-spread (recall Definition~\ref{def:wellspread}).

We now discuss the switching operation that forms the backbone of the proof of Lemma~\ref{masterswitch2}.
\begin{defin}\label{def:twistsystem}
Let $D\subseteq[n]$, let $H\in\cG_{D}$, let $\cP=(D_{i})_{i=1}^{6}$ be a partition of~$D$ and suppose $y,z\in[n]$ are distinct.
Let $u_{1},\dots,u_{6},u'_{1},\dots,u'_{8},u''_{1},\dots,u''_{8}\in[n]\setminus\{y,z\}$, where $u_{1},\dots,u_{6},u''_{1},\dots,u''_{8}$ are distinct and $\{u'_{1},\dots,u'_{8}\}\cap\{u_{1},\dots,u_{6},u''_{1},\dots,u''_{8}\}=\emptyset$.
Let $U^{\text{int}}\coloneqq\{u_{1},u_{2},\dots,u_{6}\}, U^{\text{mid}}\coloneqq\{u'_{1},u'_{2},\dots,u'_{8}\}, U^{\text{ext}}\coloneqq\{u''_{1},u''_{2},\dots,u''_{8}\}$.\COMMENT{Many of these repetitions are avoidable, but we don't have to avoid them, so this may keep the argument more concise. One could also have allowed~$U^{\text{ext}}$ to have repeated elements, which means the `floating' extra arcs in the switching cycles can be loops - but one runs into more delicate issues, for example clearly we still need at least four distinct vertices in~$U^{\text{ext}}$ for example. This just seems the neatest way to go.}
Then we say that a subgraph $T\subseteq H[\{y,z\}\cup U^{\text{int}}\cup U^{\text{mid}}\cup U^{\text{ext}}]$ is a \textit{twist system} (see Figure~\ref{fig:twistsystem}) of~$H$ if:
\begin{enumerate}[label=\upshape(\roman*)]
    \item $E(T)=\{yu_{1},yu_{2},u_{4}y,zu_{3},zu_{5},u_{6}z,u'_{1}u_{1},u_{2}u'_{2},u_{2}u'_{3},u'_{4}u_{3},u_{4}u'_{5},u'_{6}u_{5},u'_{7}u_{5},u_{6}u'_{8},u''_{2}u''_{1},\newline u''_{3}u''_{4},u''_{5}u''_{6},u''_{8}u''_{7}\}$;
    \item $\phi_{H}\left(yu_{1}\right)=\phi_{H}\left(u'_{7}u_{5}\right)=\phi_{H}\left(u''_{8}u''_{7}\right)=\phi_{H}\left(u_{6}u'_{8}\right)\in D_{1}$;
    \item $\phi_{H}\left(u_{6}z\right)=\phi_{H}\left(u_{2}u'_{2}\right)=\phi_{H}\left(u''_{2}u''_{1}\right)=\phi_{H}\left(u'_{1}u_{1}\right)\in D_{2}$;
    \item $\phi_{H}\left(u_{4}y\right)=\phi_{H}\left(u_{2}u'_{3}\right)=\phi_{H}\left(u''_{3}u''_{4}\right)=\phi_{H}\left(u'_{4}u_{3}\right)\in D_{3}$;
    \item $\phi_{H}\left(zu_{3}\right)=\phi_{H}\left(u_{4}u'_{5}\right)=\phi_{H}\left(u''_{5}u''_{6}\right)=\phi_{H}\left(u'_{6}u_{5}\right)\in D_{4}$;
    \item $\phi_{H}\left(yu_{2}\right)\in D_{5}$ and $\phi_{H}\left(zu_{5}\right)\in D_{6}$;
    \item $u_{2}u_{1},u'_{1}u''_{1},u''_{2}u'_{2},u_{2}u_{3},u''_{3}u'_{3},u'_{4}u''_{4},u_{4}u_{5},u''_{5}u'_{5},u'_{6}u''_{6},u_{6}u_{5},u'_{7}u''_{7},u''_{8}u'_{8}\notin E(H)$.
\end{enumerate}
For a twist system $T\subseteq H$, we define~$\text{twist}_{T}(H)$ to be the coloured digraph obtained from~$H$ by deleting the arcs $u'_{1}u_{1}, u''_{2}u''_{1}, u_{2}u'_{2}, u_{2}u'_{3}, u''_{3}u''_{4}, u'_{4}u_{3}, u_{4}u'_{5}, u''_{5}u''_{6}, u'_{6}u_{5}, u_{6}u'_{8}, u''_{8}u''_{7}, u'_{7}u_{5}$, and adding the arcs $u_{6}u_{5}, u'_{7}u''_{7}, u''_{8}u'_{8}$ each in colour~$\phi_{H}\left(yu_{1}\right)$, the arcs $u_{4}u_{5}, u'_{6}u''_{6}, u''_{5}u'_{5}$ each in colour~$\phi_{H}\left(zu_{3}\right)$, the arcs $u_{2}u_{3}, u'_{4}u''_{4}, u''_{3}u'_{3}$ each in colour~$\phi_{H}\left(u_{4}y\right)$, and the arcs $u_{2}u_{1}, u'_{1}u''_{1}, u''_{2}u'_{2}$ each in colour~$\phi_{H}\left(u_{6}z\right)$.
The $(y,z,\cP)$-bridge in~$\text{twist}_{T}(H)$ with arc set $\{yu_{1},yu_{2}, u_{2}u_{1}, u_{2}u_{3},zu_{3},\newline u_{6}z, zu_{5}, u_{6}u_{5},u_{4}u_{5}, u_{4}y\}$ is called the \textit{canonical} $(y,z,\cP)$\textit{-bridge} of the twist.
\end{defin}
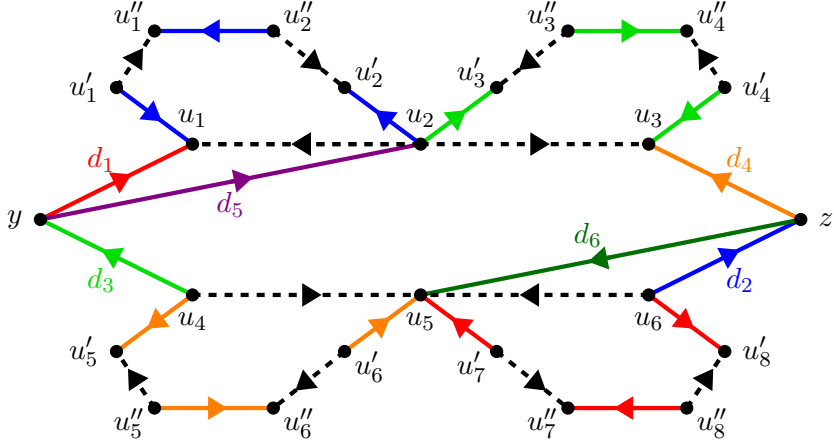
\begin{figure}
\centering
\begin{tikzpicture} %scale=[1.0]
\draw [->,ultra thick,>=Triangle,red] (0,0)--(1.2,0.6);
\draw [ultra thick, red] (1,0.5)--(2,1);
\draw [->,ultra thick,>=Triangle,blue] (1,1.75)--(1.6,1.3);
\draw [ultra thick, blue] (1.5,1.375)--(2,1);
\draw [->,ultra thick,>=Triangle,blue] (3.0625,2.5)--(2.1,2.5);
\draw [ultra thick, blue] (2.4,2.5)--(1.5,2.5);
\draw [->,ultra thick,>=Triangle,blue] (5,1)--(4.4,1.45);
\draw [ultra thick, blue] (4.5,1.375)--(4,1.75);
\draw [->,ultra thick,>=Triangle,mygreen] (9,1.75)--(8.4,1.3);
\draw [ultra thick, mygreen] (8.5,1.375)--(8,1);
\draw [->,ultra thick,>=Triangle,mygreen] (6.9375,2.5)--(7.9,2.5);
\draw [ultra thick, mygreen] (7.6,2.5)--(8.5,2.5);
\draw [->,ultra thick,>=Triangle,mygreen] (5,1)--(5.6,1.45);
\draw [ultra thick, mygreen] (5.5,1.375)--(6,1.75);
\draw [->,ultra thick,>=Triangle,orange] (10,0)--(8.8,0.6);
\draw [ultra thick, orange] (9,0.5)--(8,1);
\draw [->,ultra thick,>=Triangle,mygreen] (2,-1)--(0.8,-0.4);
\draw [ultra thick, mygreen] (1,-0.5)--(0,0);
\draw [->,ultra thick,>=Triangle,orange] (2,-1)--(1.4,-1.45);
\draw [ultra thick, orange] (1.5,-1.375)--(1,-1.75);
\draw [->,ultra thick,>=Triangle,orange] (1.5,-2.5)--(2.4625,-2.5);
\draw [ultra thick, orange] (2.1625,-2.5)--(3.0625,-2.5);
\draw [->,ultra thick,>=Triangle,orange] (4,-1.75)--(4.6,-1.3);
\draw [ultra thick, orange] (4.5,-1.375)--(5,-1);
\draw [->,ultra thick,>=Triangle,red] (8,-1)--(8.6,-1.45);
\draw [ultra thick, red] (8.5,-1.375)--(9,-1.75);
\draw [->,ultra thick,>=Triangle,red] (8.5,-2.5)--(7.5375,-2.5);
\draw [ultra thick, red] (7.8375,-2.5)--(6.9375,-2.5);
\draw [->,ultra thick,>=Triangle,red] (6,-1.75)--(5.4,-1.3);
\draw [ultra thick, red] (5.5,-1.375)--(5,-1);
\draw [->,ultra thick,>=Triangle,blue] (8,-1)--(9.2,-0.4);
\draw [ultra thick, blue] (9,-0.5)--(10,0);
\draw [->,ultra thick,>=Triangle,mypurple] (0,0)--(2.8,0.56);
\draw [ultra thick, mypurple] (2.5,0.5)--(5,1);
\draw [->,ultra thick,>=Triangle,mypink] (10,0)--(7.2,-0.56);
\draw [ultra thick, mypink] (7.5,-0.5)--(5,-1);
\draw [->,dashed,ultra thick,>=Triangle] (5,1)--(3.3,1);
\draw [dashed,ultra thick] (3.5,1)--(2,1);
\draw [->,dashed,ultra thick,>=Triangle] (5,1)--(6.7,1);
\draw [dashed,ultra thick] (6.5,1)--(8,1);
\draw [->,dashed,ultra thick,>=Triangle] (2,-1)--(3.7,-1);
\draw [dashed,ultra thick] (3.5,-1)--(5,-1);
\draw [->,dashed,ultra thick,>=Triangle] (8,-1)--(6.3,-1);
\draw [dashed,ultra thick] (6.5,-1)--(5,-1);
\draw [->,dashed,ultra thick,>=Triangle] (1,1.75)--(1.35,2.275);
\draw [dashed,ultra thick] (1.25,2.125)--(1.5,2.5);
\draw [->,dashed,ultra thick,>=Triangle] (5,1)--(3.3,1);
\draw [dashed,ultra thick] (3.5,1)--(2,1);
\draw [->,dashed,ultra thick,>=Triangle] (3.0625,2.5)--(3.7125,1.98);
\draw [dashed,ultra thick] (3.53125,2.125)--(4,1.75);
\draw [->,dashed,ultra thick,>=Triangle] (9,1.75)--(8.65,2.275);
\draw [dashed,ultra thick] (8.75,2.125)--(8.5,2.5);
\draw [->,dashed,ultra thick,>=Triangle] (6.9375,2.5)--(6.2875,1.98);
\draw [dashed,ultra thick] (6.46875,2.125)--(6,1.75);
\draw [->,dashed,ultra thick,>=Triangle] (1.5,-2.5)--(1.15,-1.975);
\draw [dashed,ultra thick] (1.25,-2.125)--(1,-1.75);
\draw [->,dashed,ultra thick,>=Triangle] (4,-1.75)--(3.35,-2.27);
\draw [dashed,ultra thick] (3.53125,-2.125)--(3.0625,-2.5);
\draw [->,dashed,ultra thick,>=Triangle] (8.5,-2.5)--(8.85,-1.975);
\draw [dashed,ultra thick] (8.75,-2.125)--(9,-1.75);
\draw [->,dashed,ultra thick,>=Triangle] (6,-1.75)--(6.65,-2.27);
\draw [dashed,ultra thick] (6.46875,-2.125)--(6.9375,-2.5);
\draw [fill] (0,0) circle [radius=0.08];
\draw [fill] (2,1) circle [radius=0.08];
\draw [fill] (5,1) circle [radius=0.08];
\draw [fill] (8,1) circle [radius=0.08];
\draw [fill] (10,0) circle [radius=0.08];
\draw [fill] (2,-1) circle [radius=0.08];
\draw [fill] (5,-1) circle [radius=0.08];
\draw [fill] (8,-1) circle [radius=0.08];
\draw [fill] (1,1.75) circle [radius=0.08];
\draw [fill] (1.5,2.5) circle [radius=0.08];
\draw [fill] (3.0625,2.5) circle [radius=0.08];
\draw [fill] (4,1.75) circle [radius=0.08];
\draw [fill] (9,1.75) circle [radius=0.08];
\draw [fill] (8.5,2.5) circle [radius=0.08];
\draw [fill] (6.9375,2.5) circle [radius=0.08];
\draw [fill] (6,1.75) circle [radius=0.08];
\draw [fill] (1,-1.75) circle [radius=0.08];
\draw [fill] (1.5,-2.5) circle [radius=0.08];
\draw [fill] (3.0625,-2.5) circle [radius=0.08];
\draw [fill] (4,-1.75) circle [radius=0.08];
\draw [fill] (9,-1.75) circle [radius=0.08];
\draw [fill] (8.5,-2.5) circle [radius=0.08];
\draw [fill] (6.9375,-2.5) circle [radius=0.08];
\draw [fill] (6,-1.75) circle [radius=0.08];
\node [left] at (-0.1,0) {$y$};
\node [above] at (2,1.1) {$u_{1}$};
\node [above] at (5,1.1) {$u_{2}$};
\node [above] at (8,1.1) {$u_{3}$};
\node [below] at (2,-1.1) {$u_{4}$};
\node [below] at (5,-1.1) {$u_{5}$};
\node [below] at (8,-1.1) {$u_{6}$};
\node [right] at (10.1,0) {$z$};
\node [left] at (0.9,1.75) {$u'_{1}$};
\node [right] at (4,1.95) {$u'_{2}$};
\node [left] at (6,1.95) {$u'_{3}$};
\node [right] at (9.1,1.75) {$u'_{4}$};
\node [left] at (0.9,-1.75) {$u'_{5}$};
\node [right] at (4,-1.95) {$u'_{6}$};
\node [left] at (6,-1.95) {$u'_{7}$};
\node [right] at (9.1,-1.75) {$u'_{8}$};
\node [left] at (1.5,2.7) {$u''_{1}$};
\node [right] at (3.0625,2.7) {$u''_{2}$};
\node [left] at (6.9375,2.7) {$u''_{3}$};
\node [right] at (8.5,2.7) {$u''_{4}$};
\node [left] at (1.5,-2.7) {$u''_{5}$};
\node [right] at (3.0625,-2.7) {$u''_{6}$};
\node [left] at (6.9375,-2.7) {$u''_{7}$};
\node [right] at (8.5,-2.7) {$u''_{8}$};
\node [above,red] at (0.8,0.5) {$d_{1}$};
\node [above,orange] at (9.2,0.5) {$d_{4}$};
\node [below,mygreen] at (0.8,-0.5) {$d_{3}$};
\node [below,blue] at (9.2,-0.5) {$d_{2}$};
\node [below,mypurple] at (2.5,0.5) {$d_{5}$};
\node [above,mypink] at (7.2,-0.5) {$d_{6}$};
\end{tikzpicture}
\caption{A twist system. Here, $d_{i}\in D_{i}$ for each $i\in[6]$, and dashed arcs indicate an arc which is absent in~$H$.}
\label{fig:twistsystem}
\end{figure}
\COMMENT{The arc set plus knowledge of~$\cP$ is enough to know which arcs to swap in the twist operation: $y$ and~$z$ are already known. Only one arc has colour in~$D_{5}$ and only one has colour in~$D_{6}$. The heads of these arcs are~$u_{2}$ and~$u_{5}$ respectively. The unique still-unlabelled out-neighbour of~$y$ must then be~$u_{1}$. Similarly we identify~$u_{3}$, $u_{4}$, and~$u_{6}$. The remaining unlabelled in-neighbour of~$u_{1}$ must be~$u'_{1}$. The out-neighbour of~$u_{2}$ via a colour in~$D_{2}$ must be~$u'_{2}$. Similarly we identify $u'_{3},\dots,u'_{8}$. Now the remaining arcs with unlabelled endpoints are each uniquely in one of $D_{1},\dots,D_{4}$. The colour identifies which and now the directions identify all remaining vertices. With all vertices labelled, it's clear how the arcs should swap.}Notice that if $H\in\cG_{D}$ and~$T$ is a twist system of~$H$, then $\text{twist}_{T}(H)\in\cG_{D}$, even if $u'_{1},\dots,u'_{8}$ are not distinct.
%We consider two twist systems~$T$ and~$T'$ to be distinct if their arc sets are distinct.\COMMENT{Thanks to~$\cP$, if one has the arc set, then one knows the label/role of every vertex, so the twist system is determined.}
We now use the twist switching operation to argue that almost all $H\in\widehat{\cQ}_{D}$ have many distinguishable $(y,z,\cP)$-bridges, for fixed $y,z,\cP$, and appropriately sized~$D$.
\begin{lemma}\label{masterswitch2}
Suppose $D\subseteq[n]$ has size $|D|=n/10^{6}$.
Let $y,z\in[n]$ be distinct, and let $\cP=(D_{i})_{i=1}^{6}$ be an equitable partition of~$D$.
Let $\mathbf{H}\in\cG_{D}$ be chosen uniformly at random.
Then
\[
\prob{r(\mathbf{H})\leq \frac{n^{2}}{10^{50}}\biggm|\mathbf{H}\in\widehat{\cQ}_{D}}\leq\exp\left(-\Omega\left(n^{2}\right)\right).
\]
\end{lemma}
\begin{proof}
Let $k\coloneqq |D|$.
Recall that $M_{1},\dots,M_{nk}$ is a partition of~$\widehat{\cQ}_{D}$ (see Definition~\ref{def:bridgestuff}).
For each $s\in[nk-1]_{0}$ we define an auxiliary bipartite digraph~$B_{s}$ with vertex bipartition~$(M_{s},M_{s+1})$ by putting an arc~$HH'$ whenever~$H\in M_{s}$ contains a twist system~$T$ for which the canonical $(y,z,\cP)$-bridge of the twist is distinguishable in~$\text{twist}_{T}(H)\eqqcolon H'$ and $H'\in M_{s+1}$.\COMMENT{Notice the simpler rule for adjacency this time. It seems we will have to find an upper bound for the number of distinct twist systems~$T$ yielding the same $H'=\text{twist}_{T}(H)$ and divide~$\delta^{+}_{s}$ by that number, but it turns out that distinguishability implies there is no overcounting.}
Define $\delta^{+}_{s}\coloneqq\min_{H\in M_{s}}d^{+}_{B_{s}}(H)$ and $\Delta^{-}_{s+1}\coloneqq\max_{H'\in M_{s+1}}d^{-}_{B_{s}}(H')$, and note that $|M_{s}|/|M_{s+1}|\leq\Delta^{-}_{s+1}/\delta^{+}_{s}$.
We will show that $|M_{s}|/|M_{s+1}|\leq 1/10$, if~$M_{s}$ is non-empty.
To that end, we first obtain an upper bound for~$\Delta^{-}_{s+1}$.
Fix $H'\in M_{s+1}$.
There are~$s+1$ choices of a distinguishable $(y,z,\cP)$-bridge~$B$ in~$H'$ which could have been the canonical $(y,z,\cP)$-bridge of a twist of a graph $H\in M_{s}$ producing~$H'$.
There are then at most~$n^{8}$ choices for the eight additional arcs added by a twist whose canonical bridge is~$B$ since the colours of these arcs are determined by~$B$ and there are~$n$ arcs of each colour in~$H'$.
For any such sequence of choices, there is a unique $H \in M_s$ and twist system $T\subseteq H$ such that $twist_T(H) = H'$, so we determine that $\Delta^{-}_{s+1}\leq (s+1)n^{8}$,\COMMENT{If multiple such antitwists lead to the same $H\in M_{s}$, this only helps us, in that we sought only an upper bound.} for all $s\in[nk-1]_{0}$.

We now find a lower bound for~$\delta^{+}_{s}$, in the case where $s\leq k^{4}/(10^{24}n^{2})$.
Fix $H\in M_{s}$.
We proceed by finding a large collection~$\cT$ of distinct twist systems in~$H$, such that for each $T\in\cT$, the canonical $(y,z,\cP)$-bridge of the twist is distinguishable in $\text{twist}_{T}(H)\eqqcolon H'$ and $H'\in M_{s+1}$.
We do this by ensuring that for any $T\in\cT$, the arc deletions involved in twisting on~$T$ do not decrease~$r(H)$, and that the only $(y,z,\cP)$-bridge created by the arc additions involved in twisting on~$T$ is the canonical $(y,z,\cP)$-bridge~$B$ of the twist (whence~$B$ is evidently distinguishable in~$H'$).
We first use the assumption on~$s$ to argue that~$H$ is not far from being upper-quasirandom (as per Definition~\ref{def:quas}).
Indeed, since $H\in M_{s}$ and $s\leq k^{4}/(10^{24}n^{2})$ we have by Definition~\ref{def:bridgestuff} that for any sets $W_{1},W_{2}\subseteq[n]$ of sizes $|W_{1}|=|W_{2}|=k$, \begin{equation}\label{eq:quasieq}
e_{H}(W_{1},W_{2})\leq\frac{2k^{3}}{n}+12s\leq \frac{3k^{3}}{n}.
\end{equation}
We now use~(\ref{eq:quasieq}) to find a large set~$\Lambda$ of choices for a sequence of colours and arcs $\lambda=(d_{1},d_{2},\dots,d_{6},e_{1},\dots,e_{4})$ with $d_{i}\in D_{i}$, and $e_{i}\in E_{d_{i}}(H)$, such that~$\lambda$ has a number of desirable properties.
We simultaneously use such a sequence~$\lambda$ to choose vertices $u_{1},\dots,u_{6},u'_{1},\dots,u'_{8}$, $u''_{1},\dots,u''_{8}$ as in Definition~\ref{def:twistsystem} and a subgraph $T_{\lambda}\subseteq H[\{y,z\}\cup U_{\text{int}}\cup U^{\text{mid}}\cup U^{\text{ext}}]$, thus constructing a set~$\cT$ of such~$T_{\lambda}$ by ranging over all $\lambda\in\Lambda$.
We will then use the known properties of the sequences $\lambda\in\Lambda$ to verify that each $T_{\lambda}\in\cT$ is a twist system for which the arc~$H\text{twist}_{T_{\lambda}}(H)$ is in~$B_{s}$.
%For presentational clarity and ease of recall, we give a claim for each~$D_{i}^{\text{good}}$.
\begin{claim}\label{d12}
There is a set~$D_{1,2}^{\text{good}}$ of pairs $(d_{1},d_{2})\in D_{1}\times D_{2}$ such that $|D_{1,2}^{\text{good}}|\geq k^{2}/100$ and each $(d_{1},d_{2})\in D_{1,2}^{\text{good}}$ satisfies the following,
where $u_{1}\coloneqq\neigh{+}{d_{1}}{y}$, $u'_{1}\coloneqq\neigh{-}{d_{2}}{u_{1}}$, $u_{6}\coloneqq\neigh{-}{d_{2}}{z}$, $u'_{8}\coloneqq\neigh{+}{d_{1}}{u_{6}}$.
\begin{itemize}
    \item [$(D_{1}1)$] There are at most~$10^{8}$ loops with colour~$d_{1}$ in~$H$;
    \item [$(D_{1}2)$] $u_{1}$ has at most~$300k^{2}/n$ in-neighbours in the set~$\neighset{+}{D_{5}}{y}$;
    \item [$(D_{1}3)$] there are at most~$k/100$ arcs~$e$ coloured~$d_{1}$ in~$H$ such that~$e$ is contained in a distinguishable $(y,z,\cP)$-bridge in~$H$; 
    \item [$(D_{2}1)$] there are at most~$10^{8}$ loops with colour~$d_{2}$ in~$H$;
    \item [$(D_{2}2)$] $u_{6}$ has at most~$300k^{2}/n$ out-neighbours in the set~$\neighset{+}{D_{6}}{z}$;
    \item [$(D_{2}3)$] there are at most~$k/100$ arcs~$e$ coloured~$d_{2}$ in~$H$ such that~$e$ is contained in a distinguishable $(y,z,\cP)$-bridge in~$H$; 
    \item [$(V_{1,2})$] the vertices~$y,z,u_{1},u_{6}$ are distinct, and $u'_{1},u'_{8}\notin\{y,z,u_{1},u_{6}\}$;
    \item [$(R_{1,2})$] there is no distinguishable $(y,z,\cP)$-bridge in~$H$ containing the arc~$u'_{1}u_{1}$ or the arc~$u_{6}u'_{8}$.
\end{itemize}
\end{claim}
\claimproof{}
For $i\in[3]$, let~$D_{1,i}$ be the set of colours $d_{1}\in D_{1}$ that fail to satisfy~$(D_{1}i)$.
Since~$H$ contains at most~$n$ loops, $|D_{1,1}|\leq n/10^{8}=k/100$.
Since $e_{H}(N_{D_{5}}^{+}(y),N_{D_{1}}^{+}(y))\leq 3k^{3}/n$ by~(\ref{eq:quasieq}), $|D_{1,2}|\leq k/100$.
Since any $d_{1}$-arc of~$H$ whose tail is not~$y$ is contained in at most one distinguishable $(y,z,\cP)$-bridge~$B$ and each~$B$ contains two such arcs, $r(H)\geq |D_{1,3}|(k/100-1)/2$.
Thus, $|D_{1,3}|\leq k/1000$.
Let $D'_{1}\coloneqq D_{1}\setminus (D_{1,1}\cup D_{1,2}\cup D_{1,3})$,
%There are at most $n/10^{8}=k/100$ colours $d_{1}\in D_{1}$ which fail to satisfy~$(D_{1}1)$, at most~$k/100$ colours $d_{1}\in D_{1}$ which fail to satisfy~$(D_{1}2)$ (indeed,~(\ref{eq:quasieq}) gives $e_{H}(N_{D_{5}}^{+}(y),N_{D_{1}}^{+}(y))\leq 3k^{3}/n$ (by arbitrarily extending~$N_{D_{5}}^{+}(y)$ and~$N_{D_{1}}^{+}(y)$ to sets of size~$k$)), and at most~$k/1000$ colours $d_{1}\in D_{1}$ which fail to satisfy~$(D_{1}3)$, since otherwise $r(H)\geq\frac{k}{1000}\left(\frac{k}{100}-1\right)/2>\frac{k^{4}}{10^{24}n^{2}}\geq s$, a contradiction.
%(Any $d_{1}$-arc whose tail is not~$y$ is contained in at most one distinguishable $(y,z,\cP)$-bridge~$B$, and each~$B$ contains two such arcs.)
%Let~$D'_{1}$ be the remaining $d_{1}\in D_{1}$
and notice that $|D'_{1}|\geq k/6-k/50-k/1000\geq 7k/50$.
Similarly there is a set $D'_{2}\subseteq D_{2}$ of size at least~$7k/50$ such that each $d_{2}\in D'_{2}$ satisfies~$(D_{2}1)$--$(D_{2}3)$.
At most two colours $d_{1}\in D'_{1}$ yield $u_{1}\in\{y,z\}$, and for any $d_{1}\in D'_{1}$ there are at most three choices of $d_{2}\in D'_{2}$ such that $u'_{1}\in\{y,z,u_{1}\}$, and at most three choices of $d_{2}\in D'_{2}$ such that $u_{6}\in\{y,z,u_{1}\}$.
For fixed $d_{1}\in D'_{1}$, by~$(D_{1}1)$ there are at most~$10^{8}$ choices of $d_{2}\in D'_{2}$ such that $u'_{8}=u_{6}$, and at most three choices of~$d_{2}$ such that $u'_{8}\in\{y,z,u_{1}\}$.
%For all $(d_{1},d_{2})$ such that $u_{6}\neq y$, we must have $u'_{8}\neq u_{1}$ (since otherwise this vertex has two distinct $d_{1}$-inneighbours).
Finally, if~$u_{1}$ and~$z$ are distinct, then we have $u'_{1}\neq u_{6}$, since otherwise $u_{1}=z$ has two distinct $d_{2}$-out-neighbours.
%, and for any $(d_{1},d_{2})$ such that $u_{1}\neq z$, we must have $u'_{1}\neq u_{6}$.
%Finally, we cannot have $u'_{1}=y$ since otherwise there are two distinct arcs~$yu_{1}$, and similarly $u'_{8}\neq z$.
Since $|D'_{1}|, |D'_{2}|\leq k/6$, we deduce that we can remove at most $2|D'_{2}|+(10^{8}+9)|D'_{1}|\leq10^{8}k/3$ pairs from $D'_{1}\times D'_{2}$ to ensure that all remaining pairs satisfy~$(V_{1,2})$.
To address~$(R_{1,2})$,
%For~$(R_{1,2})$, we define~$D_{1}^{\text{bad}}$ to be the set of colours $d_{1}\in D_{1}$ such that the set~$F_{d_{1}}$ of those $d_{1}$-arcs of~$H$ contained in some distinguishable $(y,z,\cP)$-bridge in~$H$ satisfies $|F_{d_{1}}|\geq k/1000$.
%Note that $|D_{1}^{\text{bad}}|\leq k/1000$, since otherwise $r(H)\geq \frac{k}{1000}\cdot\left(\frac{k}{1000}-1\right)/2>\frac{k^{4}}{2^{16}n^{2}}\geq s$, a contradiction. (Any $d_{1}$-arc whose tail is not~$y$ is contained in at most one distinguishable $(y,z,\cP)$-bridge~$B$, and each~$B$ contains two such arcs.)
%Defining~$D_{2}^{\text{bad}}$ analogously we have $|D_{2}^{\text{bad}}|\leq k/1000$.
notice that for fixed $d_{1}\in D'_{1}$, by~$(D_{1}3)$ there are at most~$k/100$ choices of $d_{2}\in D'_{2}$ such that~$u_{6}u'_{8}$ is contained in a distinguishable $(y,z,\cP)$-bridge~$B$ in~$H$.
Handling~$u'_{1}u_{1}$ analogously we deduce that we may remove at most $2\cdot\frac{k}{100}\cdot\frac{k}{6}$ pairs from~$D'_{1}\times D'_{2}$ to ensure all remaining pairs satisfy~$(R_{1,2})$.
In total the number of pairs in~$D'_{1}\times D'_{2}$ satisfying~$(V_{1,2})$ and~$(R_{1,2})$ is at least $(7k/50)^{2}-10^{8}k/3-k^{2}/300\geq k^{2}/100$ as claimed.
\endclaimproof{}
\begin{claim}\label{d34}
For any $(d_{1},d_{2})\in D_{1,2}^{\text{good}}$ there is a set~$D_{3,4}^{\text{good}}=D_{3,4}^{\text{good}}(d_{1},d_{2})$ of pairs $(d_{3},d_{4})\in D_{3}\times D_{4}$ such that $|D_{3,4}^{\text{good}}|\geq k^{2}/100$ and each $(d_{3},d_{4})\in D_{3,4}^{\text{good}}$ satisfies the following,
where $u_{4}\coloneqq\neigh{-}{d_{3}}{y}$, $u'_{5}\coloneqq\neigh{+}{d_{4}}{u_{4}}$, $u_{3}\coloneqq\neigh{+}{d_{4}}{z}$, $u'_{4}\coloneqq\neigh{-}{d_{3}}{u_{3}}$, and $u_{1},u'_{1},u_{6}$, and~$u'_{8}$ are defined as in Claim~\ref{d12}.
\begin{itemize}
    \item [$(D_{3}1)$] There are at most~$10^{8}$ loops with colour~$d_{3}$ in~$H$;
    \item [$(D_{3}2)$] $u_{4}$ has at most~$300k^{2}/n$ out-neighbours in the set~$\neighset{+}{D_{6}}{z}$;
     \item [$(D_{3}3)$] there are at most~$k/100$ arcs~$e$ coloured~$d_{3}$ in~$H$ such that~$e$ is contained in a distinguishable $(y,z,\cP)$-bridge in~$H$; 
    \item [$(D_{4}1)$] there are at most~$10^{8}$ loops with colour~$d_{4}$ in~$H$;
    \item [$(D_{4}2)$] $u_{3}$ has at most~$300k^{2}/n$ in-neighbours in the set~$\neighset{+}{D_{5}}{y}$;
     \item [$(D_{4}3)$] there are at most~$k/100$ arcs~$e$ coloured~$d_{4}$ in~$H$ such that~$e$ is contained in a distinguishable $(y,z,\cP)$-bridge in~$H$; 
    \item [$(V_{3,4})$] $y,z,u_{1},u_{3},u_{4},u_{6}$ are distinct vertices, and $u'_{1},u'_{4},u'_{5},u'_{8}\notin\{y,z,u_{1},u_{3},u_{4},u_{6}\}$;
    \item [$(R_{3,4})$] there is no distinguishable $(y,z,\cP)$-bridge in~$H$ containing the arc~$u'_{4}u_{3}$ or the arc~$u_{4}u'_{5}$.
\end{itemize}
\end{claim}
The proof is similar to that of Claim~\ref{d12}, so we omit it.\COMMENT{Genuinely the only change is that we need to avoid the vertices already specified when we're sorting out~$(V_{3,4})$, but one does this with similar logic anyway.}
\begin{claim}\label{d56}
For any $(d_{1},d_{2})\in D_{1,2}^{\text{good}}$ and $(d_{3},d_{4})\in D_{3,4}^{\text{good}}(d_{1},d_{2})$, there is a set~$D_{5,6}^{\text{good}}$ (depending on $(d_{1},\dots,d_{4})$) of pairs $(d_{5},d_{6})\in D_{5}\times D_{6}$ such that $|D_{5,6}^{\text{good}}|\geq k^{2}/100$ and each $(d_{5},d_{6})\in D_{5,6}^{\text{good}}$ satisfies the following,
where $u_{2}\coloneqq\neigh{+}{d_{5}}{y}$, $u'_{2}\coloneqq\neigh{+}{d_{2}}{u_{2}}$, $u'_{3}\coloneqq\neigh{+}{d_{3}}{u_{2}}$, $u_{5}\coloneqq\neigh{+}{d_{6}}{z}$, $u'_{6}\coloneqq\neigh{-}{d_{4}}{u_{5}}$, $u'_{7}\coloneqq\neigh{-}{d_{1}}{u_{5}}$, and $u_{1},u_{3},u_{4},u_{6},u'_{1},u'_{4},u'_{5},u'_{8}$ are defined as in Claims~\ref{d12} and~\ref{d34}.
\begin{itemize}
    \item [$(D_{5}1)$] $u_{2}u_{1}, u_{2}u_{3}\notin E(H)$;
    \item [$(D_{6}1)$] $u_{4}u_{5}, u_{6}u_{5}\notin E(H)$;
    \item [$(R_{5})$] there is no distinguishable $(y,z,\cP)$-bridge in~$H$ containing the arc~$u_{2}u'_{2}$ or the arc~$u_{2}u'_{3}$;
    \item [$(R_{6})$] there is no distinguishable $(y,z,\cP)$-bridge in~$H$ containing the arc~$u'_{6}u_{5}$ or the arc~$u'_{7}u_{5}$;
    \item [$(V_{5,6})$] $y,z,u_{1},u_{2},\dots,u_{6}$ are distinct, and $u'_{1},u'_{2},\dots,u'_{8}\notin\{y,z,u_{1},u_{2},\dots,u_{6}\}$;
    \item [$(A_{5,6}1)$] for each $(d'_{1},d'_{2})\in F_{1,2}(d_{5})$, where~$F_{1,2}(d_{5})$ is the set of pairs $(d'_{1},d'_{2})\in D_{1}\times D_{2}$ such that $\neigh{+}{d'_{1}}{y}=\neigh{+}{d'_{2}}{u_{2}}$, we have $\neigh{-}{d'_{1}}{u_{5}}\neq\neigh{-}{d'_{2}}{z}$;
    \item [$(A_{5,6}2)$] for each $(d'_{3},d'_{4})\in F_{3,4}(d_{5})$, where~$F_{3,4}(d_{5})$ is the set of pairs $(d'_{3},d'_{4})\in D_{3}\times D_{4}$ such that $\neigh{+}{d'_{4}}{z}=\neigh{+}{d'_{3}}{u_{2}}$, we have $\neigh{-}{d'_{4}}{u_{5}}\neq\neigh{-}{d'_{3}}{y}$.
\end{itemize}
\end{claim}
\claimproof{}
Let~$\widetilde{D}_{5}$ be the set of colours $d_{5}\in D_{5}$ which fail to satisfy~$(D_{5}1)$ and~$(R_{5})$, let~$\widetilde{D}_{6}$ be the set of colours $d_{6}\in D_{6}$ which fail to satisfy~$(D_{6}1)$ and~$(R_{6})$, and define $D'_{5}\coloneqq D_{5}\setminus \widetilde{D}_{5}$ and $D'_{6}\coloneqq D_{6}\setminus\widetilde{D}_{6}$.
By~$(D_{1}2)$ and~$(D_{4}2)$, there are at most~$600k^{2}/n$ colours $d_{5}\in D_{5}$ which fail to satisfy~$(D_{5}1)$.
By~$(D_{2}3)$ and~$(D_{3}3)$, at most~$k/50$ choices of $d_{5}\in D_{5}$ give~$u_{2}$ to be the tail of a $d_{2}$-arc or $d_{3}$-arc contained in a distinguishable $(y,z,\cP)$-bridge in~$H$, and all other choices of $d_{5}\in D_{5}$ satisfy~$(R_{5})$.
Using~$(D_{2}2)$,~$(D_{3}2)$,~$(D_{1}3)$, and~$(D_{4}3)$ similarly, we deduce that $|D'_{5}|,|D'_{6}|\geq 7k/50$.
By~$(D_{2}1)$ and~$(D_{3}1)$, for any $d_{6}\in D'_{6}$ there are at most $2\cdot 10^{8}+10$ choices of $d_{5}\in D'_{5}$ such that $u_{5}\in\{y,z,u_{1},u_{3},u_{4},u_{6},u'_{1},u'_{4},u'_{5},u'_{8}\}$ or~$u_{5}$ is incident to a loop of colour~$d_{2}$ or~$d_{3}$.
Further, for any $d_{6}\in D'_{6}$, there are at most~$12$ choices of~$d_{5}$ such that $u_{2}\in\{\neigh{-}{d}{w}\colon d\in\{d_{2},d_{3}\},w\in\{y,z,u_{1},u_{3},u_{4},u_{6}\}\}$.
Thus there are at most $(2\cdot 10^{8}+22)|D_{6}|$ choices of pair $(d_{5},d_{6})\in D'_{5}\times D'_{6}$ such that the choice of~$d_{5}$ causes~$(V_{5,6})$ to fail.
Using~$(D_{1}1)$ and~$(D_{4}1)$ to address the choice of~$d_{6}$ similarly, we conclude that we can remove a set of at most~$(5\cdot 10^{8})k/6$ pairs $(d_{5},d_{6})\in D'_{5}\times D'_{6}$ such that~$(V_{5,6})$ holds for all remaining pairs\COMMENT{If we need to save space, we could put ($(V_{5,6}$) in a comment and say it's the same as claim 1}.
%By~$(D_{1}2)$ and~$(D_{4}2)$ there are at most~$600\frac{k^{2}}{n}\cdot\frac{k}{6}$ pairs $(d_{5},d_{6})\in D_{5}\times D_{6}$ such that~$(D_{5}1)$ does not hold, and we can use~$(D_{2}2)$ and~$(D_{3}2)$ to address~$(D_{6}1)$ analogously.
%By~$(D_{2}3)$ and~$(D_{3}3)$, at most~$k/50$ choices of $d_{5}\in D_{5}$ give~$u_{2}$ to be the tail of a $d_{2}$-arc or $d_{3}$-arc contained in a distinguishable $(y,z,\cP)$-bridge in~$H$, and all pairs~$(d_{5},d_{6})$ not using such a~$d_{5}$ satisfy~$(R_{5})$.
%We can use~$(D_{1}3)$ and~$(D_{4}3)$ to address~$(R_{6})$ analogously.
For~$(A_{5,6}1)$, notice that~(\ref{eq:quasieq}) implies $e_{H}(N_{D_{5}}^{+}(y), N_{D_{1}}^{+}(y))\leq3k^{3}/n$, so that there are at most~$k/100$ colours $d_{5}\in D'_{5}$ for which~$u_{2}$ has at least~$300k^{2}/n$ out-neighbours in the set~$N_{D_{1}}^{+}(y)$.
Deleting all pairs~$(d_{5},d_{6})$ using such a~$d_{5}$, we have in particular that all remaining pairs satisfy $|F_{1,2}(d_{5})|\leq300k^{2}/n$.
For a remaining pair~$(d_{5},d_{6})$ and each $(d'_{1},d'_{2})\in F_{1,2}(d_{5})$ we define the vertex $v_{d'_{1},d'_{2}}\coloneqq\neigh{+}{d'_{1}}{\neigh{-}{d'_{2}}{z}}$.
Now further deleting all (at most~$\frac{k}{6}\cdot300\frac{k^{2}}{n}$) pairs~$(d_{5},d_{6})$ such that $u_{5}\in\{v_{d'_{1},d'_{2}}\colon (d'_{1},d'_{2})\in F_{1,2}(d_{5})\}$, all remaining pairs satisfy~$(A_{5,6}1)$.
We address~$(A_{5,6}2)$ similarly.
In total, the number of pairs in~$D'_{5}\times D'_{6}$ satisfying $(V_{5,6})$,~$(A_{5,6}1)$, and~$(A_{5,6}2)$ is at least $(7k/50)^{2}-5\cdot10^{8}k/6-2(k^{2}/600+50k^{3}/n)\geq k^{2}/100$, finishing the proof of the claim.
%\textcolor{mygreen}{(Reminder to self to come back here and add in the definitions of~$D'_{5}$ and~$D'_{6}$ so that the pairs in $D'_{5}\times D'_{6}$ satisfy the first three properties we want.)}
\endclaimproof{}
\begin{claim}\label{claime}
For any $(d_{1},d_{2})\in D_{1,2}^{\text{good}}$, $(d_{3},d_{4})\in D_{3,4}^{\text{good}}$, $(d_{5},d_{6})\in D_{5,6}^{\text{good}}$, there is a set $E^{\text{good}}=E^{\text{good}}(d_{1},d_{2},\dots,d_{6})\subseteq\prod_{i\in[4]}E_{d_{i}}(H)$ such that $|E^{\text{good}}|\geq9n^{4}/10$ and each $(e_{1},\dots,e_{4})\in E^{\text{good}}$ (with $e_{i}\in E_{d_{i}}(H)$ for each $i\in[4]$) satisfies the following,
%are sets $E_{i}^{\text{good}}\subseteq E_{d_{i}}(H)$ for $i\in[4]$ such that $|E_{i}^{\text{good}}|\geq9n/10$ for each $i\in [4]$, and each $(e_{1},\dots,e_{4})\in E_{1}^{\text{good}}\times\dots\times E_{4}^{\text{good}}$ satisfies the following;
where $u''_{1}\coloneqq\vhead{e_{2}}$, $u''_{2}\coloneqq\vtail{e_{2}}$, $u''_{3}\coloneqq\vtail{e_{3}}$, $u''_{4}\coloneqq\vhead{e_{3}}$, $u''_{5}\coloneqq\vtail{e_{4}}$, $u''_{6}\coloneqq\vhead{e_{4}}$, $u''_{7}\coloneqq\vhead{e_{1}}$, $u''_{8}\coloneqq\vtail{e_{1}}$, and $u_{1},\dots,u_{6},u'_{1},\dots,u'_{8}$ are defined as in Claims~\ref{d12}--\ref{d56}.
\begin{itemize}
    \item [$(V_{E})$] $u''_{1},u''_{2},\dots,u''_{8}$ are distinct vertices and $u''_{1},\dots,u''_{8}\notin\{y,z,u_{1},\dots,u_{6},u'_{1},\dots,u'_{8}\}$;
    \item [$(E1)$] $u'_{1}u''_{1},u''_{2}u'_{2},u''_{3}u'_{3},u'_{4}u''_{4},u''_{5}u'_{5},u'_{6}u''_{6},u'_{7}u''_{7},u''_{8}u'_{8}\notin E(H)$;
    \item [$(E2)$] $\{u''_{1},u''_{2},\dots,u''_{8}\}\cap(N(y)\cup N(z))=\emptyset$;
    \item [$(R_{E})$] there is no distinguishable $(y,z,\cP)$-bridge in~$H$ containing any of the arcs $e_{1},\dots,e_{4}$.
\end{itemize}
\end{claim}
\claimproof{}
At most~$32$ $d_{2}$-arcs of~$H$ have head or tail in~$\{y,z,u_{1},\dots,u_{6},u'_{1},\dots,u'_{8}\}$, and by~$(D_{2}1)$, at most~$10^{8}$ arcs in~$E_{d_{2}}(H)$ are loops.
Choosing any other $d_{2}$-arc to be~$e_{2}$, and proceeding similarly for $e_{3},e_{4},e_{1}$ (also avoiding the vertices of previously chosen such arcs), we deduce that we can delete at most\COMMENT{$4(10^8 +32)n^3 + 4n^3 + 8n^3 + 12n^3 \leq$}~$4\cdot10^{9}n^{3}$ tuples from~$\prod_{i=1}^{4}E_{d_{i}}(H)$ so that the remaining tuples satisfy~$(V_{E})$. 
Notice that at most~$2k$ choices of $e_{2}\in E_{d_{2}}(H)$ have head in~$N^{+}(u'_{1})$ or tail in~$N^{-}(u'_{2})$, and any other~$e_{2}$ satisfies $u'_{1}u''_{1},u''_{2}u'_{2}\notin E(H)$.
Dealing with~$e_{3},e_{4},e_{1}$ similarly addresses~$(E1)$.
Similarly, for~$(E2)$ it suffices to notice that at most~$8k$ choices of~$e_{2}$ (for example) have either head or tail in~$N(y)\cup N(z)$.
Finally, by~$(D_{i}3)$ for each $i\in[4]$, at most~$k/100$ choices of each~$e_{i}$ fail to satisfy~$(R_{E})$.
In total, at least $n^4 - 4n^{3}(10^9 - 2k - 8k - k/100) \geq 9n^4 / 10$ tuples in $\prod_{i=1}^4 E_{d_i}(H)$ satisfy~$(V_{E})$, $(E1)$, $(E2)$, and $(R_{E})$, as desired.
%\textcolor{mygreen}{(Reminder to come back here and do the change about the E's depending on each other.)}
\endclaimproof{}
Let~$\Lambda$ be the set of tuples $\lambda=(d_{1},\dots,d_{6},e_{1},\dots,e_{4})$ satisfying all properties in Claims~\ref{d12}--\ref{claime}.
For each $\lambda\in\Lambda$ we define a subgraph $T_{\lambda}\subseteq H$ with the vertices $V(T_{\lambda})=\{y, z, u_{1}, \dots, u_{6}\}\cup\{u'_{1}, \dots, u'_{8}, u''_{1}, \dots, u''_{8}\}$ as defined in Claims~\ref{d12}--\ref{claime} by the choice of~$\lambda$, and whose arcs are as in Definition~\ref{def:twistsystem}(i).
(These arcs each exist in~$H$ by the way the vertices of~$T_{\lambda}$ were defined.)
Then since $d_{i}\in D_{i}$ for each $i\in[6]$ and since~$(D_{5}1)$,~$(D_{6}1)$, and~$(E1)$ hold for each $\lambda\in\Lambda$, we have that each condition of Definition~\ref{def:twistsystem} is satisfied, so that~$T_{\lambda}$ is a twist system of~$H$ for each $\lambda\in\Lambda$.
Further, clearly the~$T_{\lambda}$ are distinct.
Define $\cT\coloneqq\{T_{\lambda}\colon\lambda\in\Lambda\}$, and
notice that $|\cT|\geq\left(k^{2}/100\right)^{3}\cdot9n^{4}/10=9k^{6}n^{4}/10^{7}$.
%\textcolor{mygreen}{(Number crunching to do here too)}
\begin{claim}\label{claim:unique}
For any $T_{\lambda}\in\cT$, the only $(y,z,\cP)$-bridge in~$\text{twist}_{T_{\lambda}}(H)$ that is not in~$H$ is the canonical $(y,z,\cP)$-bridge of the twist.
\end{claim}
\claimproof{}
Fix $T_{\lambda}\in\cT$ (fixing the notation of all the vertices and arcs as above), and let~$B$ be the canonical $(y,z,\cP)$-bridge of the twist (which has vertices $V(B)=\{y,z,u_{1},\dots,u_{6}\}$ and colours $d_{1},\dots,d_{6}$).
Suppose that~$B'$ is a $(y,z,\cP)$-bridge in~$\text{twist}_{T_{\lambda}}(H)$ that is not in~$H$, and label the vertices of~$B'$ as $V(B')=\{y,z,v_{1},\dots,v_{6}\}$ (where the role of~$v_{i}$ in~$B'$ corresponds to that of~$u_{i}$ in~$B$), and label the colours of~$B'$ as $d'_{i}\in D_{i}$, for $i\in[6]$.
By~$(V_{1,2})$,~$(V_{3,4})$, and~$(V_{5,6})$, all arcs we add when producing~$\text{twist}_{T_{\lambda}}(H)$ from~$H$ do not have~$y$ or~$z$ as an endpoint, and thus one (or more) of the arcs $v_{2}v_{1}$, $v_{2}v_{3}$, $v_{4}v_{5}$, $v_{6}v_{5}$ is added by the twist operation.
Further, due to the colour partition~$\cP$, $v_{2}v_{1}$ must either be in~$H$, or be one of the added arcs $u_{2}u_{1}$, $u'_{1}u''_{1}$, $u''_{2}u'_{2}$ with colour $d'_{2}\in D_{2}$.
But since we do not add any arcs incident to~$y$, by~$(E2)$,~$u''_{1}$ and~$u''_{2}$ are not in the neighbourhood of~$y$ in~$\text{twist}_{T_{\lambda}}(H)$, whence~$u''_{1}$ cannot be~$v_{1}$, and~$u''_{2}$ cannot be~$v_{2}$.
Thus~$v_{2}v_{1}$ must either be in~$H$, or be~$u_{2}u_{1}$.
Similarly~$v_{2}v_{3}$, $v_{4}v_{5}$, $v_{6}v_{5}$ must be $u_{2}u_{3}$, $u_{4}u_{5}$, $u_{6}u_{5}$ respectively, or be in~$H$, in some combination.
We now split the analysis into cases, depending on how many arcs in~$F\coloneqq\{v_{2}v_{1}, v_{2}v_{3}, v_{4}v_{5}, v_{6}v_{5}\}$ are in~$H$.
In each case we show either that that case does not occur or that $B'=B$, which will complete the proof of the claim.
Since $B'\nsubseteq H$, at most three arcs in~$F$ are in~$H$.
\newline\noindent\underline{Case 1}: Precisely three arcs in~$F$ are in~$H$.
\newline\noindent %In this case, precisely one of the arcs in~$F$ has the same role in~$B$ as it has in~$B'$.
Let~$e$ be the arc in~$F$ that is not in~$H$.
Suppose $e=v_{2}v_{1}$, which implies that $e=u_{2}u_{1}$, and $d'_{2}=d_{2}$.
Since $\neigh{+}{d'_{1}}{y}=v_{1}=u_{1}=\neigh{+}{d_{1}}{y}$, we have $d'_{1}=d_{1}$.
Then $v_{6}=\neigh{-}{d'_{2}}{z}=\neigh{-}{d_{2}}{z}=u_{6}$, and~$v_{5}$ is the $d'_{1}$-out-neighbour of~$v_{6}$ in~$\text{twist}_{T_{\lambda}}(H)$, which is the $d_{1}$-out-neighbour of~$u_{6}$ in~$\text{twist}_{T_{\lambda}}(H)$, namely~$u_{5}$.
This is a contradiction, since~$v_{6}v_{5}$ is in~$H$, but~$u_{6}u_{5}$ is not.
%Suppose that this arc is $v_{2}v_{1}=u_{2}u_{1}$, so that in particular, $d'_{2}=d_{2}$, $v_{1}=u_{1}$, whence $yv_{1}=yu_{1}$, so that $d'_{1}=d_{1}$.
%But then~$B'$ must also have the arcs $v_{6}z=u_{6}z$, and $v_{6}v_{5}=u_{6}u_{5}$; a contradiction since~$u_{6}u_{5}$ is not present in~$H$.
One similarly obtains a contradiction if~$e$ is $v_{2}v_{3}$, $v_{4}v_{5}$, or $v_{6}v_{5}$, so we deduce that this case does not occur.
\newline\noindent\underline{Case 2}: Precisely two arcs in~$F$ are in~$H$.
\newline\noindent Suppose that $v_{2}v_{1}=u_{2}u_{1}$ and $v_{6}v_{5}=u_{6}u_{5}$.
Then $d'_{1}=d_{1}$, $d'_{2}=d_{2}$, and $v_{2}v_{3}$, $v_{4}v_{5}$ are in~$H$.
Hence $v_{3}\neq u_{3}$, and the arc $u_{2}v_{3}$ is in~$H$ with colour~$d'_{3}$.
Moreover the arc~$zv_{3}$ is in~$H$ with colour~$d'_{4}$.
In particular, $(d'_{3},d'_{4})\in F_{3,4}(d_{5})$.
Similarly $v_{4}\neq u_{4}$, we have~$v_{4}y\in E(H)$ has colour~$d'_{3}$, and $v_{4}u_{5}\in E(H)$ has colour~$d'_{4}$.
That is, $\neigh{-}{d'_{4}}{u_{5}}=v_{4}=\neigh{-}{d'_{3}}{y}$, which contradicts~$(A_{5,6}2)$ of Claim~\ref{d56}.
Similarly one can use~$(A_{5,6}1)$ to show that assuming $v_{2}v_{3}=u_{2}u_{3}$ and $v_{4}v_{5}=u_{4}u_{5}$ yields a contradiction.
%Since~$B'$ is a $(y,z,\cP)$-bridge which cannot be~$B$ in this case, we must have $v_{3}\neq u_{3}$ (as otherwise we could determine that $d'_{3}=d_{3}$, $d'_{4}=d_{4}$ so that $v_{4}=u_{4}$ and $B'=B$), and that
Suppose instead that $v_{2}v_{1}=u_{2}u_{1}$ and $v_{2}v_{3}=u_{2}u_{3}$.
Then $d'_{2}=d_{2}$, and $v_{1}=u_{1}$ so that $d'_{1}=d_{1}$.
Further, $d'_{3}=d_{3}$, and $v_{3}=u_{3}$ so that $d'_{4}=d_{4}$.
But this now also determines that $v_{4}v_{5}=u_{4}u_{5}$ and $v_{6}v_{5}=u_{6}u_{5}$, a contradiction since then no arcs in~$F$ are in~$H$.
All remaining possibilities yield a contradiction similarly, whence this case does not occur.
\newline\noindent\underline{Case 3}: Precisely one arc in~$F$ is in~$H$.
\newline\noindent In particular either we have $v_{2}v_{1}=u_{2}u_{1}$ and $v_{2}v_{3}=u_{2}u_{3}$ or we have $v_{4}v_{5}=u_{4}u_{5}$ and $v_{6}v_{5}=u_{6}u_{5}$.
But as in Case 2, either way yields that no arcs in~$F$ are in~$H$.
We deduce that this case does not occur.
\newline\noindent\underline{Case 4}: No arcs in~$F$ are in~$H$.
\newline\noindent It is easy to see in this case that $v_{i}=u_{i}$ for all $i\in[6]$ whence $B'=B$.
\endclaimproof{}
For any $T_{\lambda}\in\cT$, we have by Claim~\ref{claim:unique} that the canonical $(y,z,\cP)$-bridge of the twist is distinguishable in~$\text{twist}_{T_{\lambda}}(H)$, since the four arcs $u_{2}u_{1}$, $u_{2}u_{3}$, $u_{4}u_{5}$, $u_{6}u_{5}$ are each added by the twist, and do not create any other $(y,z,\cP)$-bridge.
Further, since the twist operation only adds~$12$ arcs, we obtain from~(\ref{eq:quasieq}) that $e_{\text{twist}_{T_{\lambda}}(H)}(W_{1},W_{2})\leq 2k^{3}/n +12(s+1)$ for all $W_{1},W_{2}\subseteq[n]$ of sizes $|W_{1}|=|W_{2}|=k$.
Thus by~$(R_{1,2})$, $(R_{3,4})$,~$(R_{5})$,~$(R_{6})$, $(R_{E})$, and Claim~\ref{claim:unique}, we have that $\text{twist}_{T_{\lambda}}(H)\in M_{s+1}$.
We now give a final claim which ensures that $|\{\text{twist}_{T_{\lambda}}(H)\colon T_{\lambda}\in\cT\}|=|\cT|$.
\begin{claim}\label{claim:H'}
Fix $H'\in M_{s+1}$, let $\lambda, \lambda'\in\Lambda$, and suppose that $\text{twist}_{T_{\lambda}}(H)=\text{twist}_{T_{\lambda'}}(H)=H'$.
Then $\lambda=\lambda'$.
\end{claim}
\claimproof{}
Let $\lambda=(d_{1},\dots,d_{6},e_{1},\dots,e_{4})$, $\lambda'=(d'_{1},\dots,d'_{6},e'_{1},\dots,e'_{4})$, with corresponding vertices $V(T_{\lambda})=\{y,z,u_{1},\dots,u_{6}\}\cup\{u'_{1},\dots,u'_{8},u''_{1},\dots,u''_{8}\}$, $V(T_{\lambda'})=\{y,z,v_{1},\dots,v_{6}\}\cup\{v'_{1},\dots,v'_{8},v''_{1},\dots,v''_{8}\}$.
Let~$B$ and~$B'$ be the canonical $(y,z,\cP)$-bridges of the twists corresponding to~$\lambda$ and~$\lambda'$, respectively.
By Claim~\ref{claim:unique},~$B$ and~$B'$ are each the unique $(y,z,\cP)$-bridge that is in~$H'$ but not in~$H$, and thus $B=B'$.
In particular, $d_{i}=d'_{i}$ and $u_{i}=v_{i}$ for all $i\in [6]$.
%and it follows that $u'_{i}=v'_{i}$ for all $i\in[8]$.
By considering the partition~$\cP$ of~$D$, the four arcs in~$E(H)\setminus E(H')$ with no endvertex in~$\{u_{1},\dots,u_{6}\}$ must be $e_{1}=e'_{1}$, $e_{2}=e'_{2}$, $e_{3}=e'_{3}$, and $e_{4}=e'_{4}$.
We conclude that $\lambda=\lambda'$, as required.
\endclaimproof{}
%Note that for any $H'\in M_{s+1}$ there are at most~$6^{4}$ distinct $T_{\lambda}\in\cT$ such that $\text{twist}_{T_{\lambda}}(H)=H'$ (one can compare~$H$ and~$H'$, use~$\cP$ to determine which three of the new arcs have colour in~$D_{1}$ for example, then there are~$3!$ ways to choose the roles of these three arcs in the twist system). \textcolor{mygreen}{Note to self: come back to here and deal with this.}
We determine that if $s\leq k^{4}/(10^{24}n^{2})$ and~$M_{s}$ is non-empty, then~$M_{s+1}$ is non-empty, and $\delta^{+}_{s}\geq|\cT|\geq 9k^{6}n^{4}/10^{7}$, whence $|M_{s}|/|M_{s+1}|\leq \Delta^{-}_{s+1}/\delta^{+}_{s}\leq 1/10$.
Recalling that $\mathbf{H}\in\cG_{D}$ is uniformly random, it follows that if $s\leq k^{4}/(10^{25}n^{2})$ then
\begin{eqnarray*}
\prob{r(\mathbf{H})=s\,\big|\, \mathbf{H}\in\widehat{\cQ}_{D}} & \leq & \frac{|M_{s}|}{|M_{k^{4}/(10^{24}n^{2})}|}=\prod_{t=s}^{k^{4}/(10^{24}n^{2})-1}\frac{|M_{t}|}{|M_{t+1}|}\leq\left(\frac{1}{10}\right)^{k^{4}/(10^{24}n^{2})-s} \\ & \leq & \exp\left(-\frac{9k^{4}\log 10}{10^{25}n^{2}}\right).
\end{eqnarray*}
Moreover, we note that if~$M_{s}$ is empty, then clearly $\prob{r(\mathbf{H})=s\mid \mathbf{H}\in\widehat{\cQ}_{D}}=0$.
Since $k=n/10^{6}$, we obtain that $\prob{r(\mathbf{H})\leq k^{4}/(10^{25}n^{2})\mid\mathbf{H}\in\widehat{\cQ}_{D}}\leq(k^{4}/(10^{25}n^{2})+1)\exp(-\Omega(n^{2}))=\exp(-\Omega(n^{2}))$.
Since $n^{2}/10^{50}\leq k^{4}/(10^{25}n^{2})$, the result follows.
\end{proof}
%
%\textcolor{purple}{The following lemma shows that collections of distinguishable bridging gadgets have the desired well-spread property.}
%\begin{lemma}\label{justgadgets2}
%Suppose that $1/n\ll\mu$, and let $D\subseteq[n]$ be such that $|D|\leq\mu n$.
%Let $y,z\in V$ be distinct, and let $\cP=\{D_{i}\}_{i=1}^{6}$ be an equitable partition of~$D$.
%Then for any integer $t\geq 0$ and any $G\in\cG_{D}$, if $r(G)\geq t$, then~$G$ contains a $2\mu n$-well-spread collection of~$t$ distinct $(y,z)$-bridging gadgets.
%\end{lemma}
%\begin{proof}
%By definition of~$r(G)$ this means a big collection such that every $D_{1}$--$D_{4}$-edge not incident to $y$ nor $z$ is in at most one gadget in the collection.
%
%This immediately yields the colour property of well-spreadness.
%
%Now fix a vertex that isn't~$y$ or~$z$, and if it's in a gadget then it must be incident to one of the distinguishable edges. Now there's at most one gadget in our collection on each of the $D_{1}$--$D_{4}$ edges incident to this vertex not incident to~$y$ or~$z$. Boom.
%
%Now just ignore the auxiliary $D_{5}$ and $D_{6}$ edges.
%\end{proof}
%\textcolor{red}{The proof of this one is obviously similar to the corresponding one from the Andersen paper. There's no need for saturation any more though, so it's plausible we could squeeze it out in a few lines. Do we just give the proof here? Do we even mention there's a similar one in the RHP paper?}
%
We are now ready to argue that almost all $G\in\Phi(\dirK)$ contain large well-spread collections of bridging gadgets, which will complete our study of the properties we need to be satisfied by uniformly random $\mathbf{G}\in\Phi(\dirK)$.
\begin{lemma}\label{main-switching-lemma}
Let $\mathbf{G}\in\Phi(\dirK)$ be chosen uniformly at random, and let~$\cE$ be the event that for all distinct $y,z\in[n]$,~$\mathbf{G}$ contains a well-spread collection of at least~$n^{2}/10^{50}$ distinct $(y,z)$-bridging gadgets.
Then $\prob{\cE}\geq 1-\exp(-\Omega(n^{2}))$.
\end{lemma}
\begin{proof}
For $D\subseteq[n]$, let~$\cE|_{D}$ denote the event (in~$\Phi(\dirK)$) that~$\mathbf{G}|_{D}$ contains a well-spread collection of~$n^{2}/10^{50}$ distinct $(y,z)$-bridging gadgets, for each distinct $y,z,\in[n]$.
Let $\mathbf{H}\in\cG_{D}$ be chosen uniformly at random, let~$\pr_{D}$ denote the measure for this probability space, and let~$\cE_{D}^{(y,z)}$ denote the event (in~$\cG_{D}$) that~$\mathbf{H}$ contains a well-spread collection of~$n^{2}/10^{50}$ distinct $(y,z)$-bridging gadgets, and define $\cE_{D}\coloneqq\bigcap_{y,z\in[n]\,\text{distinct}}\cE_{D}^{(y,z)}$. %\textcolor{mygreen}{Not quite the same as $(y,z)\in\binom{[n]}{2}$ as in truth elements of that set are unordered pairs, and here I suppose we want all ordered pairs of distinct elements.}
\begin{claim}\label{claim:spreadgadgets}
Suppose $D\subseteq [n]$ has size $|D|\leq 3n/4$,
let~$\cP=(D_{i})_{i\in[6]}$ be an equitable partition of~$D$ into six parts, and fix $y,z,\in[n]$ distinct.
Then $\probd{r_{(y,z,\cP)}(\mathbf{H})\geq n^{2}/10^{50}}\leq\probd{\cE_{D}^{(y,z)}}$.
\end{claim}
%\textcolor{mygreen}{$\pr_{D}$ has been defined but all the way back in the notation section.}
\claimproof{}
Suppose that $H\in\cG_{D}$ and that $r(H)\geq n^{2}/10^{50}$.
By definition of~$r(H)$ (see Definition~\ref{def:bridgestuff}), there is a collection~$\cB$ of~$n^{2}/10^{50}$ distinct $(y,z,\cP)$-bridges such that for each $i\in[4]$, any $d_{i}\in D_{i}$, and any arc $e\in E_{d_{i}}(H)$ for which~$e$ does not have~$y$ nor~$z$ as an endvertex, we have that~$e$ is contained in at most one $B\in\cB$.
Note that for each $u\in[n]\setminus\{y,z\}$, we have for any $B\in\cB$ which contains~$u$, that~$B$ must contain an arc~$e$ incident to~$u$ with colour in~$D_{i}$ for some $i\in[4]$ such that~$e$ is not incident to~$y$ nor~$z$.
Therefore~$u$ is contained in at most $4\cdot2\cdot |D|/6\leq n$ distinct $B\in\cB$.\COMMENT{$4$ for $D_{1}, D_{2}, D_{3}, D_{4}$, $2$ for the fact~$u$ has an incident in-arc and out-arc of each colour,~$k/6$ for the size of~$D_{i}$.}
For any colour $d\in D_{1}$, any $B\in\cB$ which uses the colour~$d$ must be such that $\neigh{+}{d}{y}\notin\{y,z\}$ and must contain the vertex~$\neigh{+}{d}{y}$, and thus the colour~$d$ is used by at most~$n$ distinct $B\in\cB$ (and similarly for $d\in D_{i}$ for all $i\in[6]$).
Now, forming a collection~$\cB'$ of $(y,z)$-bridging gadgets in~$H$ by deleting the arcs with colours in $D_{5}\cup D_{6}$ for each $B\in\cB$, it is clear that~$\cB'$ witnesses that $H\in\cE_{D}^{(y,z)}$.
The claim follows.
\endclaimproof{}
%\textcolor{purple}{Argue that distinguishable special gadgets give us many well-spread bridging-gadgets, roughly as follows:
%By definition of~$r(G)$ this means a big collection such that every $D_{1}$--$D_{4}$-edge not incident to $y$ nor $z$ is in at most one gadget in the collection.
%This immediately yields the colour property of well-spreadness.
%Now fix a vertex that isn't~$y$ or~$z$, and if it's in a gadget then it must be incident to one of the distinguishable edges. Now there's at most one gadget in our collection on each of the $D_{1}$--$D_{4}$ edges incident to this vertex not incident to~$y$ or~$z$. Boom.
%Now just ignore the auxiliary $D_{5}$ and $D_{6}$ edges.}
%
Arbitrarily fix $c\in[n]$ and $D\subseteq[n]$ of size $|D|=n/10^{6}$, and let~$\cF$ be the set of all possible colour classes for a proper $n$-arc colouring of~$\dirK$ (more precisely,~$\cF$ is the collection of all sets~$F$ of~$n$ arcs of~$\dirK$ such that every vertex of~$\dirK$ is the head of precisely one arc in~$F$ and the tail of precisely one arc in~$F$).
Observe that for a fixed equitable partition~$\cP=(D_{i})_{i=1}^{6}$ of~$D$ into six parts, and for fixed distinct $y,z\in[n]$, by~(\ref{eq:qhat}), the law of total probability, and Lemma~\ref{quasi-lemma},
\begin{equation}\label{eq:notqhat}
\probd{\overline{\widehat{\cQ}_{D}}}\leq\probd{\overline{\cQ_{D}^{1}}}=\sum_{F\in\cF}\probd{\mathbf{F}_{c}=F}\probd{\overline{\cQ_{D}^{1}}\mid\mathbf{F}_{c}=F}\leq\exp(-\Omega(n^{2})).
\end{equation}
Then the law of total probability,~(\ref{eq:notqhat}), and Lemma~\ref{masterswitch2} give
\begin{equation}\label{eq:sortingr}
\probd{r_{(y,z,\cP)}(\mathbf{H})\leq\frac{n^{2}}{10^{50}}}\leq \probd{r_{(y,z,\cP)}(\mathbf{H})\leq\frac{n^{2}}{10^{50}}\,\bigg|\,\widehat{\cQ}_{D}}+\probd{\overline{\widehat{\cQ}_{D}}}\leq\exp(-\Omega(n^{2})).
\end{equation}
By~(\ref{eq:sortingr}) and Claim~\ref{claim:spreadgadgets}, $\probd{\overline{\cE_{D}^{(y,z)}}}\leq\exp(-\Omega(n^{2}))$, so by a union bound we have $\probd{\overline{\cE_{D}}}\leq\exp(-\Omega(n^{2}))$.
%For a coloured digraph $H\in\cG_{D}$, we define~$\text{comp}(H)$ to be the number of distinct ways to complete~$H$ to an element $G\in\Phi(\dirK)$, or more precisely the number of $H'\in\cG_{[n]\setminus D}$ having $E(H)\cap E(H')=\emptyset$ (and therefore $E(H)\cup E(H')=E(\dirK)$).
Then by Proposition~\ref{prop:wf}, we have
\[
\prob{\overline{\cE}}\leq\prob{\overline{\cE|_{D}}}=\frac{\sum_{H\in\overline{\cE_{D}}}\text{comp}(H)}{\sum_{H\in\cG_{D}}\text{comp}(H)}\leq\probd{\overline{\cE_{D}}}\cdot\exp(O(n\log^{2}n))\leq\exp(-\Omega(n^{2})),
\]
which completes the proof of the lemma.
\end{proof}
%
%\textcolor{blue}{I've again prioritized the statement being neat over being as strong as we can achieve, since all we can achieve is the correct order but missing by probably a huge factor. Proof is of course similar to the corresponding one in the RHP paper but feels too important to not give, and we could probably do so without even mentioning the RHP paper?}
%

\section{Absorption}\label{section:absorption}
The aim of this section is to show that if $G\in\Phi(\dirK)$ satisfies the conclusions of Lemmas~\ref{masterswitch1} and~\ref{main-switching-lemma}, then~$G$ admits a small robustly rainbow-Hamiltonian subdigraph~$H$ (recall Definition~\ref{def:rrh}), with arbitrarily chosen flexible sets of appropriate size.  In Section~\ref{sec:proof},~$H$ will form the key `absorbing structure'.

\begin{defin}\label{defn:(v,c)-absorber}
Let~$A$ be a $(v,c)$-absorbing gadget with abutment vertices $(x_4,  x_5)$, and let~$B$ be a $(y, z)$-bridging gadget (with all vertices retaining their notation as defined in Definitions~\ref{def:absgadg} and~\ref{def:bridgegadg}).  If $(y, z) = (x_4, x_5)$, $V(A)\cap V(B)=\{y, z\}$, and $\phi_{A}(A)\cap\phi_{B}(B)=\emptyset$, then we say $B$ \textit{bridges} $A$.
In this case, we also say a collection~$(P_{1},P_{2},P_{3},P_{4})$ of  properly arc-coloured directed paths \textit{completes} the pair~$(A,B)$ if:
\begin{itemize}
    \item $P_{1}$ has tail~$x_{2}$ and head~$x_{3}$;
    \item $P_{2}$ has tail~$w_{1}$ and head~$w_{4}$;
    \item $P_{3}$ has tail~$w_{5}$ and head~$w_{2}$;
    \item $P_{4}$ has tail~$w_{3}$ and head~$w_{6}$;
    \item $P_{1},\dots,P_{4}$ are mutually vertex-disjoint and the internal vertices of~$P_{1},\dots,P_{4}$ are disjoint from $V(A)\cup V(B)$;
    \item $\bigcup_{i=1}^{4}P_{i}$ is rainbow and shares no colour with {$A\cup B$}. 
\end{itemize}
In this case, we say that $A^* \coloneqq A\cup B\cup\bigcup_{i=1}^{4}P_{i}$ is a $(v,c)$\textit{-absorber} (see Figure~\ref{fig:vcabsorber}), and we also define the following.  
\begin{itemize}
    \item The \textit{initial vertex} of $A^*$ is $x_1$, and the \textit{terminal vertex} of $A^*$ is $x_6$.
    \item The \textit{$(v, c)$-absorbing path} in $A^*$ is the directed path with arc set $\{\myArc{x_1}{v}, \myArc{v}{x_2}, \myArc{x_3}{x_4}, \myArc{x_5}{x_6}\}\cup\bigcup_{i=1}^4E(P_i)$.
    \item The \textit{$(v, c)$-avoiding path} in $A^*$ is the directed path with arc set $\{\myArc{x_1}{x_2}, \myArc{x_3}{x_5}, \myArc{x_4}{x_6}\}\cup\bigcup_{i=1}^4E(P_i)$.
\end{itemize}
\end{defin}
Observe that the $(v,c)$-absorbing path and $(v,c)$-avoiding path of a $(v,c)$-absorber satisfy the following key properties.
\begin{enumerate}[label=(\theequation)]
\stepcounter{equation}
\item\label{abs-path-head-tail} The initial (resp.~terminal) vertex of a $(v, c)$-absorber is the tail (resp.~head) of both the $(v, c)$-absorbing path and the $(v, c)$-avoiding path.
\stepcounter{equation}
\item\label{abs-path-vtx} The $(v, c)$-absorbing path in a $(v, c)$-absorber contains all of the vertices.
\stepcounter{equation}
\item\label{abs-path-col} The $(v, c)$-absorbing path in a $(v, c)$-absorber is rainbow and contains all of the colours.
\stepcounter{equation}
\item\label{avd-path-vtx} The $(v, c)$-avoiding path in a $(v, c)$-absorber contains all of the vertices except $v$.
\stepcounter{equation}
\item\label{avd-path-col} The $(v, c)$-avoiding path in a $(v, c)$-absorber is rainbow and contains all of the colours except $c$.
\end{enumerate}
We now use $(v,c)$-absorbers to define a `$T$-absorber' for a bipartite graph~$T$, which will essentially form our absorbing structure in almost all $G\in\Phi(\dirK)$, for suitably chosen~$T$.
The role of~$T$ is to provide the `template' for which pairs~$(v,c)$ must provide a $(v,c)$-absorbing path to the rainbow directed Hamilton cycle we are building, and which pairs must provide a $(v,c)$-avoiding path.
%With~$T$ suitably chosen, we will use Lemmas~\ref{masterswitch1} and~\ref{main-switching-lemma} to show that 
\begin{defin}\label{defn:absorber}
%Let $G \in \Phi(\dirK)$, and 
Let $T$ be a bipartite graph with bipartition $(A, B)$.  A digraph $H$ equipped with a proper arc-colouring $\phi$ is a \textit{$T$-absorber} if the following holds.
\begin{enumerate}[(i)]
    \item\label{abs-def-gadgets} There exist injections $f_V : A \rightarrow V(H)$ and $f_C : B\rightarrow \phi(H)$ such that for every $ab\in E(T)$, there is a unique $(v, c)$-absorber $A_{ab} \subseteq H$, where $v = f_V(a)$ and $c = f_C(b)$, satisfying the following.
    \begin{enumerate}[(a)]
    \item\label{abs-vtx-disjoint} For every $ab\in E(T)$, if $V(A_{ab}) \cap V(A_{a'b'}) \neq \emptyset$ for some $a'b'\in E(T)$ where $a'b' \neq ab$, then $a = a'$ and $V(A_{a,b}) \cap V(A_{a'b'}) = \{f_V(a)\}$.
    \item\label{abs-col-disjoint} For every $ab\in E(T)$, if $\phi(A_{ab}) \cap \phi(A_{a'b'}) \neq \emptyset$ for some $a'b'\in E(T)$ where $a'b' \neq ab$, then $b = b'$ and $\phi(A_{ab}) \cap \phi(A_{a'b'}) = \{f_C(b)\}$.
    \end{enumerate}
    \item\label{abs-def-links} There exist pairwise vertex-disjoint length-three paths $P_1, \dots, P_{|E(T)|-1}$, each contained in $H$, satisfying the following.
    \begin{enumerate}[(a)]
    \item\label{abs-link-rainbow} $\bigcup_{i=1}^{|E(T)| - 1} P_i$ is rainbow and $\phi\left(\bigcup_{i=1}^{|E(T)| - 1} P_i\right) \cap \phi\left(\bigcup_{e\in E(T)}A_e\right) = \emptyset$.
    \item\label{abs-link-path} For some enumeration $(e_1, \dots, e_{|E(T)|})$ of $E(T)$, for each $i \in [|E(T)| - 1]$, the tail of $P_i$ is the terminal vertex of $A_{e_i}$ and the head of $P_i$ is the initial vertex of $A_{e_{i+1}}$.
    \item\label{abs-link-disjoint} For each $i \in [|E(T)| - 1]$, $P_i$ is internally vertex-disjoint from $\bigcup_{e\in E(T)}A_{e}$.
    \end{enumerate}
    \item\label{abs-def-min} Subject to~\ref{abs-def-gadgets} and~\ref{abs-def-links}, $H$ is minimal.
\end{enumerate}
In this case, we say a vertex $f_V(a)$ for some $a \in A$ is a \textit{root vertex} of $H$ and a colour $f_C(b)$ for some $b \in B$ is a \textit{root colour} of $H$. Moreover, we say the \textit{initial vertex} of $H$ is the initial vertex of $A_{e_1}$ and the \textit{terminal vertex} of $H$ is the terminal vertex of $A_{e_{|E(T)|}}$.
\end{defin}
Suppose that a bipartite graph~$T$ is robustly matchable (recall Definition~\ref{def:rmbg}) with respect to flexible sets~$A'$ and~$B'$.
The following lemma shows that a $T$-absorber is robustly rainbow-Hamiltonian (recall Definition~\ref{def:rrh}) with respect to the root vertices and colours corresponding to~$A'$ and~$B'$.
This (together with Lemmas~\ref{masterswitch1} and~\ref{main-switching-lemma}) reduces the task of finding such a subdigraph in almost all $G\in\Phi(\dirK)$ to the task of using large well-spread collections of $(v,c)$-absorbing gadgets and $(y,z)$-bridging gadgets to embed a $T$-absorber, for an appropriate robustly matchable~$T$.
%\textcolor{mygreen}{We then show (see Lemma~\ref{lemma:habsorber}) that digraphs $G\in\Phi(\dirK)$ satisfying the conclusions of Lemmas~\ref{masterswitch1} and~\ref{main-switching-lemma} contain such a $T$-absorber with arbitrary root vertices~$U$ and root colours~$D$, and arbitrary subsets of~$U$ and~$D$ corresponding to the flexible sets of~$T$.
%Together, Lemmas~\ref{lemma:2rmbg},~\ref{masterswitch1},~\ref{main-switching-lemma},~\ref{prop:rrH}, and~\ref{lemma:habsorber} imply that almost all $G\in\Phi(\dirK)$ contain a robustly rainbow-Hamiltonian subdigraph with respect to arbitrary flexible sets each of size~$\omega(1)$.}
\begin{lemma}\label{prop:rrH}
Let $G \in \Phi(\dirK)$ with proper $n$-arc-colouring $\phi$. %, let $V\coloneqq V(G)$, and let $C \coloneqq \phi(G)$.
Let ${T}$ be a bipartite graph with bipartition $(A, B)$, let $H\subseteq G$ be a $T$-absorber, and let $u$ and $v$ be the initial and terminal vertices of $H$, respectively.  Let $A'\subseteq A$ and $B' \subseteq B$, and let $V'$ and $C'$ be the set of root vertices and colours of $H$ corresponding to $A'$ and $B'$, respectively.  If $T$ is robustly matchable with respect to flexible sets $A'$ and $B'$, then $H$ is robustly rainbow-Hamiltonian with respect to flexible sets $V'$ and $C'$ and initial and terminal vertices $u$ and $v$.
\end{lemma}
\begin{proof}
Let $X \subseteq V'$ and $Y \subseteq C'$ such that $|X| = |Y| \leq \min\{|V'| / 2, |C'| / 2\}$.  It suffices to show that $H - X$ contains a rainbow directed Hamilton path which starts at $u$ and ends at $v$, not containing a colour in $Y$.  
Since $H$ is a $T$-absorber, by Definition~\ref{defn:absorber}\ref{abs-def-gadgets}, there exist injections $f_V : A \rightarrow V(H)$ and $f_C : B\rightarrow \phi(H)$ such that for every $ab\in E(T)$, there is a unique $(v, c)$-absorber $A_{ab} \subseteq H$, where $v \coloneqq f_V(a)$ and $c \coloneqq f_C(b)$, satisfying~\ref{abs-def-gadgets}\ref{abs-vtx-disjoint} and~\ref{abs-def-gadgets}\ref{abs-col-disjoint}.  By Definition~\ref{defn:absorber}\ref{abs-def-links}, there also exist pairwise vertex-disjoint length-three paths $P_1, \dots, P_{|E(T)| - 1}$, each contained in $H$, satisfying~\ref{abs-def-links}\ref{abs-link-rainbow}, \ref{abs-def-links}\ref{abs-link-path}, and~\ref{abs-def-links}\ref{abs-link-disjoint}.  Let $(e_1, \dots, e_{|E(T)|})$ be the enumeration of $E(T)$ guaranteed by~\ref{abs-def-links}\ref{abs-link-path}.

Since $V'$ and $C'$ are the sets of root vertices and colours of $H$ corresponding to $A'$ and $B'$, respectively, $f_V^{-1}(X) \subseteq A'$ and $f_C^{-1}(Y) \subseteq B'$.  Thus, since $T$ is robustly matchable with respect to $A'$ and $B'$, there exists a perfect matching $M$ in $T - (f_V^{-1}(X) \cup f_C^{-1}(Y))$. 
For each $ab \in E(T)$, define a directed path $P_{ab}$ as follows.  If $ab \in M$, then let $P_{ab}$ be the $(f_V(a), f_C(b))$-absorbing path in $A_{ab}$, and otherwise let $P_{ab}$ be the $(f_V(a), f_C(b))$-avoiding path in $A_{ab}$.  

Now let $P \coloneqq \bigcup_{e \in E(T)}P_e \cup \bigcup_{i=1}^{|E(T)| - 1}P_i$.  We claim that $P$ is a rainbow directed Hamilton path in $H - X$ which starts at $u$, ends at $v$, and does not contain a colour in $Y$.  To that end, we {first} show the following:
\begin{enumerate}[(a)]
    \item\label{claim:u-start} $u$ has out-degree one and in-degree zero in $P$;
    \item\label{claim:v-end} $v$ has in-degree one and out-degree zero in $P$;
    \item\label{claim:in-and-out} every $w\in V(P)\setminus \{u, v\}$ has in-degree and out-degree one in $P$;
    \item\label{claim:avoids-X} {$V(P)\cap X=\emptyset$};
    \item\label{claim:avoids-Y} {$\phi(P)\cap Y=\emptyset$};
    \item\label{claim:hamilton} $V(H)\setminus X \subseteq V(P)$.
\end{enumerate}
Indeed,~\ref{claim:u-start} and \ref{claim:v-end} follow from~\ref{abs-path-head-tail},  \ref{defn:absorber}\ref{abs-def-gadgets}\ref{abs-vtx-disjoint}, \ref{defn:absorber}\ref{abs-def-links}\ref{abs-link-path}, and~\ref{defn:absorber}\ref{abs-def-links}\ref{abs-link-disjoint}%
%%%%%%%%%%%%%%%%%%%%%%%%%%%%%%%%%%%%%%%%%%%%%%%%%%%%%%%%%%%%%
\COMMENT{Since $u$ is the initial vertex of $A_{e_1}$ and $v$ is the terminal vertex of $A_{e_{|E(T)|}}$,~\ref{defn:absorber}\ref{abs-def-gadgets}\ref{abs-vtx-disjoint} implies that $u$ is in $P_{e_1}$ (where it has out-degree one and in-degree zero by~\ref{abs-path-head-tail}) but not in $P_{e_i}$ for $i > 1$ and that $v$ is in $P_{e_{|E(T)|}}$ (where it has in-degree one and out-degree zero by~\ref{abs-path-head-tail}) but not in $P_{e_i}$ for $i \in [|E(T)| - 1]$, and~\ref{defn:absorber}\ref{abs-def-links}\ref{abs-link-path}  and~\ref{defn:absorber}\ref{abs-def-links}\ref{abs-link-disjoint} imply that $u$ and $v$ are not in $P_i$ for any $i \in [|E(T)| - 1]$.  Hence, both~\ref{claim:u-start} and~\ref{claim:v-end} follow.},
%%%%%%%%%%%%%%%%%%%%%%%%%%%%%%%%%%%%%%%%%%%%%%%%%%%%%%%%%%%
and~\ref{claim:in-and-out} follows from~\ref{abs-path-head-tail}, \ref{avd-path-vtx}, \ref{defn:absorber}\ref{abs-def-gadgets}\ref{abs-vtx-disjoint}, \ref{defn:absorber}\ref{abs-def-links}\ref{abs-link-path}, and~\ref{defn:absorber}\ref{abs-def-links}\ref{abs-link-disjoint}.
%%%%%%%%%%%%%%%%%%%%%%%%%%%%%%%%%%%%%%%%%%%%%%%%%%%%%%%%%%%
\COMMENT{By~\ref{defn:absorber}\ref{abs-def-links}\ref{abs-link-disjoint}, every $w\in V(P)$ is either an internal vertex of some $P_i$ for $i\in[|E(T)| - 1]$ or is contained in some $P_e$ for $e\in E(T)$, but not both.  In the former case, $w$ clearly satisfies~\ref{claim:in-and-out}.  In the latter case,~\ref{abs-path-head-tail}, \ref{defn:absorber}\ref{abs-def-gadgets}\ref{abs-vtx-disjoint} and~\ref{defn:absorber}\ref{abs-def-links}\ref{abs-link-path} imply that $w$ satisfies~\ref{claim:in-and-out} if $w \notin V'$ and that if $w \in V'$ and $w \in V(P_e)$, then $e = ab$ where $a \coloneqq f_V^{-1}(w)$.  By~\ref{avd-path-vtx}, $w \in V'$ is not in any path $P_{ab}$ where $a \coloneqq f_V^{-1}(w)$, $b \in N_T(a)$, and $ab \notin M$.  Hence, by ~\ref{defn:absorber}\ref{abs-def-gadgets}\ref{abs-vtx-disjoint} and ~\ref{defn:absorber}\ref{abs-def-links}\ref{abs-link-path}, $w$ has in-degree and out-degree one in $P_{ab}$ where $ab \in M$ and is not in any other $P_e$ or $P_i$, so $w$ satisfies~\ref{claim:in-and-out}.}

To prove~\ref{claim:avoids-X}, note that if $w \in X$, then $a \coloneqq f^{-1}_V(w) \notin V(M)$.  Thus $w$ is not in the paths $P_{ab}$ for $b \in N_T(a)$ by~\ref{avd-path-vtx} as they are all $(w, f_C(b))$-avoiding.  Therefore~\ref{claim:avoids-X} follows again by~\ref{defn:absorber}\ref{abs-def-gadgets}\ref{abs-vtx-disjoint}, \ref{defn:absorber}\ref{abs-def-links}\ref{abs-link-path}, and~\ref{defn:absorber}\ref{abs-def-links}\ref{abs-link-disjoint}.
The proof of~\ref{claim:avoids-Y} is the same, with~\ref{avd-path-col} instead of~\ref{avd-path-vtx}, \ref{defn:absorber}\ref{abs-def-gadgets}\ref{abs-col-disjoint} instead of~\ref{defn:absorber}\ref{abs-def-gadgets}\ref{abs-vtx-disjoint}, and~\ref{defn:absorber}\ref{abs-def-links}\ref{abs-link-rainbow} instead of~\ref{defn:absorber}\ref{abs-def-links}\ref{abs-link-path} and~\ref{defn:absorber}\ref{abs-def-links}\ref{abs-link-disjoint}.
To prove~\ref{claim:hamilton}, first note that if $w \in V' \setminus X$, then there exists $ab\in M$, where $a \coloneqq f^{-1}_V(w)$.  Thus, $w \in V(P_{ab})$ by~\ref{abs-path-vtx} since $P_{ab}$ is $(w, f_C(b))$-absorbing. In particular, $V'\setminus X \subseteq V(P)$.  By~\ref{abs-path-vtx} and~\ref{avd-path-vtx},~\ref{defn:absorber}\ref{abs-def-min} implies that $V(H)\setminus V' \subseteq V(P)$.  Thus, $V(H)\setminus X\subseteq V(P)$, as desired.

By~\ref{abs-path-head-tail}, \ref{defn:absorber}\ref{abs-def-links}\ref{abs-link-path}, \ref{defn:absorber}\ref{abs-def-links}\ref{abs-link-disjoint}, and~\ref{claim:u-start}-\ref{claim:in-and-out}, $P$ contains no cycle,\COMMENT{By~\ref{abs-path-head-tail},  \ref{defn:absorber}\ref{abs-def-links}\ref{abs-link-path}, and~\ref{defn:absorber}\ref{abs-def-links}\ref{abs-link-disjoint}, together with~\ref{claim:u-start}-\ref{claim:in-and-out}, if $w \in V(P_{e_i})$ for $i \in [|E(T)|]$ where $w \neq v$, then the out-neighbour of $w$ in $P$ is in $V(P_i)$ if $w$ is the terminal vertex of $A_{e_i}$ and in $V(P_{e_i})$ otherwise, and if $w \in V(P_i)$ for $i \in [|E(T)| - 1]$, then the out-neighbour of $w$ in $P$ is in $V(P_{e_{i+1}})$ if $w$ is the tail of $P_i$ and in $V(P_i)$ otherwise.  Thus, for every $i \in [|E(T)|]\setminus\{1\}$, there is no directed path in $P$ with tail in $V(P_{e_i})$ and head in $V(P_j)\cup V(P_{e_j})$ for $j < i$, and for every $i \in [|E(T)| - 1]\setminus\{1\}$, there is no directed path in $P$ with tail in $V(P_i)$ and head in $V(P_i)$ for $j < i$ or in $V(P_{e_j})$ for $j \leq i$.  Consequently, a cycle in $P$ would be contained in some $P_i$ or $P_{e_i}$, a contradiction.}
%%%%%%%%%%%%%%%%%%%%%%%%%%%%%%%%%%%%%%%%%%%%%%%%%%%%
so \ref{claim:u-start}-\ref{claim:avoids-Y} imply that $P$ is indeed a directed path in $H - X$, {not containing a colour in~$Y$,} which starts at $u$ and ends at $v$, and~\ref{claim:hamilton} implies that $P$ is Hamilton in $H - X$, as required. % 
It remains to show that~$P$ is rainbow.  By~\ref{abs-path-col}, \ref{avd-path-col}, and~\ref{defn:absorber}\ref{abs-def-gadgets}\ref{abs-col-disjoint}, $\bigcup_{e \in E(T)}P_e$ is rainbow, so~\ref{defn:absorber}\ref{abs-def-links}\ref{abs-link-rainbow} implies that $P$ is rainbow, as required. 
\end{proof}

The following proposition implies that there are many short rainbow paths that are `well-spread' in \textit{all} $G\in\Phi(\dirK)$. This enables us to embed these paths in any such~$G$ in a vertex- and colour-disjoint way whilst constructing a $T$-absorber for suitably chosen~$T$, and whilst absorbing the colours unused by the large rainbow directed path forests we find in Section~\ref{sec:proof}.
\begin{prop}\label{prop:num-links}
    Let $G \in \Phi(\dirK)$ with proper $n$-arc-colouring $\phi$, let $V\coloneqq V(G)$, and let $C \coloneqq \phi(G)$.
    Let $u, v \in V$ such that $u\neq v$, and let $c\in C$.  The following holds for $n$ sufficiently large.
    \begin{enumerate}[(i)]
        \item\label{lb:links} There are at least $n^2 / 3$ length-three directed rainbow paths in $G$ with head $v$ and tail $u$.
        \item\label{lb:covers} If at most $n / 2$ loops in $G$ are coloured $c$, then there are at least $n^2 / 5$ length-four directed rainbow paths in $G$ with head $v$ and tail $u$ such that the second arc is coloured $c$.
        \item\label{ub:link-vtx} For every $w\in V$, there are at most $2n$ length-three directed rainbow paths in $G$ with head $v$ and tail $u$ that contain $w$ as an internal vertex.
        \item\label{ub:cover-vtx} For every $w\in V$, there are at most $3n$ length-four directed rainbow paths in $G$ with head $v$ and tail $u$ that contain $w$ as an internal vertex such that the second arc is coloured $c$.
        \item\label{ub:link-col} For every $d\in C$, there are at most $3n$ length-three directed rainbow paths in $G$ with head $v$ and tail $u$ that contain an arc coloured $d$.
        \item\label{ub:cover-col} For every $d\in C\setminus\{c\}$, there are at most $3n$ length-four directed rainbow paths in $G$ with head $v$ and tail $u$ that contain an arc coloured $d$ such that the second arc is coloured $c$.
    \end{enumerate}
\end{prop}
\begin{proof}
    First, for each vertex $w \in V$, we let $B_w \coloneqq \{x \in V\setminus \{w, v\} : \phi(\myArc{w}{x}) = \phi(\myArc{x}{v})\}$, we say $w$ is \textit{bad} if $|B_w| > n / 2$, and we let $B \subseteq V$ be the set of bad vertices.  We claim that there is at most one bad vertex; that is, $|B| \leq 1$.  To that end, suppose for a contradiction that distinct vertices $w$ and $w'$ are bad.  Since $|B_w| + |B_{w'}| > n$, we have $B_w \cap B_{w'} \neq \emptyset$.  Thus, there exists some $x \in B_w\cap B_{w'}$, so $\phi(\myArc{w}{x}) = \phi(\myArc{x}{v}) = \phi(\myArc{w'}{x})$, contradicting that $\phi$ is a  proper arc-colouring.

    Now we prove~\ref{lb:links}.  Since $|B| \leq 1$, there are at least $(n - 3)(n/2 - 5)$ choices of an ordered pair $(w_1, w_2)$ where $w_1\in V\setminus(B\cup \{u, v\})$ and $w_2 \in V\setminus (B_{w_1}\cup\{u, v, w_1\})$ such that $\phi(\myArc{u}{w_1}) \notin \{\phi(\myArc{w_1}{w_2}), \phi(\myArc{w_2}{v})\}$.  For each such pair, there is a distinct directed path $P_{w_1,w_2} \coloneqq uw_1w_2v$ in $G$.  Since $\phi(\myArc{u}{w_1}) \notin \{\phi(\myArc{w_1}{w_2}), \phi(\myArc{w_2}{v})\}$ and $w_2\notin B_{w_1}$, $P_{w_1, w_2}$ is rainbow.  Therefore there are at least $n^2 / 3$ directed length-three rainbow paths in $G$ with head $v$ and tail $u$, as desired.
    
    Now we prove~\ref{lb:covers}.  Let $X \coloneqq \{w \in V : \phi(\myArc{w}{w}) = c\}$; by assumption, $|X| \leq n / 2$.  Thus, since $|B| \leq 1$, there are at least $n/2 - 6$ choices of an ordered pair $(w_1, w_2)$ where $w_2 \in V\setminus (X\cup B\cup\{u, v\})$, $w_1 \in V\setminus \{u, v, w_2\}$, $\phi(\myArc{w_1}{w_2}) = c$, and $\phi(\myArc{u}{w_1}) \neq c$.  For each such pair, there are at least $n / 2 - 8$ choices of a vertex $w_3 \in V\setminus (B_{w_2} \cup \{u, v, w_1, w_2\})$ such that $\{\phi(\myArc{w_2}{w_3}), \phi(\myArc{w_3}{v})\} \cap \{\phi(\myArc{u}{w_1}), c\} = \emptyset$.  For each such choice of $(w_1, w_2, w_3)$, there is a distinct directed length-four rainbow path $P_{w_1, w_2, w_3} \coloneqq uw_1w_2w_3v$ with head $v$ and tail $u$ such that the second arc is coloured $c$.  Therefore there are at least $(n/2 - 7)(n/ 2 - 8) \geq n^{2} / 5$ such paths, as desired.
    
    The proofs of~\ref{ub:link-vtx}-\ref{ub:cover-col} are similar, so we only provide a complete proof of~\ref{ub:cover-col}.\COMMENT{
    Now we prove~\ref{ub:link-vtx}. {Fix $w\in V$, and} let $\cP_{\ref{ub:link-vtx}}$ be the set of length-three rainbow directed paths in $G$ with head $v$ and tail $u$ that contain $w$ as an internal vertex.  Partition $\cP_{\ref{ub:link-vtx}}$ into sets $\cP_{\ref{ub:link-vtx},1}$ and $\cP_{\ref{ub:link-vtx},2}$ such that for $i \in [2]$, a path $P \in \cP_{\ref{ub:link-vtx},i}$ if $w$ is the $(i + 1)$th vertex of $P$.  Every path in $\cP_{\ref{ub:link-vtx},1}$ is uniquely determined by the vertex in the path adjacent to $v$, and similarly, every path in $\cP_{\ref{ub:link-vtx},2}$ is uniquely determined by the vertex in the path adjacent to $u$.  Therefore $|\cP_{\ref{ub:link-vtx},i}| \leq n$ for each $i \in [2]$, so $|\cP_{\ref{ub:link-vtx}}| \leq 2n$, as desired.
    Now we prove~\ref{ub:cover-vtx}. {Fix $w\in V$, and} let $\cP_{\ref{ub:cover-vtx}}$ be the set of length-four rainbow directed paths in $G$ with head $v$ and tail $u$ that contain $w$ as an internal vertex such that the second arc is coloured $c$.  Partition $\cP_{\ref{ub:cover-vtx}}$ into sets $\cP_{\ref{ub:cover-vtx},1}$, $\cP_{\ref{ub:cover-vtx},2}$, and $\cP_{\ref{ub:cover-vtx},3}$ such that for $i \in [3]$, a path $P \in \cP_{\ref{ub:cover-vtx},i}$ if $w$ is the $(i + 1)$th vertex of $P$.  Every path in $\cP_{\ref{ub:cover-vtx},1}$ or $\cP_{\ref{ub:cover-vtx},2}$ is uniquely determined by the vertex adjacent to $v$, and every path in $\cP_{\ref{ub:cover-vtx},3}$ is uniquely determined by an ordered pair $(w_1, w_2)$ such that $\phi(\myArc{w_1}{w_2}) = c$.  Therefore $|\cP_{\ref{ub:cover-vtx},i}| \leq n$ for each $i \in [3]$, so $|\cP_{\ref{ub:cover-vtx}}| \leq 3n$, as desired.
    Now we prove~\ref{ub:link-col}. {Fix $d\in C$, and} let $\cP_{\ref{ub:link-col}}$ be the set of length-three rainbow directed paths in $G$ with head $v$ and tail $u$ that contain an arc coloured $d$.   Partition $\cP_{\ref{ub:link-col}}$ into sets $\cP_{\ref{ub:link-col},1}$, $\cP_{\ref{ub:link-col},2}$, and $\cP_{\ref{ub:link-col},3}$ such that for $i \in [3]$, a path $P \in \cP_{\ref{ub:link-col},i}$ if the arc coloured $d$ is the $i$th arc of $P$.  Every path in $\cP_{\ref{ub:link-col},1}$ is uniquely determined by the vertex in the path adjacent to $v$, and similarly, every path in $\cP_{\ref{ub:link-col},3}$ is uniquely determined by the vertex in the path adjacent to $u$.  Every path in $\cP_{\ref{ub:link-col},2}$ is uniquely determined by an ordered pair $(w_1, w_2)$ such that $\phi(\myArc{w_1}{w_2}) = d$.  Therefore $|\cP_{\ref{ub:link-col},i}| \leq n$ for each $i \in [3]$, so $|\cP_{\ref{ub:link-col}}| \leq 3n$, as desired.
    }
    %Finally, we prove~\ref{ub:cover-col}.  
    {Fix $d\in C\setminus\{c\}$, and} let $\cP_{\ref{ub:cover-col}}$ be the set of length-four rainbow directed paths in $G$ with head $v$ and tail $u$ that contain an arc coloured $d$ such that the second arc is coloured $c$.  Partition $\cP_{\ref{ub:cover-col}}$ into sets $\cP_{\ref{ub:cover-col},1}, \dots, \cP_{\ref{ub:cover-col},4}$ such that for $i \in [4]$, a path $P \in \cP_{\ref{ub:cover-col},i}$ if the arc coloured $d$ is the $i$th arc of $P$. Since $d \neq c$, $\cP_{\ref{ub:cover-col},2} = \emptyset$.  Every path in $\cP_{\ref{ub:cover-col},1}$ is uniquely determined by the vertex in the path adjacent to $v$, every path in $\cP_{\ref{ub:cover-col},3}$ is uniquely determined by an ordered pair $(w_2, w_3)$ such that $\phi(\myArc{w_2}{w_3}) = d$, and every path in $\cP_{\ref{ub:cover-col},4}$ is uniquely determined by an ordered pair $(w_1, w_2)$ such that $\phi(\myArc{w_1}{w_2}) = c$.  Therefore $|\cP_{\ref{ub:cover-col}}| \leq 3n$, as desired.
    
    We conclude by outlining the necessary changes to the proof of~\ref{ub:cover-col} to obtain proofs of~\ref{ub:link-vtx}-\ref{ub:link-col}.  Define $\cP_{\ref{ub:link-vtx}}$, $\cP_{\ref{ub:cover-vtx}}$, and~$\cP_{\ref{ub:link-col}}$ in an analogous way.  Partition $\cP_{\ref{ub:link-vtx}}$ and $\cP_{\ref{ub:cover-vtx}}$ based on the position of $w$ in each path, and partition~$\cP_{\ref{ub:link-col}}$ based on the position of the arc coloured $d$.  Finally, show that each part contains at most $n$ paths.
\end{proof}
For any $m=\omega(1)$ (in Section~\ref{sec:proof} we will set $m\coloneqq\lfloor n/\log^{3}n\rfloor$, and we assume~$n$ to be sufficiently large) we have by Lemma~\ref{lemma:2rmbg} that there exists a $256$-regular $2RMBG(7m,2m)$, say~$T$ (recall Definition~\ref{def:rmbg}).
In the following lemma, we show that if~$m<n/\log n$, then we may greedily embed a~$T$-absorber in any $G\in\Phi(\dirK)$ that satisfies the conclusions of Lemma~\ref{masterswitch1} and~\ref{main-switching-lemma}, by choosing each absorber successively in three steps: for each edge $vc$ of~$T$, we first embed a $(v,c)$-absorbing gadget~$A$, then choose a bridging gadget~$B$ that bridges~$A$ (recall Definition~\ref{defn:(v,c)-absorber}), and finally use Proposition~\ref{prop:num-links} to embed the extra short rainbow paths required to complete the $(v,c)$-absorber and connect it to the previously embedded absorber.
%\textcolor{mygreen}{OK SO RECALL SOME RMBG STUFF TO THE TUNE OF THE BELOW.}
%\textcolor{purple}{Use the RMBG lemma to say that there exists an~$\Tilde{H}$ with flexible sets (or total number of vertices if simpler) of size $n/\log^{2}n$.}
%recall 2RMBG def
\begin{lemma}\label{lemma:habsorber}
Let $G \in \Phi(\dirK)$ with proper $n$-arc-colouring $\phi$, let $V\coloneqq V(G)$, and let $C \coloneqq \phi(G)$.
Let $m < n / \log n$, let $T$ be a $256$-regular $2RMBG(7m, 2m)$, let $U\subseteq V$ and $D\subseteq C$ such that $|U| = |D| = 7m$, and let $V' \subseteq U$ and $C' \subseteq D$ such that $|V'| = |C'| = 2m$.  For $n$ sufficiently large, if
\begin{itemize}
    \item for all $v \in V$ and $c\in C$,~$G$ contains a well-spread collection~$\cA_{v,c}$ of at least~$n^{2}/2^{100}$ $(v,c)$-absorbing gadgets and 
    \item for all distinct $y,z\in V$,~$G$ contains a well-spread collection~$\cB_{y,z}$ of at least~$n^{2}/10^{50}$ $(y,z)$-bridging gadgets,
\end{itemize}
then~$G$ contains a $T$-absorber $H$ rooted on vertices~$U$ and colours~$D$ such that $V'$ and $C'$ are the sets of root vertices and colours of $H$ corresponding to the flexible sets of $T$.
\end{lemma}
\begin{proof}
Denote the bipartition of $T$ by $(A, B)$, let $A' \subseteq A$ and $B'\subseteq B$ be the flexible sets of $T$, and enumerate the edges of $T$ as $e_1, \dots, e_{|E(T)|}$.  %Since $m < n / \log n$,
%\begin{equation}\label{eq:template-edge-size-bound}
%    |E(T)| \leq 1792 n / \log n.
%\end{equation}
Let $f_V : A \rightarrow U$ and $f_C : B \rightarrow D$ be bijections, chosen such that $f_V(A') = V'$ and $f_C(B') = C'$.  For each $j \in [|E(T)|]$, we inductively choose 
\begin{enumerate}[(i)]
    \item\label{constructing-abs:gadget} a $(v, c)$-absorbing gadget $A_j$ in $G$, where $ab \coloneqq e_j$, $v\coloneqq f_V(a)$, and $c \coloneqq f_V(b)$, 
    \item\label{constructing-abs:bridge} a $(y_j, z_j)$-bridging gadget $B_j$ in $G$, where $(y_j, z_j)$ is the pair of abutment vertices of $A_j$, which bridges $A_j$, 
    and 
    \item\label{constructing-abs:links} length-three rainbow {directed} paths $P_{j,1}, \dots, P_{j,5}$ in $G$, such that $P_{j,1}, \dots, P_{j,4}$ complete the pair $(A_j, B_j)$ to a $(v, c)$-absorber $A^*_j$,
    %\item a length-three rainbow path $P_{j,5}$ in $G$,  
\end{enumerate}
such that the following holds:
\begin{enumerate}[(a)$_j$]
    \item\label{constructing-abs:gadget-vtx} $V(A_j) \cap U = \{f_V(a)\}$, $V(A_j) \cap V(P_{\ell, 5}) = \emptyset$ for all $\ell < j$, and if $V(A_j) \cap V(A_\ell^*) \neq \emptyset$ for {some} $\ell < j$, then $V(A_j) \cap V(A_\ell^*) = \{f_V(a)\}$, where $a \in e_j \cap e_\ell \subseteq A$;
    \item\label{constructing-abs:gadget-col} $\phi(A_j) \cap D = \{f_C(b)\}$, $\phi(A_j) \cap \phi(P_{\ell, 5}) = \emptyset$ for all $\ell < j$, and if $\phi(A_j) \cap \phi(A_\ell^*) \neq \emptyset$ for {some} $\ell < j$, then $\phi(A_j) \cap \phi(A_\ell^*) = \{f_C(b)\}$, where $b \in e_j \cap e_\ell \subseteq B$;
    \item\label{constructing-abs:bridge-vtx} $V(B_j) \cap V(A_\ell^* \cup P_{\ell, 5}) = \emptyset$ for all $\ell < j$, and $V(B_j) \cap U = \emptyset$;
    \item\label{constructing-abs:bridge-col} $\phi(B_j) \cap \phi(A_\ell^* \cup P_{\ell, 5}) = \emptyset$ for all $\ell < j$, and $\phi(B_j) \cap D = \emptyset$;
    \item\label{constructing-abs:paths-vtx} $V(P_{j,k}) \cap V(A_\ell^* \cup P_{\ell, 5}) = \emptyset$ for all $k \in [4]$ and $\ell < j$, and $V(P_{j, k}) \cap U = \emptyset$ for all $k \in [5]$;
    \item\label{constructing-abs:paths-col} $\phi(P_{j,k}) \cap \phi(A_\ell^* \cup P_{\ell, 5}) = \emptyset$ for all $k \in [4]$ and $\ell < j$, and $\phi(P_{j, k}) \cap D = \emptyset$ for all $k \in [5]$;
    \item\label{constructing-abs:link-col} $\phi(P_{j,5}) \cap \phi(A_\ell^*) = \emptyset$ for all $\ell \leq j$, and $\phi(P_{j,5}) \cap \phi(P_{\ell, 5}) = \emptyset$ for all $\ell < j$; 
    \item\label{constructing-abs:link-vtx} if $j > 1$, then the tail of $P_{j,5}$ is the terminal vertex of $A^*_{j - 1}$, the head of $P_{j,5}$ is the initial vertex of $A^*_j$, $P_{j,5}$ is internally vertex-disjoint from $A^*_\ell$ for all $\ell \leq j$, $V(P_{j,5}) \cap V(P_{\ell, 5}) = \emptyset$ for all $\ell < j$, and if $j = 1$, then $P_{j,5} = \emptyset$.
\end{enumerate}
To that end, we let $i \in [|E(T)|]$ and assume $A_j$, $B_j$, and $P_{j,1}, \dots, P_{j,5}$ satisfying~\ref{constructing-abs:gadget-vtx}-\ref{constructing-abs:link-vtx} have been chosen for $j < i$, and we show that we can indeed choose $A_j$, $B_j$, and $P_{j,1}, \dots, P_{j,5}$ according to~\ref{constructing-abs:gadget}-\ref{constructing-abs:links} satisfying~\ref{constructing-abs:gadget-vtx}-\ref{constructing-abs:link-vtx} for $j = i$.
Let $U' \coloneqq U \cup \bigcup_{\ell < j} V(A^*_\ell \cup P_{\ell, 5})$, and let $D' \coloneqq D \cup \bigcup_{\ell < j}\phi(A^*_\ell \cup P_{\ell, 5})$.
%\textcolor{red}{Since $|V(P_{\ell, k})| = 4$ and $|\phi(P_{\ell, k})| = 3$ for every $k \in [5]$ and $\ell \leq j$ by~\ref{constructing-abs:links}, we have $|V(A^*_\ell \cup P_{\ell, 5}) \setminus (U \cup \bigcup_{\ell' < \ell}V(A^*_\ell))|, |\phi(A^*_\ell) \setminus D| \leq 22$.}
For every $\ell<j$, we have by~\ref{constructing-abs:gadget}--\ref{constructing-abs:links} that $|V(A^{*}_{\ell}\cup P_{\ell,5})|$,~$|\phi(A^{*}_{\ell}\cup P_{\ell,5})|\leq 24$.
Thus, since $T$ is $256$-regular and $m < n / \log n$, 
%\begin{equation}\label{eq:used-vtcs-col-bound}
%    \textcolor{red}{|U'|, |D'| \leq 22j+7m\leq 161m \leq  41216 n / \log n.}
%\end{equation}
\begin{equation}\label{eq:used-vtcs-col-bound}
    |U'|,|D'| \leq 7m+24j \leq 43008m < 43008n/\log n.
\end{equation}
First we show that we can choose a $(v, c)$-absorbing gadget $A_j$ according to~\ref{constructing-abs:gadget} satisfying~\ref{constructing-abs:gadget-vtx} and~\ref{constructing-abs:gadget-col}.  By assumption, $G$ contains a well-spread collection $\cA_{v, c}$ of $n^2 / 2^{100}$ $(v, c)$-absorbing gadgets.  For each $u \in U'$, let $\cA_u \coloneqq \{A^{\mathrm{gdgt}} \in \cA_{v, c} : u \in V(A^{\mathrm{gdgt}})\}$, and for each $d \in D'$, let $\cA_d \coloneqq \{A^{\mathrm{gdgt}} \in \cA_{v, c}: d \in \phi(A^{\mathrm{gdgt}})\}$.  
Let $\cA'_{v, c} \coloneqq \cA_{v, c}\setminus(\bigcup_{u \in U'\setminus\{v\}}\cA_u \cup \bigcup_{d \in D'\setminus\{c\}} \cA_d)$.  Since $\cA_{v, c}$ is well-spread, $|\cA_u|, |\cA_d| \leq n$ for every $u \in U'\setminus \{v\}$ and $d \in D'\setminus \{d\}$, so by~\eqref{eq:used-vtcs-col-bound}, $|\cA'_{v, c}| \geq |\cA_{v, c}| - 86016 n^2 / \log n > 0$.  In particular, there exists $A_j \in \cA'_{v, c}$.  
By construction of $\cA'_{v, c}$,\COMMENT{We only show~\ref{constructing-abs:gadget-vtx}, as the proof for~\ref{constructing-abs:gadget-col} is similar.  First, if $u \in V(A_j) \cap U'$, then $A_j \in \cA_u$, so $u  = v$.  Thus, $V(A_j) \cap U = \{f_V(a)\}$, as desired, and if $V(A_j) \cap V(A^*_\ell) \neq \emptyset$ for some $\ell < j$, then $V(A_j) \cap V(A^*_\ell) = \{v\}$, as required.  Moreover, by the inductive hypothesis, $V(A^*_\ell) \cap U = \{f_V(a')\}$, where $e_\ell = a'b'$, so $a' = a$.  Thus, $a \in e_j \cap e_\ell \subseteq A$, as desired.}
%We claim that $A_j$ satisfies~\ref{constructing-abs:gadget-vtx} and~\ref{constructing-abs:gadget-col}.
$A_j$ satisfies~\ref{constructing-abs:gadget-vtx} and~\ref{constructing-abs:gadget-col}, as desired.

Now let~$(y_j, z_j)$ be the abutment vertices of~$A_j$.  By a similar argument, we can choose a $(y_j, z_j)$-bridging gadget $B_j$ according to~\ref{constructing-abs:bridge} satisfying~\ref{constructing-abs:bridge-vtx} and~\ref{constructing-abs:bridge-col}.\COMMENT{Next we show that we can choose a $(y_j, z_j)$-bridging gadget $B_j$ according to~\ref{constructing-abs:bridge} satisfying~\ref{constructing-abs:bridge-vtx} and~\ref{constructing-abs:bridge-col}.  By assumption, $G$ contains a well-spread collection $\cB_{y_j, z_j}$ of at least $n^2 / 10^{50}$ $(y_j, z_j)$-bridging gadgets.  For each $u \in U' \cup V(A_j)$, let $\cB_u \coloneqq \{B^{\mathrm{gdgt}} \in \cB_{y_j, z_j} : u \in V(B^{\mathrm{gdgt}})\}$, and for each $d \in D' \cup \phi(A_j)$, let $\cB_d \coloneqq \{B^{\mathrm{gdgt}} \in \cB_{y_j, z_j} : d \in \phi(B^{\mathrm{gdgt}})\}$.  
Let $\cB'_{y_j, z_j} \coloneqq \cB_{y_j, z_j} \setminus (\bigcup_{u \in U' \cup V(A_j)\setminus\{y_j, z_j\}}\cB_u \cup \bigcup_{d \in D'\cup \phi(A_j)}\cB_d)$.  Since $\cB_{y_j, z_j}$ is well-spread, $|\cB_u|, |\cB_d| \leq n$ for every $u \in U' \cup V(A_j)\setminus\{y_j, z_j\}$ and $d \in D' \cup \phi(B_j)$, so by~\eqref{eq:used-vtcs-col-bound}, $|\cB'_{v, c}| \geq |\cB_{v, c}| - 86018 n^{2} / \log n > 0$.  In particular, there exists $B_j \in \cB'_{y_j, z_j}$.  By construction of $\cB'_{y_j,z_j}$, $B_j$ bridges $A_j$ (as required by~\ref{constructing-abs:bridge}) and satisfies~\ref{constructing-abs:bridge-vtx} and~\ref{constructing-abs:bridge-col}, as desired.}

Finally, we show that we can choose $P_{j,1}, \dots, P_{j,5}$ according to~\ref{constructing-abs:links} satisfying~\ref{constructing-abs:paths-vtx}%, \ref{constructing-abs:paths-col}, \ref{constructing-abs:link-col},
-\ref{constructing-abs:link-vtx}.  The argument is again similar, using~\ref{lb:links}, \ref{ub:link-vtx}, and~\ref{ub:link-col} of Proposition~\ref{prop:num-links} instead of the existence of a well-spread collection of gadgets, so we omit the proof.
\COMMENT{Let $t_1, h_1 \in V(A_j)$ and $t_2, t_3, t_4, h_2, h_3, h_4 \in V(B_j)$ be the vertices such that if~$(P_{1},\dots, P_{4})$ are directed paths in~$G$ which complete the pair $(A_{j},B_{j})$, then~$P_{i}$ has tail~$t_{i}$ and head~$h_{i}$.
Moreover, if $j > 1$, then let {$t_5$} be the terminal vertex of $A^*_{j - 1}$, and let {$h_5$} be the initial vertex of $A^*_j$.   
For each $i \in [5]$, by Proposition~\ref{prop:num-links}\ref{lb:links}, there is a set $\cP_i$ of at least $n^2 / 3$ directed rainbow length-three paths in $G$ with head $h_i$ and tail $t_i$.  For each $u \in U' \cup V(A_j \cup B_j)$, let $\cP_u \coloneqq \{P \in \cP_1 : u \in V(P)\}$, and for each $d \in D' \cup \phi(A_j) \cup \phi(B_j)$, let $\cP_d \coloneqq \{P \in \cP_1 : d \in \phi(P)\}$. Let $\cP'_1 \coloneqq \cP_{1} \setminus (\bigcup_{u \in U' \cup V(A_j\cup B_j)\setminus\{h_1 ,t_1\}} \cP_u \cup \bigcup_{d \in D' \cup \phi(A_j\cup B_j)}\cP_d)$.  By Proposition~\ref{prop:num-links}\ref{ub:link-vtx}, $|\cP_u| \leq 2n$ for every $u \in U'\cup V(A_j \cup B_j)\setminus\{h_1, t_1\}$, and by Proposition~\ref{prop:num-links}\ref{ub:link-col}, $|\cP_d| \leq 3n$ for every $d \in D'\cup \phi(A_j\cup B_j)$, so by~\eqref{eq:used-vtcs-col-bound}, $|\cP'_1| \geq |\cP_1| - 215045n^2 / \log n > 0$\COMMENT{$2n(|U'|+10)+3n(|D'|+7)\leq 5n(\max(|U'|,|D'|)+41/5)\leq 5n\cdot43009n/\log n$.}. In particular, there exists $P_{j,1} \in \cP'_1$.  By construction of $\cP'_1$, $P_{j,1}$ satisfies~\ref{constructing-abs:paths-vtx} and~\ref{constructing-abs:paths-col}, is internally vertex-disjoint from $A_j \cup B_j$, and shares no colour with $A_j \cup B_j$.  By repeating this argument successively for each $k \in \{2, 3, 4\}$, we can choose $P_{j,k} \in \cP_k$ to also satisfy~\ref{constructing-abs:paths-vtx} and~\ref{constructing-abs:paths-col} and to be internally vertex-disjoint from and share no colour with $A_j\cup B_j \cup \bigcup_{k' < k}P_{j,k'}$.  Thus, $P_{j,1}, \dots, P_{j,4}$ complete the pair $(A_j, B_j)$, as desired.  Moreover, the same argument allows us to choose $P_{j,5} \in \cP_5$ to satisfy~\ref{constructing-abs:paths-vtx}-\ref{constructing-abs:link-vtx}, as  desired.}

To complete the proof, we show that $H\coloneqq \bigcup_{i = 1}^{|E(T)|}A^*_i \cup \bigcup_{i = 1}^{|E(T)|-1} P_{i + 1, 5}$ is a $T$-absorber rooted on vertices $U$ and colours $D$.  Since~\ref{constructing-abs:gadget-vtx}-\ref{constructing-abs:paths-col} hold for every $j \in [|E(T)|]$, $H$ satisfies~\ref{defn:absorber}\ref{abs-def-gadgets}, and since~\ref{constructing-abs:link-col} and~\ref{constructing-abs:link-vtx} hold for every $j \in [|E(T)|]$, $H$ satisfies~\ref{defn:absorber}\ref{abs-def-links}.  Clearly $H$ is minimal with respect to these properties, so $H$ also satisfies~\ref{defn:absorber}\ref{abs-def-min}.  Thus, $H$ is a $T$-absorber rooted on $U$ and $D$, as desired. 
Moreover, by the choice of $f_V$ and $f_C$, $V'$ and $C'$ are the sets of root vertices and colours of $H$ corresponding to the flexible sets of $T$, as desired.
\end{proof}

\section{Proof of Theorem~\ref{main-thm}}\label{sec:proof}
In this section we use the results we have obtained thus far to prove Theorem~\ref{main-thm}.
We begin by arguing that a `lower-quasirandomness' condition in $G\in\Phi(\dirK)$ (which holds in almost all $G\in\Phi(\dirK)$ by Theorem~\ref{thm:discrepancy}) is enough to ensure the existence of many rainbow directed path forests spanning all but a small arbitrary set of vertices, avoiding a small arbitrary forbidden set of colours, and having few components.
We remark that the method we use to count the rainbow directed path forests is inspired by the method used by Kwan (see the proof of~\cite[Lemma 5.5]{K20}) to count large matchings in random Steiner triple systems.
\begin{defin}\label{def:lowerquas}
Let $G \in \Phi(\dirK)$ with proper $n$-arc-colouring $\phi$, let $V\coloneqq V(G)$, and let $C \coloneqq \phi(G)$.  We say that $G$ is \textit{lower-quasirandom} if for all (not necessarily distinct) sets $U_{1}, U_{2} \subseteq V$ and $D\subseteq C$, we have that $e_{G,D}(U_{1},U_{2})\geq|U_{1}||U_{2}||D|/n-n^{5/3}$.
\end{defin}

\begin{lemma}\label{lemma:pathforests}
Let $G \in \Phi(\dirK)$ with proper $n$-arc-colouring $\phi$, let $V\coloneqq V(G)$, and let $C \coloneqq \phi(G)$.
Let $U\subseteq V$ and let $D \subseteq C$ be equal-sized sets of size at most $n/\log^2 n$.  %For $n$ sufficiently large, 
If $G$ is lower-quasirandom, then there are at least $\left((1-o(1))n/e^{2}\right)^{n}$ spanning rainbow directed path forests~$Q$ of~$G-U$ such that $\phi(Q)\cap D=\emptyset$ and $Q$ has at most $n^{9/10}$ components.
\end{lemma}
\begin{proof}
  Throughout the proof we implicitly assume $n$ is sufficiently large for certain inequaities to hold.
  We say a rainbow directed path forest in $G$ is \textit{valid} if it has no vertices in $U$ and no colours in $D$.  Let $n ' \coloneqq  n - |U| - \lfloor n^{9/10} \rfloor$, and let $n'' \coloneqq n - |U|$.  If $Q$ is a spanning directed path forest of $G - U$, then the number of components of $Q$ is equal to $|V\setminus U| - |E(Q)|$, so $Q$ has at most $n^{9/10}$ components if and only if it has at least $n'$ arcs.  Thus, it suffices to count the number of valid rainbow directed path forests in $G$ that have $n'$ arcs.  To that end, we first count the number of \textit{ordered} sequences of arcs $(e_1, \dots, e_{n'})$ such that $\bigcup_{i=1}^{n'}e_i$ is a valid rainbow directed path forest in $G$.
  We claim that for every $j \in [n']$, if $e_1, \dots, e_{j - 1}$ are arcs in $G$ such that $\bigcup_{i=1}^{j-1} e_i$ is a valid rainbow directed path forest, then there are at least $\left(1 - o(1)\right)\left(n'' - (j - 1)\right)^3 / n$ choices of an arc $e_j \in E(G)$ such that $\bigcup_{i=1}^j e_j$ is also a valid rainbow directed path forest.  
  Let $V_H$ be the set of vertices $u \in V\setminus U$ such that $u$ has no out-neighbor in $\bigcup_{i = 1}^{j - 1}e_i$, let $V_T$ be the set of vertices $u \in V\setminus U$ such that $u$ has no in-neighbor in $\bigcup_{i = 1}^{j - 1}e_i$, and let $D' \coloneqq (C\setminus D)\setminus \bigcup_{i=1}^{j - 1} \phi(e_i)$. 
  Since $G$ is lower-quasirandom, since $|V_H| = |V_T| = |D'| = n - |U| - (j - 1)$, and since $j \leq n' = n'' - \lfloor n^{9/10}\rfloor$, we have that 
  \begin{equation}\label{eqn:counting-choices-lb}
  e_{G, D'}(V_H, V_T) \geq\COMMENT{$\frac{|V_H||V_T||D'|}{n} - n^{5/3} = $} 
  \frac{(n - |U| - (j - 1))^3}{n} - n^{5/3} \geq \left(1 - o(1)\right)\frac{(n'' - (j - 1))^3}{n}.\COMMENT{In the last inequality we used that $\frac{(n' + \lfloor n^{9/10}\rfloor - (j - 1))^3}{n} \geq \frac{(n^{9/10})^3}{2n} = \omega(n^{5/3})$.}
  \end{equation}
The spanning path forest in $G - U$ with edge set $\{e_1, \dots, e_{j - 1}\}$ has $k \coloneqq n'' - (j - 1)$ components, which we denote $P_1, \dots, P_k$.  For every $i \in [k]$, there is a unique arc $f_i \in E(G)$ (whose head is the tail of $P_i$ and whose tail is the head of $P_i$) such that $P_i\cup f_i$ is a directed cycle.  Let $F \coloneqq \{f_1, \dots, f_k\}$, and let $e_j \in E_{G, D'}(V_H, V_T)\setminus F$.  By the choice of $V_H$, $V_T$, $D'$, and $F$, we have that $\bigcup_{i = 1}^je_i$ is rainbow, has maximum in-degree and out-degree one, and contains no cycle.  Hence, it is a valid rainbow directed path forest, as required.  By~\eqref{eqn:counting-choices-lb}, $|E_{G, D'}(V_H, V_T)\setminus F| \geq \left(1 - o(1)\right)\left(n'' - (j - 1)\right)^3 / n$, so the claim follows.
  
  Therefore, the number of ordered sequences of arcs $e_1, \dots, e_{n'}\in E(G)$ such that $\bigcup_{i = 1}^{n'}e_i$ is a valid rainbow directed path forest is at least
  \begin{equation}\label{eqn:counting-main-lb}
      \prod_{j = 1}^{n'} \left(1 - o(1)\right)\frac{(n'' - (j - 1))^3}{n} = \left(\left(1 - o(1)\right)\frac{(n'')^3}{n}\right)^{n'}\exp\left(3\sum_{j = 1}^{n'}\log\left(1 - \frac{j - 1}{n''}\right)\right).
  \end{equation}
  Since $|U| \leq n / \log^{2} n$, we have $n' \geq n - 2 n / \log^{2} n$ and $(n'')^3 / n \geq (1 - 3 / \log^{2}n)n^2$.  Hence,
  \begin{equation}\label{eqn:counting-error-term}
      \left(\frac{(n'')^3}{n}\right)^{n'}% \geq \left(\left(1 - \frac{2\log^9 n}{n^{1/3}}\right)\left(1 - \frac{3\log^{3/2} n}{n}\right)n^2\right)^{n'}\\ 
      \geq \left(\left(1 - \frac{3}{\log^2 n}\right)n^2\right)^{n}n^{-4 n / \log^{2}n} \geq \left(\left(1 - o(1)\right)n^2\right)^{n}.
  \end{equation}
  Since the function $x \mapsto \log(1 - x)$ is negative and monotonically decreasing for $x \in [0, 1)$, 
  \begin{equation}\label{eqn:counting-riemann}
      \sum_{j=1}^{n'}\frac{1}{n''}\log\left(1 - \frac{j - 1}{n''}\right) \geq \int_{0}^{n' / n''}\log(1 - x)\,dx \geq \int_0^1 \log x\, dx =\COMMENT{Using that $\int \log x\, dx = x(\log x - 1)$ and $\lim_{x\rightarrow 0^{+}}x\log x = 0$ by L'H\^{o}pital's Rule} -1.
  \end{equation}
  By substituting~\eqref{eqn:counting-error-term} and~\eqref{eqn:counting-riemann} into the right side of~\eqref{eqn:counting-main-lb}, the expression in~\eqref{eqn:counting-main-lb} is at least
  \begin{equation}\label{eqn:counting-final-lb}
      \left(\left(1 - o(1)\right)n^2\right)^{n}e^{-3n''} \geq \left(\left(1 - o(1)\right)\frac{n^2}{e^3}\right)^{n}.
  \end{equation}
  
  By Stirling's approximation, $n! = (1 + O(1 / n))\sqrt{2\pi n}(n / e)^n \leq ((1 + o(1))n / e)^n$.  Hence, since~\eqref{eqn:counting-final-lb} provides a lower bound on the number of ordered sequences of edges $e_1, \dots, e_{n'}\in E(G)$ such that $\bigcup_{i=1}^{n'}e_i$ is a valid rainbow directed path forest, the total number of valid rainbow directed path forests in $G$ with at most $n^{9/10}$ components is at least
  \begin{equation*}
      \left.\left(\left(1 - o(1)\right)\frac{n^2}{e^3}\right)^{n}\middle/ n!\right. \geq \left(\left(1 - o(1)\right)\frac{n}{e^2}\right)^{n},
  \end{equation*}
  as desired.
\end{proof}
We now have all the tools we need to prove Theorem~\ref{main-thm}.
\lateproof{Theorem~\ref{main-thm}}
Let $G \in \Phi(\dirK)$ with proper $n$-arc-colouring $\phi$, let $V\coloneqq V(G)$, and let $C \coloneqq \phi(G)$.
By Lemma~\ref{fewloops}, Theorem~\ref{thm:discrepancy}, and Lemmas~\ref{masterswitch1} and~\ref{main-switching-lemma}, it suffices to show that if 
\begin{enumerate}[label=(\theequation)]
    \stepcounter{equation}
    \item\label{pf:loops} for every $c \in C$, at most $n / 2$ loops in $G$ are coloured $c$,
    \stepcounter{equation}
    \item\label{pf:quasirandom} $G$ is lower-quasirandom,
    \stepcounter{equation}
    \item\label{pf:abs} $G$ contains a well-spread collection of at least $n^2 / 2^{100}$ $(v, c)$-absorbing gadgets for every $v\in V$ and $c \in C$, and
    \stepcounter{equation}
    \item\label{pf:bridges} $G$ contains a well-spread collection of at least $n^2 / 10^{50}$ $(y, z)$-bridging gadgets for every $y, z \in V$, 
\end{enumerate}
then $G$ contains at least $\left(\left(1 - o(1)\right)n / e^2\right)^n$ rainbow directed Hamilton cycles.

Let $m = \lfloor n / \log^3 n\rfloor$, and let $T$ be a $256$-regular $2RMBG(7m, 2m)$ (which exists by Lemma~\ref{lemma:2rmbg}).  We build a $T$-absorber $H$ in $G$ using Lemma~\ref{lemma:habsorber}, but first we need the following claim to choose the roots of $H$ that will correspond to the flexible sets of $T$.
\begin{claim}\label{claim:flex-sets}
  There exist $V' \subseteq V$ and $C' \subseteq C$ such that $|V'| = |C'| = 2m$ and for every $u, v\in V$ such that $u\neq v$ and for every $c \in C$, there are at least $n^{99/50}$ directed paths $P$ in $G$ such that
  \begin{enumerate}[(i)]
    \item\label{covers:path} $P$ is rainbow and has length four, 
    \item\label{covers:head-tail} $P$ has head $v$ and tail $u$,
    \item\label{covers:abs-col} the second arc of $P$ is coloured $c$,
    \item\label{covers:internal-vtcs} $V(P)\setminus\{u, v\} \subseteq V'$, and
    \item\label{covers:flex-col} $\phi(P)\setminus \{c\} \subseteq C'$.
  \end{enumerate}
\end{claim}

\claimproof{}
  Let $p \coloneqq (2m + n^{9/10}) / n$, and let $U' \subseteq V$ and $D' \subseteq C$ be chosen randomly by including every $v \in V$ in $U'$ and every $c \in C$ in $D'$ independently with probability $p$.  We claim that the following holds with high probability:
  \begin{enumerate}[(a)]
      \item\label{claim:flex-set-size} $|U'|, |D'| = pn \pm n^{4/5}$, and
      \item\label{claim:covers-survive} for every $u, v\in V$ such that $u\neq v$ and for every $c\in C$, there are at least $p^6 n^2 / 10$ directed paths $P$ in $G$ satisfying~\ref{covers:path}-\ref{covers:abs-col}, such that $V(P)\setminus\{u, v\}\subseteq U'$ and $\phi(P)\setminus\{c\}\subseteq D'$. 
  \end{enumerate}
  Indeed,~\ref{claim:flex-set-size} follows from a standard application of the Chernoff Bound.\COMMENT{Applying the Chernoff Bound with $t \coloneqq n^{4/5}$, we have
  \begin{equation*}
      \Prob{||U'| - np| > n^{4/5}} \leq 2\exp\left(-\frac{n^{8/5}}{3np}\right) = \exp\left(-\Omega(n^{3/5}\log^3 n)\right),
  \end{equation*}
  and the same holds for $|D'|$.}
  To prove~\ref{claim:covers-survive}, we use McDiarmid's Inequality.  To that end, fix $u,v\in V$ distinct and $c\in C$. We let~$f$ denote the random variable counting the number of paths satisfying~\ref{claim:covers-survive}.  Note that $f$ is determined by the independent binomial random variables $\{X_w : w \in V\} \cup \{X_d : d \in C\}$, where $X_w$ indicates if $w \in U'$ and $X_d$ indicates if $d \in D'$.  By~\ref{pf:loops} and Proposition~\ref{prop:num-links}\ref{lb:covers}, there are at least $n^2 / 5$ paths satisfying~\ref{covers:path}, \ref{covers:head-tail}, and~\ref{covers:abs-col}, and each such path satisfies $V(P)\setminus\{u, v\}\subseteq U'$ with probability $p^3$ and $\phi(P)\setminus\{c\}\subseteq D'$ with probability $p^3$, independently.  Hence, $\Expect{f} \geq p^6 n^2 / 5$.  For each $w \in V$, by Proposition~\ref{prop:num-links}\ref{ub:cover-vtx}, $X_w$ affects $f$ by at most $3n$, and for each $d \in C$, by Proposition~\ref{prop:num-links}\ref{ub:cover-col}, $X_d$ affects $f$ by at most $3n$.  Therefore by McDiarmid's Inequality (Theorem~\ref{mcd}) applied with $t \coloneqq \Expect{f} / 2$, there are at least $p^6 n^2 / 10$ paths satisfying~\ref{claim:covers-survive} with probability at least $1 - \exp(-\Omega(p^{12} n))$.  Hence, by a union bound,~\ref{claim:covers-survive} holds for every $u,v\in V$ with $u \neq v$ and every $c\in C$ with high probability, as claimed.
  
  Now we fix a choice of $U'$ and $D'$ satisfying both~\ref{claim:flex-set-size} and~\ref{claim:covers-survive} simultaneously.  By~\ref{claim:flex-set-size}, $|U'|, |D'| \geq 2m$, so there exists $V' \subseteq U'$ and $C' \subseteq D'$ such that $|V'| = |C'| = 2m$, as required.  Moreover, by~\ref{claim:flex-set-size}, $|U' \setminus V'|, |D' \setminus C'| \leq 2n^{9/10}$.  Thus, by Proposition~\ref{prop:num-links}\ref{ub:cover-vtx} and \ref{prop:num-links}\ref{ub:cover-col},~\ref{claim:covers-survive} implies that {for every $u,v\in V$ distinct and $c\in C$,} there are at least $p^6 n^2 / 10 - 2(2n^{9/10})(3n) \geq n^{99/50}$ directed paths satisfying~\ref{covers:path}-\ref{covers:flex-col}, as desired.
\endclaimproof{}

Now let $U \subseteq V$ and $D \subseteq C$ such that $|U| = |D| = 7m$, $V' \subseteq U$, and $C' \subseteq D$.  By Lemma~\ref{lemma:habsorber},~\ref{pf:abs}, and~\ref{pf:bridges}, there is a $T$-absorber $H$ in $G$ rooted on~$U$ and~$D$ such that $V'$ and $C'$ are the sets of root vertices and colours of $H$ corresponding to the flexible sets of $T$. 
By Lemma~\ref{prop:rrH}, $H$ is robustly rainbow-Hamiltonian with respect to flexible sets $V'$ and $C'$ and initial and terminal vertices $t$ and $h$, where $t$ is the initial vertex of $H$ and $h$ is the terminal vertex of $H$.  Note that
\begin{equation}\label{eqn:size-of-absorber}
    |V(H)| = 22|E(T)| - 2 + |U| \text{ and } |\phi(H)| = 22|E(T)| - 3 + |D|.
\end{equation}

\begin{claim}\label{claim:absorb-path-forest}
If $Q$ is a spanning rainbow path forest in $G - V(H)$, sharing no colour with $H$, with at most $n^{9/10}$ components, then $G$ has a rainbow directed Hamilton cycle~$F$ such that $F - V(H) = Q$.
\end{claim}
\claimproof{}
Let $P_1, \dots, P_k$ be the components of $Q$, and for each $i \in [k]$, let $h_i$ be the head of $P_i$, and let $t_i$ be the tail of $P_i$.  Let $t_{k + 1}$ denote the {initial vertex} of $H$, and let $h_0$ denote the {terminal vertex} of $H$.  That is, $t_{k + 1} \coloneqq t$ and $h_0 \coloneqq h$.  Let $c_1, \dots, c_{k'}$ be an enumeration of the colours in $C \setminus (\phi(H)\cup \phi(Q))$.  Since $Q$ is rainbow, $|\phi(Q)| = |E(Q)| = |V\setminus V(H)| - k$, so by~\eqref{eqn:size-of-absorber}, $k' = \COMMENT{$n - |\phi(H)| - |\phi(Q)| = n - (|V(H)| - 1) - (n - |V(H)| - k) = k + 1 = $} k + 1$.  

By Claim~\ref{claim:flex-sets}, for each $j \in [k + 1]$, there is a collection $\cP_j$ of at least $n^{99/50}$ directed paths satisfying~\ref{covers:path}-\ref{covers:flex-col} where $u = h_{j-1}$, $v = t_{j}$, and $c = c_j$.  For each $j \in [k + 1]$, we inductively choose a path $P'_j \in \cP_j$ such that
\begin{enumerate}[(a)$_j$]
    \item\label{constructing-cov:vtx} $\phi(P'_j) \cap \phi(P'_\ell) = \emptyset$ for every $\ell < j$ and
    \item\label{constructing-cov:col} $V(P'_j) \cap V(P'_\ell) \cap V' = \emptyset$ for every $\ell < j$.
\end{enumerate}
To that end, we let $i \in [k + 1]$ and assume $P'_j\in\cP_j$ has been chosen to satisfy~\ref{constructing-cov:vtx} and~\ref{constructing-cov:col}
for each $j < i$, and we show that we can indeed choose such a $P'_j$ for $j = i$.  Let $U_j \coloneqq V' \cap \bigcup_{\ell = 1}^{j - 1}V(P'_\ell)$, and let $D_j \coloneqq D' \cap \bigcup_{\ell = 1}^{j - 1}\phi(P'_\ell)$.  For each $u \in U_j$, let $\cP_u \coloneqq \{P \in \cP_j : u \in V(P)\}$, and for each $d \in D_j$, let $\cP_d \coloneqq \{P \in \cP_j : d \in \phi(P)\}$.  Let $\cP'_j \coloneqq \cP_{j}\setminus \left(\bigcup_{u \in U_j}\cP_u\cup\bigcup_{d \in D_j}\cP_d\right)$.  By Proposition~\ref{prop:num-links}\ref{ub:cover-vtx}, $|\cP_u| \leq 3n$ for each $u \in U_{j}$, and by Proposition~\ref{prop:num-links}\ref{ub:cover-col}, $|\cP_d| \leq 3n$ for each $d \in D_{j}$.  Since $P'_\ell$ has length four for each $\ell < j$, we have $|U_j|, |D_j| \leq 3k$.  Hence, $|\cP'_j| \geq n^{99/50} - (6k)(3n) > 0$.  In particular, there exists $P'_j \in \cP'_j$.  By construction of $\cP'_j$, $P'_j$ satisfies~\ref{constructing-cov:vtx} and~\ref{constructing-cov:col}, as desired.

Since $Q$ shares no vertices or colours with $H$, for each $j \in [k + 1]$, since the internal vertices of $P'_j$ are in $V' \subseteq V(H)$ and since $\phi(P'_j) \setminus D' = \{c_j\}$, it follows that $P'_j \cup P_j$ is a rainbow directed path with tail $h_{j - 1}$ and head $h_j$. Moreover, by induction, using~\ref{constructing-cov:vtx} and~\ref{constructing-cov:col}, if $j \leq k$, then $\bigcup_{\ell = 1}^j (P'_\ell \cup P_\ell)$ is a rainbow directed path with tail $h_0$ and head $h_j$, and in particular, $P^*_1 \coloneqq P'_{k + 1} \cup \bigcup_{\ell=1}^{k}(P'_\ell \cup P_\ell)$ is a rainbow directed path with tail $h_0$ and head $t_{k + 1}$.  

Let $X \coloneqq V' \cap \bigcup_{j = 1}^{k + 1}V(P'_j)$, and let $Y \coloneqq D' \cap \bigcup_{j = 1}^{k + 1}\phi(P'_j)$.  Note that $|X|, |Y| \leq 3(k + 1) \leq m$.  Therefore, since $H$ is robustly rainbow-Hamiltonian, there exists a rainbow directed Hamilton path $P^*_2$ in $H - X$ with tail $t_{k + 1}$ and head $h_0$, not containing a colour in $Y$.  Now $F \coloneqq P^*_1 \cup P^*_2$ is a rainbow directed Hamilton cycle in $G$, and $F - V(H) = Q$, as desired.
\endclaimproof{}
By Lemma~\ref{lemma:pathforests}\COMMENT{applied with $U = V(H)$ and $D = \phi(H) \cup \{c\}$ for some arbitrary $c \in C \setminus \phi(H)$} and~\eqref{eqn:size-of-absorber}, there is a collection~$\cQ$ of at least $\left((1 - o(1))n / e^2\right)^n$ spanning rainbow directed path forests in~$G-V(H)$ that share no colours with $\phi(H)$ and have at most $n^{9/10}$ components.  By Claim~\ref{claim:absorb-path-forest}, for every $Q \in \cQ$, there is a rainbow directed Hamilton cycle $F_Q$ such that $F - V(H) = Q$.  Therefore $\{F_Q : Q \in \cQ\}$ is a collection of at least $\left((1 - o(1))n / e^2\right)^n$ distinct rainbow directed Hamilton cycles in $G$, as desired.
%By Lemmas~\ref{lemma:abstemp} and~\ref{lemma:habsorber} and Proposition~\ref{prop:rrH}, there is an absorbing template~$H$ for~$G$ whose flexible sets have size~$O(n/\log^{2}n)$, and a robustly rainbow-Hamiltonian subgraph $R\subseteq G$ with respect to the same flexible sets as~$H$, satisfying $|R|,|\phi(R)|=O(n/\log^{2}n)$.
%By Lemma~\ref{lemma:pathforests} there is a collection~$\cF$ of spanning rainbow directed path forests of~$G-R$ such that $|\cF|=\left((1-o(1))n/e^{2}\right)^{n}$ and each $\cP\in\cF$ has $O(n/\log^{3}n)$ components and satisfies $\phi(\cP)\cap\phi(R)=\emptyset$.
%By Lemma~\ref{lemma:hamcyc}, for each $\cP\in\cF$, there is a rainbow directed Hamilton cycle~$C_{\cP}$ of~$G$ whose edge set contains~$E(\cP)$.
%
%Define $\cC\coloneqq\{C_{\cP}\colon\cP\in\cF\}$ and use the fact that $E(\cP)\subseteq E(C_{\cP})$ to argue there's not too much overcounting in the absorption step.
\endproof
\section*{Acknowledgements}
We are grateful to both anonymous referees for their helpful comments.
In particular, we thank one of the referees for suggesting a simplification to the proof of Lemma~\ref{fewloops}.

\bibliographystyle{amsabbrv}
\bibliography{References}
%\vspace{20mm}
\APPENDIX{

\appendix
\renewcommand\sectionname{}
\renewcommand\appendixname{A}
\renewcommand{\thesection}{\Alph{section}}
\section{Proofs of Lemmas~\ref{quasi-lemma} and~\ref{masterswitch1}}
In this section we make clear the changes one needs to make to the arguments of~\cite{GKKO20} to obtain Lemmas~\ref{quasi-lemma} and~\ref{masterswitch1}.  Lemma~\ref{masterswitch1} is analogous to~\cite[Lemma 3.8]{GKKO20} (minus the notion of `edge-resilience'), and Lemma~\ref{quasi-lemma} is a direct analogue of~\cite[Lemma 6.3]{GKKO20}.  The proof of Lemma~\ref{quasi-lemma} can be obtained via a straightforward modification of the proof of~\cite[Lemma 6.3]{GKKO20}, and we describe this first.  Then, we show how to adapt the proof of Lemma~\ref{masterswitch1} from the proof of~\cite[Lemma 3.8]{GKKO20}, which can be be summarized as follows:
\begin{itemize}
    \item the main ingredients in the proof of~\cite[Lemma 3.8]{GKKO20} are~\cite[Lemmas 6.3, 6.8, and 6.9]{GKKO20}, and we will need an analogue of each of these;
    \item as mentioned, Lemma~\ref{quasi-lemma} is the analogue of~\cite[Lemma 6.3]{GKKO20}, and the proof adapts easily to this setting;
    \item we will need the obvious directed analogues of~\cite[Definitions 6.4--6.7]{GKKO20} used in~\cite[Lemmas 6.8 and 6.9]{GKKO20};
    \item we will need an analogue of~\cite[Lemma 6.8]{GKKO20} (namely Lemma~\ref{lemma:switch1}), which we break down into three claims, of which only the second does not quickly follow as an analogue of claims in the proof of~\cite[Lemma 6.8]{GKKO20};
    \item the proof of~\cite[Lemma 6.9]{GKKO20} adapts easily to this setting;
    \item the proof of~\cite[Lemma 3.8]{GKKO20} from~\cite[Lemmas 6.3, 6.8, and 6.9]{GKKO20} is essentially the same (see Lemma~\ref{lemma:gearup}), with an extra step to deal with the conditioning on the number of loops in each colour class.
\end{itemize}
%\textcolor{myblue}{UPDATE THIS PART TO PROPERLY DESCRIBE NEW STYLE OF BIG PROOF.}
To prove Lemma~\ref{quasi-lemma}, the main idea is that we can perform a switching operation very similar to that in the proof of~\cite[Lemma 6.3]{GKKO20}, whilst avoiding any arcs coloured~$c$, thus leaving the colour class of~$c$ unchanged.
\lateproof{Lemma~\ref{quasi-lemma}}
Modify the proof of~\cite[Lemma 6.3]{GKKO20} by fixing an outcome~$F$ of~$\mathbf{F}_{c}$ and analyzing the following `rotate' switching operation on the set of~$H\in\cG_{D}$ whose $c$-colour class is~$F$.
For fixed $A,B\subseteq[n]$ satisfying $|A|=|B|=|D|=n/10^{6}$, we instead say that a `rotation system' of~$H$ is a subdigraph $R\subseteq G$ together with a labelling of its vertices $V(R)=\{a,b,v,w\}$ such that $E(R)=\{ab,vw\}$ where $a\in A$, $b\in B$, $v\notin A$, $w\notin B$, $aw,vb\notin E(H)$, and $\phi_{H}(ab)=\phi_{H}(vw)\neq c$.
Then the `rotate' switching operation replaces the arcs~$ab$ and~$vw$ with the arcs~$aw$ and~$vb$, each in colour~$\phi_{H}(ab)\neq c$, thus leaving the $c$-colour class,~$F$, unchanged.
Analyzing auxiliary bipartite graphs~$B_{s}$ as in the proof of~\cite[Lemma 6.1]{GKKO20} and defining~$\delta_{s}$ and~$\Delta_{s-1}$ to be the analogous quantities, where an edge captures a rotation operation destroying an arc from~$A$ to~$B$ in some $H\in M_{s}$, it is simple to see that $\Delta_{s-1}\leq|D|^{3}$, and $\delta_{s}\geq (n-|A|-|B|-2|D|)(s-|A|)=\frac{999996}{10^{6}}n(s-|D|)$  (here the `$-|D|$' term occurs due to avoiding any arc of colour~$c$ to be the one that the rotation switching operation destroys between~$A$ and~$B$). Then for $s\geq \frac{3|D|^{3}}{2n}$ and~$n$ sufficiently large we obtain that $|M_{s}|/|M_{s-1}|\leq 9/10$. Proceeding as in the end of the proof of~\cite[Lemma 6.1]{GKKO20}, the result follows.
\endproof
The rest of the appendix is dedicated to the proof of Lemma~\ref{masterswitch1}, which we split into two lemmas (Lemmas~\ref{lemma:switch1} and~\ref{lemma:gearup}), which roughly speaking correspond to~\cite[Lemmas 6.8 and 6.9]{GKKO20} and the proof of~\cite[Lemma 3.8]{GKKO20}, respectively.
We begin by defining six events (in addition to~$\cE$ and~$\cC$ defined in the lemma statement and~$\cQ_{D}^{1}$ defined in Definition~\ref{def:quas}) that we will use, as well as giving some extra notation.
%For $G\in\Phi(\dirK)$ and $D\subseteq[n]$, we define~$G|_{D}$ to be the coloured digraph in~$\cG_{D}$ obtained by deleting the arcs of~$G$ having colours not in~$D$.
%We use the notation~$\pr$ for the measure of the probability space~$\cS$ corresponding to uniformly random choice of $\mathbf{G}\in\Phi(\dirK)$.
For fixed $c\in[n]$, we define~$\cE^{c}$ to be the event in the probability space~$\cS$ corresponding to uniformly random choice of $\mathbf{G}\in\Phi(\dirK)$, that~$\mathbf{G}$ contains a well-spread collection of~$n^{2}/2^{100}$ $(v,c)$-absorbing gadgets, for all $v\in[n]$.
We define~$\cC^{c}$ to be the event in~$\cS$ that there are at most~$n/10^{9}$ $c$-loops in~$\mathbf{G}$, and for fixed $D\subseteq[n]$ we define~$\left(\cE^{c}|_{D}\right)$ to be the event in~$\cS$ that~$\mathbf{G}|_{D}$ contains a well-spread collection of~$n^{2}/2^{100}$ $(v,c)$-absorbing gadgets, for all $v\in[n]$.

%For fixed $D\subseteq[n]$ we use the notation~$\pr_{D}$ for the measure of the probability space~$\cS_{D}$ corresponding to uniformly random choice of~$\mathbf{H}\in\cG_{D}$.
For fixed $D\subseteq[n]$ and~$c\in D$ we define~$\cE_{D}^{c}$ to be the event in the probability space~$\cS_{D}$ corresponding to uniformly random choice of $\mathbf{H}\in\cG_{D}$, that~$\mathbf{H}$ contains a well-spread collection of~$n^{2}/2^{100}$ $(v,c)$-absorbing gadgets, for all $v\in[n]$, and we define~$\cC_{D}^{c}$ to be the event in~$\cS_{D}$ that there are at most~$n/10^{9}$ $c$-loops in~$\mathbf{H}$.
Finally, we define the event~$\widetilde{\cQ}_{D}$ in~$\cS_{D}$ in an analogous way to the definition of~$\widetilde{\cQ}_{D^{*}}^{\text{col}}$ in~\cite{GKKO20} (see the text preceding Definition 6.7 in the cited paper).
Indeed, for an equitable partition~$(D_{i})_{i=1}^{4}$ of~$D$ we define `$(v,c,\cP)$-gadgets', `distinguishability' of $(v,c,\cP)$-gadgets, `saturation' of $c$-arcs, and the function $r_{(v,c,\cP)}\colon \cG_{D}\rightarrow [n|D|]_{0}$ in ways corresponding to~\cite[Definitions 6.4--6.6]{GKKO20}. (Just add the necessary directions as per Definition~\ref{def:absgadg} of the current paper, orienting the $d_{4}$-arc (with $d_{4}\in D_{4}$) of a $(v,c,\cP)$-gadget away from~$v$ (say).)
Then~$\widetilde{\cQ}_{D}$ is the set of $H\in\cG_{D}$ such that if $r_{(v,c,\cP)}(H)=s$, then $e_{H}(A,B)\leq 2|D|^{3}/n +6s$ for all $A,B\subseteq[n]$ such that $|A|=|B|=|D|$.~$\widetilde{\cQ}_{D}$ is just a reformulation of upper-quasirandomness which is closed under the switching operation we use to find $(v,c,\cP)$-gadgets.
The following lemma plays a role analogous to~\cite[Lemmas 6.8 and 6.9]{GKKO20}.
First, we discuss the switching operation that we will use in the proof.
Suppose $D\subseteq[n]$, $v,c\in[n]$, let $H\in\cG_{D}$, and fix a partition $\cP=(D_{i})_{i=1}^{4}$ of~$D$.
Let $u_{1},u_{2},\dots,u_{14}\in[n]\setminus\{v\}$, where $u_{1},u_{2},u_{5},u_{6},u_{7},u_{8},u_{9},u_{10},u_{13},u_{14}$ are distinct and $\{u_{1},u_{2},u_{5},u_{6},u_{7},u_{8},u_{9},u_{10},u_{13},u_{14}\}\cap\{u_{3},u_{4},u_{11},u_{12}\}=\emptyset$.
Then we say that a subgraph $T\subseteq H[\{v,u_{1},u_{2},\dots,u_{14}\}]$ is a \textit{twist system} of~$H$ if~$T$ satisfies~\cite[Definition 6.7(i)--(vii)]{GKKO20} (with directions added as discussed above), and we define the switching operation~`$\text{twist}_{T}$' and the `canonical $(v,c,\cP)$-gadget of the twist' analogously to~\cite[Definition 6.7]{GKKO20}.
We use the notation~$\pr_{D}$ for the measure of the probability space~$\cS_{D}$.
%We will use the switching operation~$\text{twist}_{T}$ as defined in~\cite[Definition 6.7]{GKKO20}, 
%The left column are events in the probability space~$\cS$ corresponding to uniformly random choice of $\mathbf{G}\in\Phi(\dirK)$, for which we use the measure~$\pr$. Here, $D\subseteq[n]$ is any fixed subset and $c\in[n]$ any fixed colour. For fixed $D\subseteq[n]$ and $c\in D$, the right column are events in the probability space~$\cS_{D}$ corresponding to uniformly random choice of $\mathbf{H}\in\cG_{D}$ for which we use the measure~$\pr_{D}$. For $G\in\Phi(\dirK)$ and $D\subseteq[n]$, we define~$G|_{D}$ to be the coloured digraph in~$\cG_{D}$ obtained by deleting the arcs of~$G$ having colours not in~$D$.
\begin{lemma}\label{lemma:switch1}
Suppose $D\subseteq[n]$ has size $|D|=n/10^{6}$, and fix $c\in D$.
Then
\[
\probd{\cE_{D}^{c}\,\big|\,\widetilde{\cQ}_{D}\cap\cC_{D}^{c}}\geq 1-\exp\left(-\Omega(n^{2})\right).
\]
\end{lemma}
\begin{proof}
It is simple to repurpose the arguments of~\cite[Lemma 6.9]{GKKO20} to show that if $r_{(v,c,\cP)}(\mathbf{H})\geq n^{2}/2^{100}$ for all~$v$ then~$\cE_{D}^{c}$ occurs, whence it suffices to show that
\begin{equation}\label{eq:switch}
\probd{\left(\exists v\in[n]\colon r_{(v,c,\cP)}(\mathbf{H})\leq \frac{n^{2}}{2^{100}}\right)\bigg|\, \widetilde{\cQ}_{D}\cap\cC_{D}^{c}}\leq\exp\left(-\Omega\left(n^{2}\right)\right),
\end{equation}
for an arbitrary fixed equitable partition $\cP=(D_{i})_{i=1}^{4}$.
Fix such a~$\cP$.
We show that~(\ref{eq:switch}) holds by analysing the twist switching operation.
%We use the same switching operation as in~\cite[Definition 6.7]{GKKO20}, with directions added to reflect the $(v,c,\cP)$-gadgets we seek to find in~$\cG_{D}$.
%However, using the same notation, we do not insist that $u_{3}\neq u_{4}$, $u_{11}\neq u_{12}$, $u_{3}\neq u_{11}$, or $u_{4}\neq u_{12}$, but rather only that all other pairs are distinct vertices, and that no host vertex has more than two of the above labels.
We will ensure that we only perform twists which increase~$r(H)$ (by precisely one). Note that the set~$\widetilde{\cQ}_{D}\cap\cC_{D}^{c}$ is closed under such twists, since we only add six arcs, and we do not add nor delete any arc coloured~$c$.
Let $k\coloneqq|D|=n/10^{6}$, and
%We list the key changes to the proof of~\cite[Lemma 6.8]{GKKO20} that one must make when using the above switching operation to prove that~(\ref{eq:switch}) holds.
fix~$v\in[n]$.
Analyzing auxiliary bipartite graphs~$B_{s}$ analogous to those in the proof of~\cite[Lemma 6.8]{GKKO20}, it is simple to see that $\Delta_{s+1}\leq (s+1)n^{4}$ for all $s\in[nk-1]_{0}$.
We now seek a lower bound for $\delta_{s}$ in the case that $s\leq k^{4}/2^{16}n^{2}$.
To that end, we fix such an~$s$, fix~$H$ in (the analogue of)~$T_{s}^{D}$, and bound from below the degree of~$H$ in~$B_{s}$ by finding many twist systems in~$H$ with desirable properties.
Note that~\cite[Lemma 6.8, Equation (6.2)]{GKKO20} holds as stated, with~$e_{G}(A,B)$ replaced by~$e_{H}(A,B)$.
The remainder of the proof now largely splits into three claims.
\begin{claim}\label{appclaim1}
There is a set $D_{3}^{\text{good}}\subseteq D_{3}$ of size $|D_{3}^{\text{good}}|\geq k/10$ such that for all $d\in D_{3}^{\text{good}}$ we have
\begin{enumerate}[label=\upshape(\roman*)]
\item $|E_{d}(N^{-}_{D_{1}}(x),N^{+}_{D_{2}}(x))|\leq200k^{2}/n$ (in~$H$);
\item there are at most~$64k^{3}/n^{2}$ $d$-arcs~$e$ in~$H$ with the property that~$e$ lies in some distinguishable $(v,c,\cP)$-gadget in~$H$ whose $c$-arc is not supersaturated;
\item there are at most~$n/100$ $d$-loops in~$H$.
\end{enumerate}
\end{claim}
\claimproof{}
We have that~(i) and~(ii) hold for at least~$k/8$ colours $d\in D_{3}$ analogously to~\cite[Lemma 6.8, Claim 1]{GKKO20}, so
it suffices to note that at most~$100$ colours fail condition~(iii).
\endclaimproof{}
\begin{claim}\label{appclaim2}
There is a set $\Lambda\subseteq\{(d_{1},d_{2},d_{3},d_{4},f_{1},f_{2})\colon d_{1}\in D_{1}, d_{2}\in D_{2}, d_{3}\in D_{3}^{\text{good}}, d_{4}\in D_{4}, f_{1},f_{2}\in E_{d_{3}}(H)\}$ of size $|\Lambda|\geq k^{4}n^{2}/10000$ such that for each $(d_{1},d_{2},d_{3},d_{4},f_{1},f_{2})\in\Lambda$ the following holds:
\begin{enumerate}[label=\upshape(\roman*)]
\item $v\notin\{u_{1},u_{2},\dots,u_{14}\}$;
\item $u_{1},u_{2},u_{5},u_{6},u_{7},u_{8},u_{9},u_{10},u_{13},u_{14}$ are distinct;
\item $\{u_{1},u_{2},u_{5},u_{6},u_{7},u_{8},u_{9},u_{10},u_{13},u_{14}\}\cap\{u_{3},u_{4},u_{11},u_{12}\}=\emptyset$,
\end{enumerate}
where  for each $(d_{1},d_{2},d_{3},d_{4},f_{1},f_{2})\in\Lambda$, we set $u_{7}\coloneqq \neigh{+}{d_{4}}{v}, u_{5}\coloneqq\neigh{-}{d_{1}}{u_{7}}, u_{3}\coloneqq\neigh{-}{d_{3}}{u_{5}}, u_{8}\coloneqq\neigh{-}{c}{u_{7}}, u_{6}\coloneqq\neigh{-}{d_{2}}{u_{8}}, u_{4}\coloneqq\neigh{+}{d_{3}}{u_{4}}, u_{2}\coloneqq\vtail{f_{1}}, u_{1}\coloneqq\vhead{f_{1}}, u_{9}\coloneqq\neigh{-}{d_{1}}{v}, u_{11}\coloneqq\neigh{+}{d_{3}}{u_{9}}, u_{10}\coloneqq\neigh{+}{d_{2}}{v}, u_{12}\coloneqq\neigh{-}{d_{3}}{u_{10}}, u_{13}\coloneqq\vtail{f_{2}}, u_{14}\coloneqq\vhead{f_{2}}$.
\end{claim}
\claimproof{}
We define~$R\coloneqq\left\{\neigh{-}{c}{\neigh{+}{d_{4}}{v}}\colon d_{4}\in D_{4}\right\}$. Since~$|R|\leq k$ we have that $e_{H}(N_{D_{2}}^{+}(v),R)\leq 10k^{3}/n$, whence for all $d_{4}\in D_{4}$ but a set~$D_{4}^{\text{bad}}$ of size at most~$k/10$, there are at most $100k^{2}/n$ colours $d_{2}\in D_{2}$ for which there is an arc from~$\neigh{+}{d_{2}}{v}$ to~$\neigh{-}{c}{\neigh{+}{d_{4}}{v}}$.
We now begin construction of the tuples $\lambda\in\Lambda$ by selecting $d_{4}\in D_{4}$ avoiding any colours for which we have a $d_{4}$-loop at~$v$, or $d_{4}\in D_{4}^{\text{bad}}$, or the $c$-arc with head~$N_{d_{4}}^{+}(v)$ is saturated or a loop. Due to our conditioning on~$\cC_{D}^{c}$, we have at least $k/4-1-k/1000-k/16-k/10\geq k/20$ acceptable choices for~$d_{4}$.
Next we choose $d_{1}\in D_{1}$ avoiding the colours of the arcs~$vv$,~$u_{7}u_{7}$,~$vu_{8}$ (if they are present in~$H$) so that we have at least~$k/5$ acceptable choices.
Now we choose $d_{2}\in D_{2}$ avoiding the colours of~$8$ arcs which if chosen would cause us to give the label~$u_{6}$ or~$u_{10}$ to a vertex already labelled, and also avoiding any of the at most~$100k^{2}/n=n/10^{10}$ choices of~$d_{2}$ for which there is an arc from~$u_{10}$ to~$u_{8}$. (This ensures that~$u_{6}$ and~$u_{10}$ are distinct vertices.) There are at least~$k/5$ acceptable choices for~$d_{2}$.
There are at most~$21$ colours $d_{3}\in D_{3}^{\text{good}}$ we must avoid for relabelling reasons. Our choice of~$d_{3}$ may cause $u_{3}=u_{4}$ and/or $u_{11}=u_{12}$, or instead may cause $u_{3}=u_{11}$ and/or $u_{4}=u_{12}$, but this seems difficult to avoid, and does not cause any problems. Thus there are at least~$k/12$ acceptable choices of~$d_{3}\in D_{3}^{\text{good}}$.
We now choose~$f_{1}$ to be a non-loop $d_{3}$-arc (recall condition~(iii) of Claim~\ref{appclaim1}) with neither endvertex labelled so far. %, and so that the head of~$f_{1}$ is not in~$N^{+}(u_{3})$ and the tail is not in~$N^{-}(u_{4})$.
Choosing~$f_{2}$ similarly, we conclude that $|\Lambda|\geq\frac{k}{20}\cdot\frac{k}{5}\cdot\frac{k}{5}\cdot\frac{k}{12}\cdot\frac{19n}{20}\cdot\frac{19n}{20}\geq\frac{k^{4}n^{2}}{10000}$.
\endclaimproof{}
For each tuple $\lambda=(d_{1},d_{2},d_{3},d_{4},f_{1},f_{2})\in\Lambda$, define~$T_{\lambda}$ to be the subgraph of~$H$ with vertex set $V(T_{\lambda})=\{v,u_{1},u_{2},\dots,u_{14}\}$ and arc set $E(T_{\lambda})=\{u_{2}u_{1},u_{3}u_{5},u_{6}u_{4},u_{5}u_{7},u_{6}u_{8}\}\cup\{u_{8}u_{7},vu_{7},u_{9}v,vu_{10},u_{9}u_{11},u_{12}u_{10},u_{13}u_{14}\}$ (with the vertices~$u_{i}$ defined as in Claim~\ref{appclaim2}).
\begin{claim}\label{appclaim3}
%For each tuple $\lambda=(d_{1},d_{2},d_{3},d_{4},f_{1},f_{2})\in\Lambda$, define~$T_{\lambda}$ to be the subgraph of~$H$ with vertex set $V(T_{\lambda})=\{v,u_{1},u_{2},\dots,u_{14}\}$ and arc set $E(T_{\lambda})=\{u_{2}u_{1},u_{3}u_{5},u_{6}u_{4},u_{5}u_{7},u_{6}u_{8}\}\cup\{u_{8}u_{7},vu_{7},u_{9}v,vu_{10},u_{9}u_{11},u_{12}u_{10},u_{13}u_{14}\}$ (with the vertices~$u_{i}$ defined as in Claim~\ref{appclaim2}).
There is a set~$\Lambda^{*}\subseteq\Lambda$ of size at least $|\Lambda|/2$ such that each $\lambda\in\Lambda^{*}$ satisfies the following properties:
\begin{itemize}
\item [$(Q1)$] $T_{\lambda}$ is a twist system of~$H$;
\item [$(Q2)$] deleting the six $d_{3}$-arcs in~$T_{\lambda}$ does not decrease~$r(H)$;
\item [$(Q3)$] the canonical $(v,c,\cP)$-gadget of the twist~$\text{twist}_{T_{\lambda}}(H)$ is distinguishable, and it is the only $(v,c,\cP)$-gadget that is in~$\text{twist}_{T_{\lambda}}(H)$ but not in~$H$.
\end{itemize}
\end{claim}
\claimproof{}
For all $\lambda\in\Lambda$, by the definition of $u_{1},u_{2},\dots,u_{14}$ and the result of Claim~\ref{appclaim2}, we have that~$T_{\lambda}$ satisfies the analogues of~\cite[Definition 6.7(i)--(vi)]{GKKO20}.
Thus, to check that there is a set $\Lambda_{1}\subseteq\Lambda$ of size $|\Lambda_{1}|\geq 9|\Lambda|/10$ such that each $\lambda\in\Lambda_{1}$ satisfies~$(Q1)$, it suffices to check that at most~$|\Lambda|/10$ tuples $\lambda\in\Lambda$ fail to satisfy the analogue of~\cite[Definition 6.7(vii)]{GKKO20}.
To do this, we use the analogue of~\cite[Lemma 6.8, Equation (6.2)]{GKKO20} in the same way as in the proof of~\cite[Lemma 6.8, Claim 2]{GKKO20}, so we omit the details here.

An analogue of~\cite[Lemma 6.8, Claim 3]{GKKO20} (with the same proof) now proves that at most $|\Lambda|/10$ tuples $\lambda\in\Lambda_{1}$ fail to satisfy~$(Q2)$, and analogues of~\cite[Lemma 6.8, Claims~4 and~5]{GKKO20} (with the same proof) address~$(Q3)$. We remark that we do not require an analogue of~\cite[Lemma 6.8, Claim 6]{GKKO20} in this setting.
\endclaimproof{}
Observe that each twist adds only six arcs, and that each $\lambda\in\Lambda^{*}$ produces a unique twist system $T_{\lambda}\subseteq H$.
Thus, from Claims~\ref{appclaim2} and~\ref{appclaim3} we deduce that if $s\leq k^{4}/2^{16}n^{2}$ then $\delta_{s}\geq |\Lambda^{*}|\geq k^{4}n^{2}/20000$, whence either $T_{s}=\emptyset$ or $|T_{s}|/|T_{s+1}|\leq \Delta_{s+1}/\delta_{s}\leq 2/3$, where~$T_{s}$ is the set of $H\in\widetilde{\cQ}_{D}\cap\cC_{D}^{c}$ for which $r(H)=s$.
We now obtain $\probd{r_{(v,c,\cP)}(\mathbf{H})\leq n^{2}/2^{100}\, \big|\, \widetilde{\cQ}_{D}\cap\cC_{D}^{c}}\leq\exp(-\Omega(n^{2}))$ by repurposing the calculations at the end of the proof of~\cite[Lemma 6.8]{GKKO20}.
By a union bound over $v\in[n]$,~(\ref{eq:switch}) holds, which completes the proof of the lemma.
\end{proof}
The following lemma does most of the analogous work to the proof of~\cite[Lemma 3.8]{GKKO20}.
\begin{lemma}\label{lemma:gearup}
Fix $c\in[n]$.
Then we have $\prob{\cE^{c}\mid\cC^{c}}\geq 1-\exp(-\Omega(n^{2}))$.
\end{lemma}
\begin{proof}
Arbitrarily select $D\subseteq[n]$ such that $c\in D$ and $|D|=n/10^{6}$.
Define~$\cF$ to be the collection of all sets~$F$ of~$n$ arcs of~$\dirK$ such that every vertex of~$\dirK$ is the head of precisely one arc in~$F$ and the tail of precisely one arc in~$F$, and define~$\cF^{\text{good}}\subseteq\cF$ to be the set of $F\in\cF$ such that~$F$ contains at most~$n/10^{9}$ loops.
Note that
\begin{equation}\label{eq:q}
\probd{\overline{\widetilde{\cQ}_{D}}\mid\cC_{D}^{c}} \leq \probd{\overline{\cQ_{D}^{1}}\mid\cC_{D}^{c}}=\sum_{F\in\cF^{\text{good}}}\probd{\mathbf{F}_{c}=F\mid\cC_{D}^{c}}\probd{\overline{\cQ_{D}^{1}}\mid\mathbf{F}_{c}=F} \leq \exp(-\Omega(n^{2})),
\end{equation}
where we have used Lemma~\ref{quasi-lemma} and the fact that $\cQ_{D}^{1}\subseteq\widetilde{\cQ}_{D}$.
Now using Lemma~\ref{lemma:switch1},~(\ref{eq:q}), and the law of total probability, we obtain that
\begin{equation}\label{eq:probofgadg}
\probd{\overline{\cE_{D}^{c}}\mid\cC_{D}^{c}}\leq\probd{\overline{\cE_{D}^{c}}\mid\widetilde{\cQ}_{D}\cap\cC_{D}^{c}}+\probd{\overline{\widetilde{\cQ}_{D}}\mid \cC_{D}^{c}}\leq\exp(-\Omega(n^{2})).
\end{equation}
%Also write~$\cE^{c}$ for the event (in~$\Phi(\dirK)$) that~$\mathbf{G}$ contains a well-spread collection of~$\Omega(n^{2})$ $(v,c)$-absorbing gadgets, for all $v\in[n]$.
%\textcolor{blue}{(Maybe define all the different events at the same time somewhere in a logically intuitive way.)}
%For a coloured digraph $H\in\cG_{D}$, we define~$\text{comp}(H)$ to be the number of distinct ways to complete~$H$ to an element $G\in\Phi(\dirK)$, or more precisely the number of $H'\in\cG_{[n]\setminus D}$ having $E(H)\cap E(H')=\emptyset$ (and therefore $E(H)\cup E(H')=E(\dirK)$).
By Proposition~\ref{prop:wf} and~(\ref{eq:probofgadg}), we obtain
\begin{eqnarray*}
\prob{\overline{\cE^{c}}\mid\cC^{c}} & \leq & \prob{\overline{\left(\cE^{c}|_{D}\right)}\mid\cC^{c}}=\frac{\sum_{H\in\overline{\cE_{D}^{c}}\cap\cC_{D}^{c}}\text{comp}(H)}{\sum_{H'\in\cC_{D}^{c}}\text{comp}(H')}\leq\probd{\overline{\cE_{D}^{c}}\mid\cC_{D}^{c}}\cdot\exp\left(O\left(n\log^{2}n\right)\right)\\ & \leq & \exp\left(-\Omega\left(n^{2}\right)\right),
\end{eqnarray*}
completing the proof of the lemma.
\end{proof}
Lemma~\ref{masterswitch1} now follows very quickly from Lemmas~\ref{lemma:switch1} and~\ref{lemma:gearup}.
\lateproof{Lemma~\ref{masterswitch1}}
For fixed $c\in[n]$, we can use Lemma~\ref{lemma:gearup} to see that
\[
\prob{\overline{\cE^{c}}\mid\cC}=\frac{|\overline{\cE^{c}}\cap\cC|}{|\cC|}\leq\frac{|\overline{\cE^{c}}\cap\cC^{c}|}{|\cC^{c}|}\cdot\frac{|\cC^{c}|}{|\cC|}\leq2\prob{\overline{\cE^{c}}\mid\cC^{c}}\leq\exp\left(-\Omega\left(n^{2}\right)\right),
\]
where we have used $|\cC^{c}|\leq|\Phi(\dirK)|$ and $|\cC|\geq|\Phi(\dirK)|/2$, the latter following from Lemma~\ref{fewloops}.
A union bound over $c\in[n]$ now gives that $\prob{\overline{\cE}\mid\cC}\leq\exp\left(-\Omega\left(n^{2}\right)\right)$, as desired.
The final claim in the statement now follows by applying Lemma~\ref{fewloops} with $t=n/10^{9}$, and using the law of total probability.\COMMENT{$\prob{\overline{\cE}}=\prob{\cC}\prob{\overline{\cE}\mid\cC}+\prob{\overline{\cC}}\prob{\overline{\cE}\mid\overline{\cC}} \leq\prob{\overline{\cE}\mid\cC}+\prob{\overline{\cC}}\leq\exp\left(-\Omega\left(n^{2}\right)\right)+\exp\left(-\Omega\left(n\log n\right)\right)=\exp\left(-\Omega\left(n\log n\right)\right).$}
\endproof
}
\vspace{20mm}
\end{document}